\documentclass[10pt]{article}
\usepackage[letterpaper,twoside,outer=1.3in,vmargin=1.3in,]{geometry}
\usepackage[ruled,vlined]{algorithm2e}

\usepackage{fullpage}
\usepackage{enumerate}
\usepackage{color}
\usepackage{paralist}
\usepackage{float}
\usepackage{tikz}
\usetikzlibrary{decorations.pathreplacing}

\usepackage{amsmath,amsfonts,amsthm,amssymb,bbm}
\def \bR {\mathbb{R}}

\def \R {\mathbb{R}}
\def \Z {\mathbb{Z}}
\def \conv {\text{conv}}

\def \cS {\mathcal{S}}
\def \XJD {X^{0,u}}%^{J,D}}
\def \XJDd {X^{l,u}}
\def \XJd {X^{l,n}}
\def \DJD {\Delta^{0,u}}%^{J,D}}
\def \DJDd {\Delta^{l,u}}
\def \DJd {\Delta^{l,n}}
\def \goodfamily {proper} %{$D$\text{-complete}}
\def \Goodfamily {Proper} %{$D$\text{-complete}}
\def\mswcc{ multilinear set with cardinality constraints}
\def \mydelta{\boldsymbol{\delta}}

\DeclareMathOperator{\proj}{proj}

\newtheorem{thm}{Theorem}
\newtheorem{lem}[thm]{Lemma}
\newtheorem{prop}[thm]{Proposition}

\newtheorem{rmk}[thm]{Remark}

\theoremstyle{definition}
\newtheorem{defn}{Definition}

\def\alphapattern{pattern of $\alpha$} %\def\alphapattern{$\alpha$-pattern}

\newtheorem{example}{Example}

\title{Convexifying Multilinear Sets with Cardinality Constraints: Structural Properties, Nested Case and Extensions}
\author{Rui Chen \\ University of Wisconsin-Madison \\[-.2cm]\small (rchen234@wisc.edu) \and 
	Sanjeeb Dash\\ IBM Research \\[-.2cm]\small (sanjeebd@us.ibm.com) \and
	Oktay G\"{u}nl\"{u}k \\ Cornell University \\[-.2cm] \small (ong5@cornell.edu)}
\date{\today}

\begin{document}
\maketitle\allowdisplaybreaks

\begin{abstract}
	The problem of minimizing a multilinear function of binary variables is a well-studied NP-hard problem. The set of solutions of the standard linearization of this problem is called the multilinear set. We study a cardinality constrained version of it with upper and lower bounds on the number of nonzero  variables. We call the set of solutions of the standard linearization of this problem a  \mswcc. We characterize a set of conditions on these  multilinear terms (called {\em properness}) and observe that under these conditions the convex hull description of the set is tractable via an extended formulation. We then give an explicit  polyhedral description of the convex hull when the multilinear terms have a nested structure. Our description has an exponential number of inequalities which can be separated in polynomial time. Finally, we generalize these inequalities to obtain valid inequalities for the general case.
\end{abstract}
\section{Introduction}
In this paper, we study the convex hull of the set
\[ X =\Big\{(x,\delta)\in\{0,1\}^{n}\times\{0,1\}^{m}:\delta_i=\prod_{j\in S_i}x_j,~i=1,\ldots,m,~L\leq\sum_{j=1}^nx_j\leq U\Big\},
\]
where $m,n$ are positive integers, $S_i \subseteq J = \{1, \ldots, n\}$ for $i=1,\ldots,m$ and $L,U$ are integers such that $0 \leq L \leq U \leq n$.
We call $X$ the {\em\mswcc}.
We investigate the structural properties of $\conv(X)$, give a polyhedral characterization in the special case that the sets $S_i$ are nested, i.e., $S_1 \subset S_2 \subset \cdots \subset S_m$, and give a family of valid inequalities for the non-nested case.

The problem of minimizing a polynomial objective function of binary variables subject to polynomial constraints is called the binary polynomial optimization problem, and is  often solved by formulating it as an integer linear programming problem.
The first step in creating such a formulation is to replace each polynomial function by an equivalent (for all $x \in \{0,1\}^n$) multilinear expression of the form
\begin{displaymath}
f(x)=\beta+\sum_{i=1}^m \gamma_i\prod_{j\in S_i} x_j
\end{displaymath}
where $S_i \subseteq \{1, \ldots, n\}$ for $i=1,\ldots, m$, $\beta\in\R$, and $\gamma\in\R^m$.
Minimizing $f(x)$ over the set of binary $x$-vectors is equivalent to minimizing the linear expression $\beta + \sum_{i=1}^m \gamma_i \delta_i$ over
the set $Y$, obtained from $X$ by dropping the cardinality constraints:
\[ Y =\Big\{(x,\delta)\in\{0,1\}^{n}\times\{0,1\}^{m}:\delta_i=\prod_{j\in S_i}x_j,~i=1,\ldots,m\Big\}.
\]
Set $Y$ is called the {\em multilinear set} and is well-studied in mixed-integer nonlinear optimization because of the connection to binary polynomial optimization.

The convex hull of the multilinear set is called the {\em multilinear polytope} \cite{del2017polyhedral}, and several classes of valid inequalities for the multilinear polytope have been proposed recently \cite{crama2017class,del2017polyhedral,del2018decomposability,del2018multilinear,del2021running}.
The {\em boolean quadric polytope} \cite{padberg1989boolean} is equal to $\conv(Y)$ when $|S_i|=2$ for $i=1, \ldots, m$.
A complete polyhedral characterization of $\conv(Y)$ has been given in some cases, for example, when the multilinear set is defined by a single nonlinear monomial ($m = 1$), see \cite{crama1993concave}, or  when the multilinear sets are associated with certain acyclic hypergraphs \cite{del2018multilinear,del2021running, del2020complexity}.
When the nonlinear monomials have a nested structure, i.e., the sets $S_i$ have a nested structure, then the convex hull equals the {\em 2-link polytope}, which is obtained by augmenting the standard linearization constraints with the 2-link inequalities \cite{crama2017class}. This result follows from the work of Fischer, Fischer and McCormick \cite{fischer2018matroid}.
The problem of minimizing a linear function over $Y$ contains as a special case the {\em maximum monomial agreement} problem which has been analyzed in the context of machine learning \cite{dgm, dbs, eg2, eg3} and solved via branch-and-bound methods and heuristics.

Mehrotra \cite{m1997} studied the boolean quadric polytope with an upper bound constraint on the number of nonzero $x_j$ variables, i.e., the set $X$ with $|S_i| = 2$ and $L = 0$.
When the nonlinear monomials have a nested structure, 
Fischer, Fischer and McCormick \cite{fischer2018matroid} gave a polyhedral description of the convex hull of $V = \{(x, \delta) \in Y : x \in \mathcal{M}\}$, where $\mathcal{M}$ is the independent set polytope of a matroid over $n$ elements, generalizing earlier results by Buchheim and Klein \cite{bk} and Fischer and Fischer \cite{ff}.
When the matroid is a uniform matroid, $V$ is the same as $X$ with $L = 0$.
In this paper, we study the case when the nonlinear monomials have a nested structure and $L > 0$. Our results do not follow from the work of Fischer, Fischer and McCormick mentioned above.
Minimizing a linear function over $X$ generalizes the cardinality constrained maximum monomial agreement problem, which was studied in the context of binary classification in machine learning by Dash, G\"unl\"uk and Wei \cite{dash2018boolean}.

When $m \geq 2$ and the sets $S_i$ are not nested, $X$ is quite a complicated object. We have given a  complete characterization of the convex hull of $X$ when $m = 2$  in \cite{chen2021multilinear}. This generalizes the work of Crama and Rodr\'iguez-Heck \cite{crama2017class} who showed that the 2-link polytope is equal to $\conv(Y)$ when $m = 2$.
In this paper, we give a general family of valid inequalities for the case $m \geq 2$.

The paper is organized as follows. In Section 2, we study general multilinear sets with cardinality constraints and give some facetial conditions of the convex hull under a set of properness assumptions. In Section 3, we propose new valid inequalities and give a complete polyhedral description of the convex hull for cases when $\{S_i\}_{i=1}^m$ are nested. In Section 4, we give necessary conditions and sufficient conditions for a set of valid inequalities to define facets of the convex hull for the nested case. In Section 5, we generalize these valid inequalities to the non-nested case.

\section{Preliminaries}
%The additional variables in our characterization correspond to the values of the monomial $p(x)$ evaluated at different rows of $A$.

Let $I=\{1,\ldots,m\},~ J=\{1,\ldots,n\}$, $0 \leq l\leq u$ and $u\geq 2$.
Let $S_1, \ldots, S_m$ be distinct subsets of $J$ with $1 \leq |S_i| \leq n-l$ for $i=1, \ldots, m$. Note that the assumptions imply that $n-l\geq 1$. Define $\cS:=\{S_i\}_{i\in I}$. We will study the set
\begin{displaymath}
\XJDd:=\Big\{(z,\delta)\in\{0,1\}^{n}\times\{0,1\}^{m}:\delta_i=\prod_{j\in S_i}(1-z_j),~i\in I,~l\leq\sum_{j\in J}z_j\leq u \Big\},
\end{displaymath}
which is equivalent to the set $X$ in the previous section (let $z_j = 1-x_j$, $l=n-U$ and $u=n-L$).%and complement $l$ and $u$).

The standard linearization of the set $\XJDd$ is given by the following system of inequalities: %(Dash, G\"unl\"uk and Wei \cite{dash2018boolean} model (\ref{mma-eq2}) using these inequalities and 0-1 restrictions on the variables)
\begin{align}
l\leq\sum_{j\in J}z_j\leq&~ u,\label{xmc0}\\
z_j+\delta_i\leq& ~1, && j\in S_i,i\in I,\label{xmc1}\\
%	1-\delta_i-\sum_{j\in S_i}z_j\leq& ~0, &&i\in I,\label{xmc2}\\
\delta_i+\sum_{j\in S_i}z_j\ge& ~1, &&i\in I,\label{xmc2}\\
\delta_i\geq& ~0, &&i\in I,\label{xmc6}\\
%\text{** why do we have the upper bound? implied by \eqref{xmc1} ** }
1~\ge z_j\geq& ~0, &&j\in J\label{xmc7}.
\end{align}

We say that $\cS$ is closed under nonempty intersection if for each pair $S_i, S_j \in \cS$ such that $S_i \cap S_j \neq \emptyset$, their intersection $S_i \cap S_j$ is also contained in $\cS$.
Let $\DJDd = \proj_{\delta}(\XJDd)$ denote the orthogonal projection of $\XJDd$ onto the space of $\delta$ variables.
The next result gives a simple characterization of the convex hull of the set
\[ \XJDd(\bar\delta) = \{z \in \{0,1\}^n: (z,\bar\delta)\in\XJDd\} \]
for each $\bar\delta\in\DJDd$ under the assumption that $\cS$ is closed under nonempty intersection.
%is defined by a polynomial number of facets.
\begin{lem}\label{lem:disj}
	If $\cS$ is closed under nonempty intersection, then for each $\bar \delta\in \DJDd$, there exists a subset $I^*$ of $I$ and disjoint subsets $J_0$ and $\{J_i : i \in I^*\}$ of $J$ such that $\conv(\XJDd(\bar\delta))$ is defined by the inequalities
	\begin{align}
	&&z_j=&~0,&j&\in J_0,\label{c11}&&\\
	&&\sum_{j\in J_i}z_j\geq& ~1,&i&\in I^*,&&\label{c122}\\
	&& l \leq \sum_{j \in J} z_j \leq& u &&\label{c13}\\
	&& 0 \leq z_j \leq& 1, &j& \in J && \label{c14}
	\end{align}        
\end{lem}
\begin{proof}
	Let $\bar{\delta}\in\DJDd$. Then a binary vector $z \in \XJDd(\bar\delta)$ if and only if (\ref{c13}) is satisfied and 
	\begin{align}
	\prod_{j\in S_i}(1-z_j)=&\bar{\delta}_i,\quad i\in I.\label{c01}
	\end{align}
	Let $I_0:=\{i\in I:\bar{\delta}_i=0 \}$ and let $J_0:=\bigcup_{i\in I: \bar\delta_i=1}S_i$.
	Note that \eqref{c01}  is equivalent to (\ref{c11}) and the inequalities
	\begin{align}
	&&\sum_{j\in S_i}z_j\geq& ~1,&i&\in I_0,&&\label{c12}
	\end{align}	  
	Let $ I^*\subseteq I_0$ denote the index set of minimal elements (with respect to inclusion) of $\{S_i: i\in I_0\}$. 
	Then replacing $I_0$ by $I^*$ in (\ref{c12}) yields an equivalent set of constraints.
	
	For each $i\in  I^*$, let
	\begin{displaymath}
	J_i:=S_i\setminus J_0.
	\end{displaymath}
	Then $\emptyset \neq J_i \subseteq S_i$. The nonemptyness of $J_i$ for $i \in I^* \subseteq I_0$ follows from the fact that $\bar \delta_i = 0$ and for some $j \in S_i$ we must have $z_j = 1$. But for all $j \in J_0$, we must have $z_j = 0$.
	Therefore, any nonzero binary vector $z$ that satisfies equations \eqref{c11} will  also satisfy \eqref{c12} if and only if (\ref{c122}) is satisfied.
	
	We next argue that $J_0$ and $\{J_i\}_{i\in  I^*}$ are disjoint sets. 
	By definition, $J_0 \cap J_i = \emptyset$, for all $i\in  I^*$.
	For any $i_1 < i_2\in  I^*$, if $J_{i_1}\cap J_{i_2}\neq\emptyset$, then  we must have $\emptyset \neq S_{i_1}\cap S_{i_2}\in \cS$, as $\cS$ is closed under nonempty intersection. Therefore $S_{i_1}\cap S_{i_2} = S_{i_3}$ for some  $i_3\in I$, and $S_{i_3} \supseteq J_{i_1}\cap J_{i_2}$. Then $S_{i_3} \setminus J_0 \neq \emptyset$, and therefore $\bar\delta_{i_3} = 0$.
	Consequently,  $i_3\in I_0$, which contradicts the fact that $S_{i_1}$ (or $S_{i_2}$) is a minimal element  in $\{S_i: i\in I_0\}$. 
	
	We have shown that if $z\in \{0,1\}$, then $z \in \XJDd(\bar \delta)$ if and only if $z$ satisfies the constraints in (\ref{c11})-(\ref{c14}).
	Note that the constraint matrix associated with this system of inequalities is totally unimodular. 
	This is because each $z_j$ occurs once in  (\ref{c13}) and possibly once more in \eqref{c11} or (\ref{c122}) and therefore the rows of the associated constraint matrix admits an equitable row bi-coloring \cite{conforti2014integer}.	
%	at most twice in the inequalities (\ref{c122}) and $\sum_{j\in J} z_j \leq u$. Adding the remaining inequalities corresponds to %duplicating a row of the constraint matrix or adding unit vectors as rows.
	Therefore the polyhedron defined by (\ref{c11})-(\ref{c14}) is an integral polyhedron, and has only 0-1 vertices. The result follows.
\end{proof}

The previous result implies that if $\cS$ is closed under nonempty intersection and $\DJDd$ has polynomially many elements, 
%a {\goodfamily} family, 
optimizing a linear function over $\XJDd$ can be formulated as a linear program of polynomial size using Balas' disjunctive model \cite{balas1979disjunctive}. In particular, optimizing a linear function over $\XJDd$ is equivalent to optimizing linear functions over $\conv(\XJDd(\bar{\delta}))$ for all $\bar{\delta}\in\DJDd$. However, we are interested in characterizing $\conv(\XJDd)$ in the original space in order to deal with problems where $\XJDd$ appears as a substructure.

\subsection{{\Goodfamily} families}
We next present a definition where we call $\cS$ that defines $\XJDd$ a  \textit{{\goodfamily} family} if it satisfies some simple conditions. We will then show that inequalities that define $\conv(\XJDd)$ satisfy certain properties if $\cS$ is a {\goodfamily} family.

\begin{defn}\label{DCdef}
	A family $\cS=\{S_i\}_{i\in I}$ of subsets of $J$ is called a \textit{{\goodfamily} family} if it satisfies the following properties:\begin{enumerate}
		\item $\DJDd$ is a set of exactly $m+1$ affinely independent vectors in $\bR^{m}$;
		\item $\cS$ is closed under nonempty intersection.
		%			\item For each $\delta\in \DJDd,$ there exists disjoint subsets $J_0$ and $\{J_i\}_{i\in I^*}$ of $J$ such that a binary vector $(z,\delta)\in \XJDd$ if and only if $z$ satisfies $\sum_{j\in J}z_j\leq D$, $z_j=0$ for all $j\in J_0$ and $\sum_{j\in S^+_k(\delta)}z_j\geq 1$ for all $k\in I^*$.
	\end{enumerate} 
\end{defn}
Note that if $\cS$ is a proper family, then it is closed under nonempty intersection and the size of  $\DJDd$ is polynomial in $m$ and consequently a polynomial-sized extended formulation of  $\conv(\XJDd)$ can be obtained using Balas' disjunctive model.
In particular, we will show that if $\cS$ is proper, then we can characterize $\conv(\XJDd)$ by enumerating a set of valid inequalities.
We next present three examples of {\goodfamily} families $\cS$ together with the corresponding sets $I^*$ and $J_i$ for $i\in  I^*\cup\{0\}$, for each $\delta\in\DJDd$.
\begin{example}\label{nestedexample}
	If $S_1,S_2,\ldots,S_m$ are nested subsets of $J$, $l\leq n-|S_m|$ and $u\geq 2$, then $\cS=\{S_1,S_2,\ldots,S_m\}$ is {\goodfamily}. 
	In this case,  $S_1\subset S_2\subset\ldots\subset S_m$ and  $\DJDd=\{\delta\in\{0,1\}^m:\delta_1\geq\delta_2\geq\ldots\geq\delta_m\}$ is a set of $m+1$ affinely independent points in $\bR^{m}$. For each $\delta\in \DJDd$, $J_0$ and $\{J_i\}_{i\in I^*}$ are defined as follows.\begin{displaymath}
	\begin{array}{c|c|c}
	\delta &~~J_0~~ & \{J_i\}_{i\in I^*}\\
	\hline
	(0,0,\ldots,0) & \emptyset & \{S_1\}\\
	(\underbrace{1,\ldots,1}_{\text{first $p$ entries}},0,\ldots,0) \text{ for some $1\leq p\leq m-1$} & S_p &\{S_{p+1}\setminus S_p\}\\
	(1,1,\ldots,1) & S_m & \emptyset
	\end{array}
	\end{displaymath}
\end{example}

\begin{example}
	If $S_1,S_2$ are two disjoint subsets of $J$, $l\leq n-|S_1\cup S_2|$ and $u\geq 2$, then $\cS=\{S_1, S_2, S_1\cup S_2 \}$ is {\goodfamily}. In this case, $\DJDd=\{(0,0,0),(1,0,0),(0,1,0),(1,1,1)\}$ is a set of 4 affinely independent points in $\bR^3$. For each $\delta\in \DJDd$, $J_0$ and $\{J_i\}_{i\in I^*}$ are defined as follows.\begin{displaymath}
	\begin{array}{c|c|c}
	\delta & J_0 & \{J_i\}_{i\in I^*}\\
	\hline
	(0,0,0) & \emptyset & \{S_1,S_2\}\\
	(1,0,0) & S_1 & \{S_2\}\\
	(0,1,0) & S_2 & \{S_1\}\\
	(1,1,1) & S_1\cup S_2 & \emptyset
	\end{array}
	\end{displaymath}
\end{example}

\begin{example}
	
	If $S_1$ and $S_2$ are two subsets of $J$ satisfying $S_1\cap S_2\neq\emptyset,S_1\nsubseteq S_2,S_2\nsubseteq S_1$, $l\leq n-|S_1\cup S_2|$ and $u\geq 2$, then $\cS=\{S_1\cap S_2,S_1,S_2, S_1\cup S_2\}$ is {\goodfamily}. In this case, $\DJDd=\{(0,0,0,0),(1,0,0,0),(1,1,0,0),$ $(1,0,1,0),(1,1,1,1)\}$ is a set of 5 affinely independent points in $\bR^4$. For each $\delta\in \DJDd$, $J_0$ and $\{J_i\}_{i\in I^*}$ are defined as follows.\begin{displaymath}
	\begin{array}{c|c|c}
	\delta & J_0 & \{J_i\}_{i\in I^*}\\
	\hline
	(0,0,0,0) & \emptyset & \{S_1\cap S_2\}\\
	(1,0,0,0) & S_1\cap S_2 & \{S_1\setminus S_2,S_2\setminus S_1\}\\
	(1,1,0,0) & S_1 & \{S_2\setminus S_1\}\\
	(1,0,1,0) & S_2 & \{S_1\setminus S_2\}\\
	(1,1,1,1) & S_1\cup S_2 & \emptyset
	\end{array}
	\end{displaymath}
\end{example}
\vskip.5cm
We next present an alternate way to certify that $\cS$ is a  \textit{{\goodfamily} family.}  We say that $\cS$ is closed under union if for each pair $S_i,S_j\in \cS$, their union $S_i\cup S_j$ is also contained in $\cS$.

\begin{prop}\label{prop:prop}
	Let  $\cS$ be a family of nonempty subsets of $J$ that is closed under union and nonempty intersection.
	Then,  $\cS$ is a {\goodfamily} family provided that $\DJDd=\Delta^{0,n}$.
\end{prop}
\begin{proof}
	It suffices to show that $\Delta^{0,n}$ satisfies property 1 in Definition \ref{DCdef} if $\cS$ is closed under union and nonempty intersection. We show this by induction on $m$. When $m=1$, then we have $\Delta^{0,n}=\{0,1\}$ and the statement holds.
	
	For the inductive step, assume the statement holds for all $\cS$ with $m\leq k$ for a given $k\geq 1$. We will next show that the statement then  also holds for $k+1$.
	Let $\cS'=\{S_i\}_{i\in I'}$ be a family of distinct nonempty subsets of $J$ that are closed under union and nonempty intersection with $m':=|I'|=k+1$. Without loss of generality, assume $S_{1}$ is a minimal set (with respect to inclusion) in $\cS'$. Let $I_1:=\{i\in I':S_i\nsupseteq S_{1}\}=\{i\in I':S_i\cap S_{1}=\emptyset\}$ and $I_2:=\{i\in I':S_i\supsetneq S_{1}\}$, $\cS_1:=\{S_i\}_{i\in I_1}$ and $\cS_2:=\{S_i\setminus S_{1}\}_{i\in I_2}$. Note that both $\cS_1$ and $\cS_2$ are closed under union and nonempty intersection.
	
	Define\begin{align*}
	\Delta':=&\Big\{\delta\in\{0,1\}^{I'}:\exists z\in\{0,1\}^n\text{ s.t. }\delta_i=\prod_{j\in S_i}(1-z_j),i\in I'\Big\},\\
	\Delta_1:=&\Big\{\delta\in\{0,1\}^{I_1}:\exists z\in\{0,1\}^n\text{ s.t. }\delta_i=\prod_{j\in S_i}(1-z_j),i\in I_1\Big\},\\
	\Delta_2:=&\Big\{\delta\in\{0,1\}^{I_2}:\exists z\in\{0,1\}^n\text{ s.t. }\delta_i=\prod_{j\in S_i\setminus S_{1}}(1-z_j),i\in I_2\Big\}.
	\end{align*}
	Family $\cS_2$ cannot be empty as $S_1$ is minimal and $\cS'$ is closed under union with $m'= k+1\geq 2$. Now we consider two cases.
	
	First assume that $\cS_1=\emptyset$, then $|\cS_2|=|\cS'\setminus\{S_1\}|=k$. Therefore, by inductive hypothesis, $\Delta_2$ contains exactly $k+1(=m')$ affinely independent points. Then\begin{displaymath}
	\Delta'=\{0\}\cup\{(1,\delta):\delta\in \Delta_2\}
	\end{displaymath}
	is a set of $m'+1$ affinely independent points.
	
	Next, consider the case when $\cS_1\neq\emptyset$. In this case, $\cS_1$ and $\cS_2$ are closed under union and nonempty intersection with $m_1:=|I_1|\leq k$ and $m_2:=|I_2|\leq k$. Without loss of generality, assume $I_1=\{2,3,\ldots,m_1+1\}$ and $I_2=\{m_1+2,\ldots,m'\}$. By inductive hypothesis, we have that $\Delta_1$ and $\Delta_2$ contain exactly $m_1+1$ and $m_2+1$ affinely independent points, respectively. Observe that $\cS'=\{S_{1}\}\cup\cS_1\cup\{S\cup S_{1}\}_{S\in\cS_2}$. Since $\cS$ is closed under union, for each $S\in \cS_1$, there exists $i\in I_2$ such that $S\cup S_{1}=S_i$. It follows that for all $\delta\in\Delta'$ with $\delta_1=1$, for each $i_1\in I_1$, there exists $i_2\in I_2$ such that $\delta_{i_1}=\delta_{i_2}$. Therefore, there exists a mapping $F:\Delta_2\rightarrow \Delta_1$ such that\begin{displaymath}
	\Delta'=\{(0,\delta^1,0):\delta^1\in\Delta_1\}\cup\{(1,F(\delta^2),\delta^2):\delta^2\in\Delta_2\}.
	\end{displaymath}
	Then it is easy to verify that $\Delta'$ is a set of $m'+1(=|\Delta_1|+|\Delta_2|)$ affinely independent points.
	\end{proof}

	It is easy to see that given a family $\cS=\{S_i\}_{i\in I}$, the condition $\Delta^{l,u}=\Delta^{0,n}$ holds provided that $l\leq n-|\bigcup_{i\in I}S_i|$ and 
	$$u\geq \max_{\tilde{I}\subseteq I}\{|\tilde{I}|: S_i\setminus S_j\neq\emptyset,S_j\setminus S_i\neq\emptyset\text{ for any }i\neq j\in\tilde{I} \}.$$
	When the sets are nested, as in Example \ref{nestedexample}, these conditions simply reduce to $l\leq n-|S_m|$ and $u\geq 1$.

%In general, for any family $\cS$, $\Delta^{l,u}=\Delta^{0,n}$ holds if $l\leq n-|\bigcup_{i\in I}S_i|$ and $u\geq m_\cS$ where\begin{displaymath}
%m_{\cS}=\max_{\tilde{I}\subseteq I}\{|\tilde{I}|: S_i\setminus S_j\neq\emptyset,S_j\setminus S_i\neq\emptyset\text{ for any }i\neq j\in\tilde{I} \}.
%\end{displaymath}
%For the case when $\cS$ is nested as in Example \ref{nestedexample}, the conditions reduce to $l\leq n-|S_m|$ and $u\geq 1$.

\subsection{Properties of valid inequalities for $\XJDd$}\label{sec:prop}
Notice that an inequality $\alpha^Tz+\beta^T\delta\leq\gamma$ is valid for $\XJDd$  if and only if 
$$		\gamma~\geq~\max_{(z,\delta)\in \XJDd}\{\alpha^Tz+\beta^T\delta\}~=~\max_{\bar{\delta}\in \DJDd}\Big\{\beta^T\bar{\delta}+\max_{z \in \XJDd(\bar\delta)}\alpha^Tz\Big\}.$$
In other words, it is valid if and only if 
\begin{equation}
\gamma-\beta^T\bar{\delta}\geq \max_{z \in \XJDd(\bar\delta)}\alpha^Tz \label{eq:validity}
\end{equation}
holds for all $\bar{\delta}\in\DJDd$.
We next characterize some properties of facet-defining inequalities  for $\conv(\XJDd)$.
%Also note that when $\alpha=0$ then the  inequality  $\beta^T\delta\leq\gamma$ is valid for $\XJDd$  if and only if $\gamma-\beta^T{\delta}\geq 0 $ for all $\delta\in \DJDd$. In other words,

\begin{lem}\label{lem:ineq}
	Assume $\cS$ is a {\goodfamily} family. Let $\alpha^Tz+\beta^T\delta\leq \gamma$ be a facet-defining inequality for $\conv(\XJDd)$, and let $F$ be the associated facet. Then, either $F$ is defined by a facet-defining inequality of the form $(\beta')^T\delta \leq \gamma'$ which also defines a facet of $\conv(\DJDd)$, or  \eqref{eq:validity} holds as equality for all $\bar\delta\in\DJDd$.
\end{lem}
\begin{proof}
	Let $\alpha^Tz+\beta^T\delta\leq \gamma$ be a facet-defining inequality and assume that inequality \eqref{eq:validity} is strict for some $\bar\delta \in \DJDd$.
	In this case,  the facet $F$ does not contain any integral points of the form $( z, \bar\delta)$ and consequently, for all  integral points $( z, \delta)\in F$ we have $\delta\in\DJDd\setminus\{\bar\delta\}$.
	
	As $\cS$ is a {\goodfamily} family, $\conv(\DJDd)$ is a full-dimensional simplex in $\R^m$ with $m+1$ facets.
	Let $(\beta')^T\delta \leq \gamma'$ be the (unique) facet-defining inequality for $\conv(\DJDd)$ such that $\bar \delta$ is not contained in the corresponding facet $F'$. 
	Note that all points in $\DJDd \setminus \{\bar\delta\}$ satisfy $(\beta')^T\delta = \gamma'$.
	As all integral points in $F$ have their $\delta$ components in $\DJDd\setminus\{\bar\delta\}$, we conclude that all integral points in $F$ satisfy  $(\beta')^T\delta = \gamma'$.
	Therefore, $F$ is defined by the inequality $(\beta')^T\delta \leq \gamma'$.
%	Let $\alpha^Tz+\beta^T\delta\leq \gamma$ satisfy the conditions of the lemma.
%Assume that the inequality \eqref{eq:validity} is strict for some $\bar\delta \in \DJDd$.
%As $\cS$ is a {\goodfamily} family, $\conv(\DJDd)$ is a full-dimensional simplex in $\R^m$.
%Let $(\beta')^T\delta \leq \gamma'$ be the (only one) facet-defining inequality for $\conv(\DJDd)$ such that $\bar \delta$ is not contained in the corresponding facet $F'$. Note that all points in $\DJDd \setminus \{\bar\delta\}$ satisfy $(\beta')^T\delta = \gamma'$.
%
%As inequality  \eqref{eq:validity} is strict, there is no point of the form  $(\bar z, \bar\delta)\in F$.
%Therefore, $\bar\delta$ is not contained in $\proj_{\delta}(F)$. We conclude that every point $\delta'$ in $\proj_{\delta}(F)$ satisfies $(\beta')^T\delta' = \gamma'$, and so does every point in $F$.
%Therefore, $F$ is defined by the inequality $(\beta')^T\delta \leq \gamma'$.
\end{proof}

Given a proper family $\cS$ with $\DJDd=\{\mydelta^1,\ldots,\mydelta^{m+1} \}$ and a vector $\alpha\in\bR^{n}$, let $A_\cS\in\bR^{(m+1)\times(m+1)}$  be the matrix with rows\begin{displaymath}
[A_\cS]_i=[1,-(\mydelta^i)^T],\quad i=1,\ldots,m+1,
\end{displaymath}
and $\nu_\alpha\in\bR^{(m+1)}$  be the vector with entries\begin{equation}
[\nu_\alpha]_i=\max_{z\in\XJDd(\mydelta^i) }\alpha^Tz,\quad i=1,\ldots,m+1.\label{eq:nus}
\end{equation}
When \eqref{eq:validity} holds as equality for all $\bar{\delta}\in \DJDd$,
we can now write  \eqref{eq:validity}  in matrix form as $A_\cS\binom{\gamma}{\beta}=\nu_{\alpha}.$
Note that as $\cS$ is proper, the vectors in $\DJDd$ are affinely independent and therefore  $A_\cS$ is nonsingular.
Then, for any given $\alpha\in\bR^{n}$, we can construct a valid inequality $\alpha^Tz+\beta_\alpha^T\delta\leq\gamma_\alpha$ for $\XJDd$ where
\begin{equation}\label{inq:qwer}	
\left(\begin{array}{c}\gamma_\alpha\\\beta_\alpha\end{array}\right):=(A_\cS)^{-1}\nu_{\alpha}.
\end{equation}
Moreover, by Lemma \ref{lem:ineq}, if $\alpha^Tz+\beta^T\delta\leq \gamma$ defines a facet $F$ of  $\conv(\XJDd)$, then either $\beta = \beta_\alpha$ and $\gamma = \gamma_\alpha$, or, $F$ is defined by an inequality of the form $(\beta')^T\delta \leq \gamma'$.

We will need the following definition in the next lemma.
\begin{defn}
For any two vectors $\alpha,\alpha'\in\bR^n$, we say {\em$\alpha'$ follows the \alphapattern} if\begin{enumerate}
	%	\item For any $j\in J$, $(i)$ $\alpha_j> 0$ implies $\alpha'_j\geq 0$, $(ii)$ $\alpha_j= 0$ implies $\alpha'_j= 0$, and, $(iii)$ $\alpha_j< 0$ implies $\alpha'_j\leq 0$;
%	\item For any $j_1,j_2\in J$, $\alpha_{j_1}\geq \alpha_{j_2}$ implies $\alpha'_{j_1}\geq \alpha'_{j_2}$.
	\item For each $j\in J$, $(i)$ if $\alpha_j\ge 0$, then $\alpha'_j\geq 0$, and   $(ii)$ if $\alpha_j\le 0$, then $\alpha'_j\leq 0$;
	\item For each pair $j_1,j_2\in J$, if $\alpha_{j_1}\geq \alpha_{j_2}$, then $\alpha'_{j_1}\geq \alpha'_{j_2}$.
\end{enumerate}
\end{defn}

By definition, it can be shown that if $\alpha'$ follows the pattern of $\alpha$, then there exists an optimal solution of \eqref{eq:nus} such that it remains optimal if we replace $\alpha$ by $\alpha'$ in \eqref{eq:nus}.
The next lemma has a similar flavor of this observation, and will be used to show that we can put a restriction on $\alpha$ when we consider any facet-defining inequality with coefficients defined by \eqref{inq:qwer}. 
\begin{lem}\label{lem:pattern}
	Assume $\cS$ is a proper family and $(\alpha,\beta,\gamma)$ satisfies \eqref{eq:validity} as equality for all $\bar{\delta}\in\DJDd$. 
	If two vectors $\alpha^+,\alpha^-\in\R^n$  both follow the \alphapattern~and $\alpha=\lambda \alpha^++\mu\alpha^-$ for some $\lambda\geq 0$ and $\mu\geq 0$, then\begin{displaymath}
	\nu_\alpha=\lambda \nu_{\alpha^+}+\mu\nu_{\alpha^-}
	\end{displaymath}
	where $\nu_\alpha$ is defined as in \eqref{eq:nus}.
%	and inequalities\begin{equation}\label{inq:purtb}
%	(\alpha^+)^Tz+(\beta^+)^T\delta\leq\gamma^+,\quad (\alpha^-)^Tz+(\beta^-)^T\delta\leq\gamma^-
%	\end{equation}
%	are both valid for $\XJDd$ where\begin{displaymath}
%	\left(\begin{array}{c}
%	\gamma^+\\
%	\beta^+
%	\end{array}
%	\right)=(A_\cS)^{-1}\nu_{\alpha^+},\quad\left(\begin{array}{c}
%	\gamma^-\\
%	\beta^-
%	\end{array}\right)=(A_\cS)^{-1}\nu_{\alpha^-}.
%	\end{displaymath}
\end{lem}
\begin{proof}
	%Inequalities in \eqref{inq:purtb} are valid as they satisfy \eqref{eq:validity} for all $\bar{\delta}\in\DJDd$ by definition.
	As in the proof of Lemma \ref{lem:disj}, for a given  $\bar{\delta}\in\DJDd$ let $J_0:=\bigcup_{i\in I: \bar\delta_i=1}S_i$, $I_0:=\{i\in I:\bar{\delta}_i=0 \}$, let $ I^*\subseteq I_0$ denote the index set of minimal elements (with respect to inclusion) of $\{S_i: i\in I_0\}$, and let $J_i:=S_i\setminus J_0$ for $i\in I^*$. 
	Note that the optimal value of $\max_{z\in\XJDd(\bar{\delta})}\alpha^Tz$ is equal to $\sum_{i \in I^*}\max_{j\in J_i}\alpha_j$ (the sum of one largest $\alpha_j$ in $\{\alpha_j\}_{j\in J_i}$ for $i\in I^*$) plus the largest sum of at least $(\max\{l-|I^*|,0\})$ and up to $(u-|I^*|)$ largest remaining $\alpha_j$ values for $j\in J\setminus J_0$.

	For  $=1,\ldots,m+1$, let \begin{equation}
	\bar{z}^i\in\arg\max_{z\in \XJDd(\mydelta^i)}\alpha^Tz.\label{zbardef}
	\end{equation}
	As $\alpha^+$ and $\alpha^-$ both follow the \alphapattern, $\bar{z}^i$ remains optimal for \eqref{zbardef} after replacing $\alpha$ by $\alpha^+$ or $\alpha^-$, i.e.,\begin{displaymath}
	\bar{z}^i\in (\arg\max_{z\in \XJDd(\mydelta^i)}(\alpha^+)^Tz)\cap(\arg\max_{z\in \XJDd(\mydelta^i)}(\alpha^-)^Tz),
	\end{displaymath}
	 for $i=1,\ldots,m+1$. We next construct a matrix $\bar{Z}$ with   columns $\bar{z}^i$ and observe that 
	 \begin{displaymath}
		\bar{Z}^T\alpha=\nu_\alpha,\quad\bar{Z}^T\alpha^+=\nu_{\alpha^+},\text{~and,~} \bar{Z}^T\alpha^-=\nu_{\alpha^-}.
	\end{displaymath}
	Therefore, we have\begin{displaymath}
	\nu_\alpha=\bar{Z}^T\alpha=\lambda\bar{Z}^T\alpha^++\mu\bar{Z}^T\alpha^-=\lambda\nu_{\alpha^+}+\mu\nu_{\alpha^-}.
	\end{displaymath}
\end{proof}

Using this technical result, we next make an observation on the coefficients of facet-defining inequalities.

\begin{lem}\label{basicthm}
	Assume $\cS$ is a {\goodfamily} family. Then each facet $F$ of $\conv(\XJDd)$ is defined by an inequality $\bar\alpha^Tz+\bar\beta^T\delta\leq\bar\gamma$ where $\bar\alpha\in\ \{ 0, \kappa\}^{|J|} $ for some $\kappa\in\R$.
	%Then, without loss of generality,  $\bar{\alpha}\in\{0,\kappa\}^m$ for some $\kappa\in\R$.
\end{lem}
\begin{proof}
	%	We will argue that if $\bar\alpha$ has more than two distinct nonzero elements, then $\bar\alpha^Tz+\bar\beta^T\delta\leq\bar\gamma$ can be expressed as a convex combination of two distinct valid inequalities. 
	Assume that the claim does not hold. Then there is a facet $F$ such that any inequality $\alpha^Tz+\beta^T\delta\leq\gamma$ defining $F$ has the property that $\alpha$ has at least two distinct nonzero components.
	If $\conv(\XJDd)$ is full-dimensional, then there is a unique inequality (up to positive scaling) $\bar\alpha^T z + \bar\beta^T\delta \leq \bar\gamma$ defining $F$.
	If $\conv(\XJDd)$ is not full-dimensional, we chose $\bar\alpha^T z + \bar\beta^T\delta \leq \bar\gamma$ be an inequality defining $F$ such that $\bar\alpha$ has the smallest number ($\geq 2$) of distinct nonzero components.
	% (facet-defining inequalities for $F$ do not have a unique representation in terms of $(\alpha,\beta,\gamma)$ in this case).
%	\if 0
%	For each $\delta \in \DJDd$, let $\tau(\delta,\bar\alpha)=\max_{z\in\XJDd(\delta)}\bar\alpha^Tz$.
%	%satisfies $|\{\bar\alpha_j:\bar\alpha_j\neq 0,j\in J \}|\geq 2$.
%	As $\cS$ is {\goodfamily},  Lemma \ref{lem:disj} implies $\tau(\delta,\bar\alpha)$ can be obtained by maximizing $\bar\alpha^Tz$ subject to the constraints (\ref{c11})-(\ref{c14}).
%	The optimal value is equal to $\sum_{i \in I^*}(\max_{j\in J_i}\bar\alpha_j)$ plus the largest sum of at least $(\max\{l-|I^*|,0\})$ and up to $(u-|I^*|)$ largest remaining elements of $\{\bar\alpha_j:j\in J\setminus J_0 \}$.
%	Consider the set $J'(\delta) = \{j : z_j = 1\}$ defined by an optimal solution. Then
%	\[ \tau(\delta, \bar\alpha) = \sum_{j\in J'(\delta)} \bar\alpha_j. \]
%	
%	Lemma \ref{lem:ineq} implies that $\bar\gamma - \bar\beta^T\delta = \tau(\delta,\bar\alpha)$ for all $\delta \in \DJDd$, and $(\bar\beta,\bar\gamma)$ is the unique solution to this system of inequalities.
%	If we let $A$ be the matrix with rows $(1, -\mydelta^i)$ where $\mydelta^i \in \DJDd$ for $i=1, \ldots, m+1$, and $b$ be the vector defined by $b_i = \tau(\mydelta^i,\bar\alpha)$, then $(\bar\beta,\bar\gamma)$ is equal to $A^{-1}b$.
%	\fi
	
	Let $\alpha_{\min}$  denote the smallest nonzero component of $\bar\alpha$ and let $J_{\min} = \{j\in J\::\:\bar\alpha_j=\alpha_{\min}\} $. 
	Let $\alpha^+$ and $\alpha^-$ be obtained from $\bar\alpha$ as follows
	\begin{equation}
	\alpha^+_j=\left\{\begin{array}{ll}	\bar\alpha_j+\epsilon, 	&\text{if }j\in J_{\min},\\	\bar\alpha_j, 	&\text{otherwise,}	\end{array}	\right.
	~~~~~~~~~~
	\alpha^-_j=\left\{\begin{array}{ll}		\bar\alpha_j-\epsilon, 	&\text{if }j\in J_{\min},\\	\bar\alpha_j,	&\text{otherwise,}	\end{array}	\right.
	\label{eq:alphas}
	\end{equation}
	where $\epsilon>0$ is sufficiently small so that $\alpha^+$ and $\alpha^-$ follow the pattern of $\bar\alpha$.
%	\if 0
%	(i) $\alpha_{\min}+\epsilon$ and $\alpha_{\min}-\epsilon$ have the same sign as $\alpha_{\min}$, and (ii) $\alpha_{\min} +\epsilon$ and  $\alpha_{\min} -\epsilon$ is the smallest non-zero component of $\alpha^+$ and $\alpha^-$ respectively.
%	Notice that both $\alpha^+$ and $\alpha^-$ have the same sign pattern as $\bar\alpha$ and  for any $j,k\in J$, if $\bar\alpha_j\ge\bar\alpha_k$, then  $\alpha^+_j\ge\alpha^+_k$ and $\alpha^-_j\ge\alpha^-_k$. The last property implies that $\tau(\delta, \alpha^+) = \sum_{j \in J'(\delta)} \alpha^+_j$, and $\tau(\delta,\alpha^-)$ is defined similarly in terms of $\alpha^-$. Therefore $\tau(\delta,\alpha^+)$ and $\tau(\delta,\alpha^-)$ change linearly with $\epsilon$ as long as (i) and (ii) hold, and so do
%	\[ (\beta^+,\gamma^+) = A^{-1}(\tau(\mydelta^i,\alpha^+))_{i=1}^{m+1}, ~~ (\beta^-,\gamma^-) = A^{-1}(\tau(\mydelta^i,\alpha^-))_{i=1}^{m+1}. \]
%	\fi
	Then by Lemma \ref{lem:pattern}, we have $\nu_{\bar{\alpha}}=\frac{1}{2}\nu_{\alpha^+}+\frac{1}{2}\nu_{\alpha^-}$ and using \eqref{inq:qwer}, we can define two valid inequalities\begin{equation}%\label{inq:purtb}
	(\alpha^+)^Tz+(\beta_{\alpha^+})^T\delta\leq\gamma_{\alpha^+},\quad (\alpha^-)^Tz+(\beta_{\alpha^-})^T\delta\leq\gamma_{\alpha^-}.
%	(\alpha^+)^Tz+(\beta^+)^T\delta\leq\gamma^+,\quad (\alpha^-)^Tz+(\beta^-)^T\delta\leq\gamma^-
	\end{equation}
%	 with \begin{displaymath}
%	\left(\begin{array}{c}
%	\gamma^+\\
%	\beta^+
%	\end{array}
%	\right)=(A_\cS)^{-1}\nu_{\alpha^+},\quad\left(\begin{array}{c}
%	\gamma^-\\
%	\beta^-
%	\end{array}\right)=(A_\cS)^{-1}\nu_{\alpha^-}.
%	\end{displaymath}
	In this way, $(\bar\beta,\bar\gamma) =\frac12(\beta_{\alpha^+},\gamma_{\alpha^+})+\frac12(\beta_{\alpha^-},\gamma_{\alpha^-})$. Consequently, $\bar\alpha^Tz+\bar\beta^T\delta\leq\bar\gamma$ can be expressed as a strict convex combination of two valid inequalities. 
	Moreover, these two inequalities are distinct (not a multiple of the original inequality) as $|\{\bar\alpha_j:\bar\alpha_j\neq 0,j\in J \}|\geq 2$.
	When $\conv(\XJDd)$ is full dimensional, this leads to a contradiction.
	% and we can therefore conclude that  $\bar{\alpha}\in\{0,\alpha_{\min}\}^m$.
	
	On the other hand, if $\conv(\XJDd)$ is contained in an affine subspace, then it is possible that both inequalities define the same facet as the original one.
	%it is possible that intersection of both inequalities (and the original one) with this affine space is the same. In this case, 
	In this case, we can increase $\epsilon$ in \eqref{eq:alphas} as much as possible while  $\alpha^+$ and $\alpha^-$ follow the \alphapattern. 
	The largest such $\epsilon$ would give an $\alpha^+$ or $\alpha^-$ with one fewer distinct nonzero entries than $\alpha$.
	This again  leads to a contradiction as $\alpha$ was assumed to have the smallest number of distinct nonzero components. 
	%Repeating this procedure if necessary, we can again  conclude that  $\bar{\alpha}\in\{0,\alpha_{\min}\}^m$.
\end{proof}

We conclude this section by showing that the convex hull of $\XJDd$ can simply be obtained from convex hulls of $\XJD$ and $\XJd$ provided that $\cS$ satisfies some simple conditions.
\begin{thm}\label{thm:decomp}
	Assume $\cS$ is a proper family and $\DJDd=\Delta^{0,n}$. Then\begin{displaymath}
	\conv(\XJDd)=\conv(\XJD)\cap\conv(\XJd).
	\end{displaymath}
\end{thm}

\begin{proof}
 	As $\XJDd=\XJD\cap\XJd$, we have  $\conv(\XJDd)\subseteq \conv(\XJD)\cap\conv(\XJd)$. 
	We next show that the reverse inclusion also holds. %$\conv(\XJDd)\supseteq\conv(\XJD)\cap\conv(\XJd)$. 
	We first consider the case when $\conv(\XJDd)$ is not full-dimensional and argue that the affine hull of $ \conv(\XJDd)$ is the same as that of $\conv(\XJD)\cap\conv(\XJd)$.
	Let  $\alpha^Tz+\beta^T\delta=\gamma$ be an equation satisfied by all points in $\conv(\XJDd)$. 
	Consider now only one direction of the equation $\alpha^Tz+\beta^T\delta\leq\gamma$. Using the notation defined in the proof of Lemma \ref{lem:ineq}, we have $(\beta,\gamma)=(A_\cS)^{-1}\nu_\alpha$ as $\alpha^Tz+\beta^T\mydelta^i=\gamma$ for any $z\in\XJDd(\mydelta^i)$. Let $\alpha^+$ and $\alpha^-$ denote the nonnegative part and the nonpositive part of $\alpha$, respectively, i.e.,\begin{displaymath}
	\alpha^+_j=\max\{\alpha_j,0 \},\ \alpha^-_j=\min\{\alpha_j,0 \},\quad j=1,\ldots,n.
	\end{displaymath}
	Letting $\gamma^+=\gamma_{\alpha^+}$,  $\beta^+=\beta_{\alpha^+}$, $\gamma^-=\gamma_{\alpha^-}$,  $\beta^-=\beta_{\alpha^-}$ as defined in \eqref{inq:qwer}, we see that the following inequalities	are valid for $\XJDd$: 
	\begin{align*}
	(\alpha^+)^Tz+(\beta^+)^T\delta\leq\gamma^+,\quad(\alpha^-)^Tz+(\beta^-)^T\delta\leq\gamma^-.
	\end{align*}
%	\begin{displaymath}
%	\left(\begin{array}{c}
%	\gamma^+\\
%	\beta^+
%	\end{array}
%	\right)=(A_\cS)^{-1}\nu_{\alpha^+},\quad\left(\begin{array}{c}
%	\gamma^-\\
%	\beta^-
%	\end{array}\right)=(A_\cS)^{-1}\nu_{\alpha^-}.
%	\end{displaymath}
	Moreover, as $\alpha=\alpha^++\alpha^-$ and both $\alpha^+$ and $\alpha^-$  follow the \alphapattern, by Lemma \ref{lem:pattern} we have $\nu_\alpha=\nu_{\alpha^+}+\nu_{\alpha^-}$, and therefore $\beta^++\beta^-=\beta$ and $\gamma^++\gamma^-=\gamma$.
	Note that when $\bar{\alpha}\geq 0$, $\nu_{\bar{\alpha}}$ does not depend on $l$ as its $i$-th entry is equal to $\sum_{i \in I^*}(\max_{j\in J_i}\bar\alpha_j)$ plus the sum of the $(u-|I^*|)$ largest remaining $\alpha_j$ values for $j\in J\setminus J_0 $, where $I^*,\{J_i\}_{i\in I^*\cup\{0\}}$ are associated with $\mydelta^i$ (as defined in Lemma \ref{lem:disj}). It follows that\begin{displaymath}
	[\nu_{\alpha^+}]_i=\max_{z\in\XJDd(\mydelta^i)}(\alpha^+)^Tz=\max_{z\in\XJD(\mydelta^i)}(\alpha^+)^Tz, ~~i=1,\ldots,m+1.
	\end{displaymath}
	This implies that $(\alpha^+)^Tz+(\beta^+)^T\delta\leq\gamma^+$ is valid for $\conv(\XJD)$. Using a similar argument it is easy to see that $(\alpha^-)^Tz+(\beta^-)^T\delta\leq\gamma^-$ is valid for $\conv(\XJd)$. Note that $(\alpha,\beta,\gamma)=(\alpha^+,\beta^+,\gamma^+)+(\alpha^-,\beta^-,\gamma^-)$. Combining both inequalities, we have $\alpha^Tz+\beta^T\delta\leq\gamma$ is valid for $\conv(\XJD)\cap\conv(\XJd)$.
	
	When we consider the other direction $-\alpha^Tz-\beta^T\delta\leq-\gamma$, by repeating the argument above for $(-\alpha,-\beta,-\gamma)$, we see that $-\alpha^Tz-\beta^T\delta\leq-\gamma$ is valid for $\conv(\XJD)\cap\conv(\XJd)$.
	This implies that $\alpha^Tz+\beta^T\delta=\gamma$ is valid for $\conv(\XJD)\cap\conv(\XJd)$.
	
 	We now consider an arbitrary facet $F$ of $\conv(\XJDd)$. Let $\DJDd=\{\mydelta^i \}_{i=1}^{m+1}$. As $\DJDd\subseteq\DJD,\DJd\subseteq\Delta^{0,n}$, the assumption of the theorem implies that 
	\begin{equation}\label{eq:allsame}\DJDd=\Delta^{0,n}=\DJD=\DJd.\end{equation}

	  By Lemmas \ref{lem:ineq} and  \ref{basicthm}, we only need to discuss the following two cases:\begin{enumerate}
	\item $F$ can be defined by an inequality $(\beta')^T\delta\leq\gamma'$ which also defines a facet of $\DJDd$. In this case, by \eqref{eq:allsame},  we have $\DJDd=\Delta^{l,n}=\Delta^{0,u}$, and $(\beta')^T\delta\leq\gamma'$ is also valid for $\conv(\XJD)\cap\conv(\XJd)$.
	\item $F$ can be defined by an inequality $\bar{\alpha}^Tz+\bar{\beta}^T\delta\leq\bar{\gamma}$ where $\bar{\alpha}_j\in\{0,\kappa \}$ for some $\kappa\in\bR$ and \eqref{eq:validity} holds as equality for all $\bar{\delta}\in\DJDd$. 
	If $\kappa\geq 0$, then $\bar{\alpha}^Tz+\bar{\beta}^T\delta\leq\bar{\gamma}$ is valid for $\conv(\XJD)$. 
	On the other hand, if $\kappa\leq 0$, then $\bar{\alpha}^Tz+\bar{\beta}^T\delta\leq\bar{\gamma}$ is valid for $\conv(\XJd)$. 
	In both cases, $\bar{\alpha}^Tz+\bar{\beta}^T\delta\leq\bar{\gamma}$ is also valid for $\conv(\XJD)\cap\conv(\XJd)$.
	\end{enumerate}
	
	We therefore conclude that any inequality valid for $\conv(\XJDd)$ is also valid for $\conv(\XJD)\cap\conv(\XJd)$, and consequently $\conv(\XJD)\cap\conv(\XJd)\subseteq\conv(\XJDd)$.
\end{proof}

\section{Convex hull description when  $\cS$ is a family of nested sets}\label{nestedsec}
%=============================================================================================
In this section, we consider the special case when $\cS=\{S_i\}_{i\in I}$ is a family of nested sets.
%As before, let $I=\{1,\ldots, m\}$, and $J=\{1,\ldots, n\}$.
In other words, we assume  that $S_1\subset S_2\subset\ldots\subset S_m\subset J=\{1,\ldots, n\}$, and without loss of generality, we use $S_i=\{1,\ldots, k_i\}$ where $2\leq k_1<k_2<\ldots<k_m$. Remember that  $I=\{1,\ldots, m\}$.
To avoid trivial cases (see Remark \ref{remark} below), we further assume that $u \geq 2$ and $l \leq n - |S_m|$ (i.e., $k_m \leq n-l$).
For convenience, we define $S_0=\emptyset$, $S_{m+1}=J$, $\delta_0=1$ and $\delta_{m+1}=0$. 

Without loss of generality, we also  assume that $l<u$. 
Note that if  $l=u$,  then $z_n=u-\sum_{j\in J\setminus\{n\}}z_j$ and  % $\bigcup_{i\in I}S_i=S_m\subseteq \{1,\ldots,n-l\}$. In this case, 
any problem of the form $\min\{c^Tz+d^T\delta:~(z,\delta)\in\XJDd\}$ is equivalent to\begin{align*}
\min\Big\{\sum_{j\in J\setminus\{n\}}c_jz_j+c_n(u-\sum_{j\in J\setminus\{n\}}z_j)+d^T\delta:~&\delta_i=\prod_{j\in S_i}(1-z_j),i\in I;\\
&u-1\leq\sum_{j\in J\setminus\{n\}}z_j\leq u;z_j\in\{0,1\},j\in J\setminus\{n\}\Big\}
\end{align*}
and we can then work in the projected space without variable $z_n$.
%We will assume $l<u$ when we consider nested $\cS$.

\subsection{Basic properties of $\conv(\XJDd)$ and its continuous relaxation}
Recall from Example~\ref{nestedexample} that $\cS$ is a proper family. % of sets with respect to $\XJDd$.
As $S_i\subset S_{i+1}$, all  $(z,\delta) \in \XJDd$ satisfy $\delta_{i+1}\le\delta_i$ for all $i < m$.
Moreover, if $z_j=0$ for all $j\in S_{i+1}\setminus S_i$, then $\delta_{i+1}=\delta_i$. 
Consequently, the following inequalities are valid for $\conv(\XJDd)$ for all $i=1,\ldots,m-1$:
\begin{align}	\delta_{i+1}-\delta_i\leq& 0, \label{mc4}\\
\delta_i-\delta_{i+1}-\sum_{j\in S_{i+1}\setminus S_i}z_j\leq& 0, \label{mc5}
\end{align}
These inequalities are called 2-link inequalities by Crama and Rodr{\'\i}guez-Heck \cite{crama2017class}.
When $\cS$ is nested, Fischer, Fischer and McCormick \cite{fischer2018matroid} show that  (\ref{mc4})-(\ref{mc5})  along with the standard linearization (\ref{xmc0})-(\ref{xmc7}) define the convex hull of  $X^{0,n}$ (i.e. when  $l = 0$, $u = n$).
Crama and Rodr{\'\i}guez-Heck \cite{crama2017class} show the same result holds when  $|\cS|=2$ without assuming $\cS$ is nested.

%They show that these inequalities along with the standard linearization in (\ref{xmc0}) - (\ref{xmc7}) define the convex hull of $\XJDd$, for a nested family of sets $\cS$ in the case $l = 0$ and $u = n$, i.e., when the cardinality constraints are removed.

After adding  (\ref{mc4})-(\ref{mc5})  to the standard linearization of $\XJDd$, some of the initial inequalities \eqref{xmc0}-\eqref{xmc7}  become redundant.  We next give the subset of the inequalities \eqref{xmc0}-\eqref{xmc7} that give a correct formulation when combined with  (\ref{mc4})-(\ref{mc5}):
%the following inequalities give a strengthened McCormick relaxation $\conv(\XJDd)$:
\begin{align}
l\leq\sum_{j\in J}z_j\leq& u,\label{mc3}\\
z_j+\delta_i\leq& 1, && j\in S_i,i\in I,\label{mc1}\\
1-\delta_1-\sum_{j\in S_1}z_j\leq& 0,\label{mc2}\\
%\delta_{i+1}-\delta_i\leq& 0, &&i\in\{1,\ldots,m-1\},\label{mc4}\\
%\delta_i-\delta_{i+1}-\sum_{j\in S_{i+1}\setminus S_i}z_j\leq& 0, &&i\in\{1,\ldots,m-1\},\label{mc5}\\
-\delta_m\leq& 0, &&\label{mc6}\\ %i\in I,
-z_j\leq& 0, &&j\in J\label{mc7},\\
z_j\leq&1, &&j\in J\setminus S_m.\label{mc8}
\end{align}
Note that unlike inequality \eqref{xmc2}, %in the standard linearization \eqref{xmc0}-\eqref{xmc7},  
inequality \eqref{mc2} is only written for $S_1$ as \eqref{mc5} and \eqref{mc2} together imply the remaining inequalities in \eqref{xmc2}.
Similarly, \eqref{mc4} and \eqref{mc6} imply that each $\delta_i$ is nonnegative. % due to inequalities \eqref{mc4}.

%Also note that if $\delta\in\DJDd$ then $\delta_i=1$ for all $i\le p$ and $\delta_i=0$ for all $i\ge p+1$ for some $0\leq p\leq m$.
Also note  that given any $z\in\{0,1\}^{n}$ satisfying $l\leq\sum_{j\in J}z_j\leq u$, there exists a unique $\delta$ such that $(z,\delta)\in\XJDd$.
We next define this formally.
\begin{defn}\label{def:pts}
Given $U\subseteq J$ with $l\leq|U|\leq u$, we  define the point  $v^U\in\XJDd$ as follows: 
	%$=(z^U,\delta^U)\in\XJDd$, where %for each $j\in J$,
\begin{displaymath}v^U=(z^U,\delta^U)\text{~~~where~~~}
z^U_j=\left\{\begin{array}{ll}
1, & \text{if }j\in U,\\
0, & \text{otherwise},
\end{array}
\right.
%	\end{displaymath}	and for each $i\in I$,\begin{displaymath}
~~~\text{and}~~~~~
\delta^U_i=	\prod_{j\in S_i}(1-z^U_j).
\end{displaymath}
\end{defn}

\begin{lem}\label{dimlem}
	The polytope $\conv(\XJDd)$ is full-dimensional.
\end{lem}
\begin{proof}(sketch) We consider the two following cases separately:
	
	If $l=0$, we argue that the following $m+n+1$ points in $\conv(\XJD)$ are  affinely independent: 
	$$ 	v^{\{2\}}	=\begin{bmatrix}	~e^2~	\\		~0_m~	\end{bmatrix},~~	
	v^{\{k_i+1\}}	=\begin{bmatrix}	~ e^{k_i+1}~\\	~~d^i~		\end{bmatrix}~~\text{for}~ i\in I,~~
	v^{\{1\}} 	=\begin{bmatrix}	~ e^1~	\\	0_m	\end{bmatrix},~~
	v^{\{1,j\}} 	=\begin{bmatrix}	~ e^1+e^j~	\\	~~0_m	\end{bmatrix}~~\text{for}~ j\in J\setminus\{1\}. %d^{q(j)}
	$$
	
	If, on the other hand, $l\geq 1$, then we let $Q=\{n-l+1,\ldots,n\}\subseteq J\setminus S_m$ where $|Q|=l$, and consider the following $m+n+1$ points in $\conv(\XJDd)$:
	$$ 	v^{\{2\}\cup Q}	,~~~	v^{\{k_i+1\}\cup Q}~~\text{for}~ i\in I,~~~
	v^{\{1\}\cup Q} ,~~~
	v^{\{1,j\}\cup Q\setminus\{n\}}~~\text{for}~ j\in J\setminus\{1\}\setminus Q,~~~
	v^{\{1\}\cup Q\setminus\{j\}}~~\text{for}~ j\in Q,
	$$
	and argue that they are affinely independent.
	The detailed proof is presented in Appendix.
\end{proof}

Lemma \ref{dimlem} also implies that $\conv(\DJDd)$ %, the convex hull of the  orthogonal projection of $\XJDd$ onto the space of $\delta$ variables,  
is  full-dimensional.
In addition, it is easy to see that 
\begin{align}
\DJDd &= \big\{\delta\in\{0,1\}^m:\delta_1\geq\delta_2\geq\ldots\geq\delta_m\big\}. \label{eqdelta1}\\
\intertext{Moreover, as the constraint matrix defining $\DJDd$ above is totally unimodular, we also have}
\conv(\DJDd) &= \big\{\delta\in\R^m:1\ge\delta_1\geq\delta_2\geq\ldots\geq\delta_m\ge0\big\}. \label{eqdelta2}
\end{align}
From now on we will denote the $m+1$ vectors in $\DJDd$ as 
\begin{align}\DJDd~=~\big\{\mydelta^{[0]},\mydelta^{[1]},\ldots,\mydelta^{[m]}\big\}\end{align}
where $\mydelta^{[0]} = {\bf 0}$ and, for  $i \in I$, the vector $\mydelta^{[i]}$ has the first $i$ components equal to 1 and the rest equal to zero.
Note that these vectors are affinely independent. % and consequently $\conv(\DJDd)$ is full-dimensional.

We start with characterizing facet-defining inequalities for $\conv(\XJDd)$ that have zero coefficients for all of the $z_j$ variables.
\begin{lem}\label{zerocoef}
	If $\beta^T\delta\leq\gamma$ defines a facet of $\conv(\XJDd)$, then it is a multiple of an inequality from (\ref{mc4}) or (\ref{mc6}).
\end{lem}
\begin{proof}
	%Assume $\beta^T\delta\leq\gamma$ is  a facet-defining inequality for $\conv(\XJD)$.
	As  $\conv(\XJDd)$ and $\conv(\DJDd)$ are full-dimensional polytopes, if $\beta^T\delta\leq\gamma$ defines a facet of $\conv(\XJDd)$, then it also defines a facet of $\conv(\DJDd)$. 			
	The only facet-defining inequality for $\conv(\DJDd)$, see (\ref{eqdelta2}), that is not of the form (\ref{mc4}) or (\ref{mc6}) is  $1\ge\delta_1$. However,  $1\ge\delta_1$ cannot define a facet of $\conv(\XJDd)$ as it is implied by \eqref{mc1} and \eqref{mc7} for $i=1$ and any $j\in S_1$.
	Therefore, the only facet-defining inequalities of  $\conv(\DJDd)$ that can also define facets of $\conv(\XJDd)$ are of the form (\ref{mc4}) and (\ref{mc6}).
\end{proof}

Under the assumptions (i) $u \geq 2$ and (ii)  $l \leq n - |S_m|$, we have $\DJDd=\Delta^{0,n}$ and consequently
%to characterize $\conv(\XJDd)$, we only need to characterize $\conv(\XJD)$ and $\conv(\XJd)$ separately.
$$\conv(\XJDd)=\conv(\XJD)\cap\conv(\XJd),$$
 by Theorem \ref{thm:decomp}.
We next study $\conv(\XJD)$ and $\conv(\XJd)$ separately.
\subsection{Trivial facets of $\conv(\XJD)$}\label{nestedconv}
%=====================================================================	

As $\conv(\XJD)$ is full-dimensional, all facet-defining inequalities for $\conv(\XJD)$ are uniquely defined up to multiplication by a positive scalar. We have already characterized all facets of the form $\beta^T\delta\leq\gamma$ in Lemma \ref{zerocoef}. We now characterize facet-defining inequalities of the form $\alpha^Tz+\beta^T\delta\leq\gamma$ for $\conv(\XJD)$ with $\alpha\leq 0$ and $\alpha\neq 0$.

\begin{lem}\label{negativecoef}
	Let $\alpha^Tz+\beta^T\delta\leq\gamma$ be a facet-defining inequality for $\conv(\XJD)$. If $\alpha\leq 0$ and $\alpha\neq 0$, then
	the inequality is a multiple of one of the inequalities  (\ref{mc5}),(\ref{mc2}) or (\ref{mc7}).
\end{lem}
\begin{proof}
	Let $\alpha^Tz + \beta^T\delta \leq \gamma$ be a facet-defining inequality for $\conv(\XJD)$ satisfying the conditions of the Lemma, and assume it defines the facet $F=\{( z, \delta) \in\conv(\XJD)\::\:\alpha^T z + \beta^T \delta = \gamma\}$. 
	By Lemma~\ref{basicthm}, we can assume without loss of generality that all nonzero components of $\alpha$ are equal to $-1$.
	If $\alpha^Tz + \beta^T\delta \leq \gamma$ is the same as $z_j \geq 0$ for some $j \in J$, then it is one of the inequalities in \eqref{mc7} and the result follows. We will henceforth assume this is not the case.

	If  $|S_i\setminus S_{i-1}|\ge2$ for any $i \in I$, then we will next argue that
%	Consider any $i \in I$ such that $S_i\setminus S_{i-1}$ has multiple elements. We will prove, by contradiction, that
	\begin{equation} \alpha_j = \alpha_k  \mbox{ for all distinct } j,k \in S_i \setminus S_{i-1}. \label{eq-equalalpha}\end{equation} If this is not true, then we can assume, without loss of generality, that $\alpha_j = -1$ and $\alpha_k = 0$ for some $j,k\in S_i\setminus S_{i-1}$. 
	As $F$ is not defined by $z_j \geq 0$, we can assume that there is a point $(\bar z,\bar \delta) \in F$ with the property that $\bar z_j = 1$.
	Consider the point $(z',\bar \delta)$ where the components of $z'$ are the same as the components of $\bar z$, except that $z'_j = 0$ and $z'_k = 1$. It is easy to see that $(z',\bar \delta) \in \XJD$, and  $\alpha^Tz' =  \alpha^T\bar z+1 $ which implies that $\alpha^T z' + \beta^T\bar\delta > \gamma$. This contradicts the fact that $\alpha^Tz + \beta^T\delta \leq \gamma$ is a valid inequality for $ \XJD$.
	Similarly, for any $j\in J\setminus S_m$, there exists a point  $(\bar z,\bar \delta) \in F$ such that $\bar z_j = 1$.
	If $\alpha_j=-1$, then constructing a new point by changing  $\bar z_j$ to 0 shows that the inequality cannot be valid.
	Consequently, $\alpha_j=0$ for all  $j\in J\setminus S_m$.
	
	As $\conv(\XJD)$ is full-dimensional, Lemma~\ref{lem:ineq} and $\alpha \neq 0$ together imply that 
	\begin{equation}\label{simple1}	\gamma-\beta^T\mydelta^{[i]}=\max_{z\in\XJD(\mydelta^{[i]})}\alpha^Tz \end{equation}
	for $i=0, \ldots, m$.
 	First note that as $\alpha \leq 0$ and $z\ge0$, we have  $\max_{z\in \XJD(\mydelta^{[m]})} \alpha^Tz = 0.$
	Moreover, 	for $i=0, \ldots, m-1$, if  $\bar z \in \XJD(\mydelta^{[i]})$, then $\bar z_j = 0$ for $j \in S_i$ and $\sum_{j \in S_{i+1}\setminus S_i} \bar z_j \geq 1$. Therefore, 
	\[ \max_{z\in\XJD(\mydelta^{[i]})}\alpha^Tz = \max_{j\in S_{i+1}\setminus S_i}\{\alpha_j\}. \]
	Consider  $\theta\in\R^{m+1}$ where $\theta_i$ equals to the right-hand side of (\ref{simple1}). Then $\theta_m = 0$, and for $i=0, \ldots, m-1$ we have $\theta_i \in \{0,-1\}$, with $\theta_i = \alpha_j$ for all $j \in S_{i+1} \setminus S_i$.
	Then (\ref{simple1}) implies that
	$$	\gamma=\theta_0,~~~
	\gamma-\sum_{i=1}^k\beta_i= \theta_k\text{ for } k\in\{1,\ldots,m-1\},~~~
	\gamma-\sum_{i=1}^m\beta_i=0.
	$$
	These equations have the unique solution:
	\begin{equation}\gamma=\theta_0,~~~
	\beta_i=\theta_{i-1}-\theta_{i}\text{ for } i\in\{1,\ldots,m-1\},~~~
	\beta_m=\theta_{m-1}. \label{eq-betaval}\end{equation}
	
	We next observe that  $\alpha_j=0$ for all  $j\in J\setminus S_m$ and $\theta_i\le0$ for all $i\in\{0\}\cup I$, and therefore
	\begin{align*}
	\alpha^Tz+\beta^T\delta = 
	&~\sum_{i =1}^m\theta_{i-1}\Big(\sum_{j\in S_i\setminus S_{i-1}}z_j\Big)
	+\sum_{i=1}^{m-1}(\theta_{i-1}-\theta_{i})\delta_i
	+\theta_{m-1}\delta_m\\
%	=&~\theta_0(\delta_1+\sum_{j\in S_1}z_j)+\sum_{i=1}^{m-1}\theta_{i}(\delta_{i+1}-\delta_{i}+\sum_{j\in S_{i+1}\setminus S_{i}}z_j),\\
%	\gamma = &\theta_0 + \sum_{i=1}^{m-1} \theta_i \times 0\\
		=&~\theta_0\Big(\underbrace{\delta_1+\sum_{j\in S_1}z_j}_{\ge1}\Big)+\sum_{i=1}^{m-1}\theta_{i}\Big(\underbrace{\delta_{i+1}-\delta_{i}+\sum_{j\in S_{i+1}\setminus S_{i}}z_j}_{\ge0}\Big)\le \theta_0+0 ~=~ \gamma.
		\end{align*}
	Therefore, inequality  $\alpha^Tz+\beta^T\delta\leq\gamma$ is implied by inequalities (\ref{mc2}) and  (\ref{mc5}). As it is facet-defining, it must indeed be one of them.
\end{proof}

\subsection{Convex hull description of $\XJD$}\label{nested_mixsect}
%========================================================================
We next derive a family of valid inequalities for  $\conv(\XJD)$ using the  mixing procedure \cite{gunluk2001mixing}.
The inequalities we derive here apply when $\cS=\{S_i\}_{i\in I}$ is a family of nested sets and as we show later, together with inequalities  (\ref{mc3})-(\ref{mc8}), they give a complete description of $\conv(\XJD)$.
Later in Section \ref{general_mixsect}, we will generalize these inequalities for the case when $\cS$ is not necessarily nested.

For some positive integer $k$, let  $1> b_k >b_{k-1}>\ldots>  b_1 >0,$ be given and let
\begin{equation}
Q~=~\Big\{s\in\R,z\in\Z^k\::\: s+z_i\ge b_i\quad \text{for }i=1,\ldots,k,~~s\ge0\Big\}.\label{eq:mixset}\end{equation} 
Then, the following {\em type I mixing inequality} is known to be valid for $Q$ (see \cite{gunluk2001mixing}):
\begin{alignat}{10}
s&+  b_1z_1 &+& \sum_{i=2}^k (b_i-b_{i-1})z_i &~\ge~& b_k\label{eq:mixing1}.%\\
%s&+ (b_1+1-b_n)z_1 &+& \sum_{k=2}^n (b_k-b_{k-1})z_k &~\ge~& b_n.\label{eq:mixing2}
\end{alignat}
The inequalities $s+z_i\ge b_i$ are called  {\em base} inequalities and note that inequality \eqref{eq:mixing1} combines the  mixed-integer rounding inequalities  $  s + b_i z_i\ge b_i$ associated with the base inequalities using a ``telescopic'' sum. 
We next derive some valid inequalities for $\XJD$ to use as base inequalities for applying the mixing procedure.

Let $S'\subseteq J$  and  $i\in I$ be given and let $M>n$ be a fixed constant. 
Using the fact that $z_j\le1$,  $1-\delta_i-z_j\ge 0$, and $\delta_i\le1$, for all $j\in S_i$,
we can derive the following valid (base) inequality for  $\conv(\XJD)$:
\begin{align*}
\frac1{M}\Big({u-\sum_{j\in S'}z_j}\Big)+(1-\delta_i)
&=~ \frac{1}{M} \Big( u-\sum_{j\in S'\setminus S_i}{z_j}\Big)
+\frac{1}{M} \sum_{j\in S'\cap S_i}({1-\delta_i-z_j})
+\frac{1}{M} \Big({M-|S'\cap S_i|}\Big)(1-\delta_i)\\[.2cm]
&\geq~~ \frac{1}{M} \Big( u-\sum_{j\in S'\setminus S_i}{z_j}\Big)~~\ge~~\frac{1}{M}\Big( u-| S'\setminus S_i|\Big).
\end{align*}
%\begin{align*}
%\frac1{M}({u-\sum_{j\in S'}z_j})+(1-\delta_i)
%=&~ \frac{u}{M}-\sum_{j\in S'\setminus S_i}\frac{z_j}{M}
%+\sum_{j\in S'\cap S_i}\frac{1-\delta_i-z_j}{M}
%+\frac{M-|S'\cap S_i|}{M}(1-\delta_i)\\
%\geq &\frac{u}{M}-\sum_{j\in S'\setminus S_i}\frac{1}{M}+\sum_{j\in S'\cap S_i}\frac{0}{M}+\frac{M-|S'\cap S_i|}{M}\cdot 0~=~\frac{u-|S'\setminus S_i|}{M}.
%\end{align*}
Moreover, if  $|S'\setminus S_p|\leq u-1$ for some $p\in I$, then the right-hand side of this valid inequality
$$ \frac1{M}\Big({u-\sum_{j\in S'}z_j}\Big)+(1-\delta_i)\geq~ \frac{1}{M}\Big( u-| S'\setminus S_i|\Big)~~~$$
is strictly between 0 and 1 for all $i=p,p+1,\ldots,m$. 
Therefore, we can write a set of the form \eqref{eq:mixset}  using these  inequalities as the base inequalities where we  treat the term $\frac1{M}({u-\sum_{j\in S'}z_j})$  as a nonnegative continuous variable and the term $(1-\delta_i)$ as an integer variable for all $i=p,p+1,\ldots,m$.
Consequently, the resulting type I mixing inequality,
\begin{displaymath}
\frac1{M}\Big({u-\sum_{j\in S'}z_j}\Big)+\frac1{M}\Big({u-|S'\setminus S_p|}\Big)(1-\delta_p)
+\frac1{M}\sum_{i=p+1}^{m}\Big({|S'\setminus S_{i-1}|-|S'\setminus S_i|}\Big)(1-\delta_i)\ge \frac{1}{M}\Big( u-| S'\setminus S_m|\Big),
\end{displaymath}
which can be simplified to
\begin{displaymath}
\frac1{M}\Big({u-\sum_{j\in S'}z_j}\Big)\geq\frac1{M}\Big({u-|S'\setminus S_p|}\Big)\delta_p
+\frac1{M}\sum_{i=p+1}^{m}\Big({|S'\setminus S_{i-1}|-|S'\setminus S_i|}\Big)\delta_i,
\end{displaymath}
is valid for $\conv(\XJD)$.
After multiplying  the inequality by $M$ and rearranging the terms, we obtain the following valid inequality  for  $\conv(\XJD)$

\begin{equation}\label{nestedmix}
\sum_{j\in S'}z_j+ \Big(u-|S'\setminus S_p|\Big)\delta_p+\sum_{i=p+1}^m\Big(|S'\setminus S_{i-1}|-|S'\setminus S_i|\Big)\delta_i\leq u.
\end{equation}

We next give an inequality description  of  $\conv(\XJD)$ using the mixing inequalities.

%==========================================================
\begin{thm}\label{thm:0u}
	A complete inequality description of  $\conv(\XJD)$ is given by inequalities (\ref{mc3})-(\ref{mc8}) together with inequalities (\ref{nestedmix})  for all $p\in I$ and $S'\subseteq J$ such that  $|S'\setminus S_p|\leq u-1$.
%	Inequalities (\ref{mc3})-(\ref{mc8}) together with inequalities \eqref{nestedmix},  for all $p\in I$ and $S'\subseteq J$ such that  $|S'\setminus S_p|\leq u-1$, give a complete description of $\conv(\XJD)$.
\end{thm}
\begin{proof} 
	Let  $\alpha^Tz+\beta^T\delta\leq\gamma$  be a facet-defining inequality for $\conv(\XJD)$ and note that by Lemma \ref{dimlem} it has a unique representation up to multiplication.
	By Lemma \ref{basicthm}, we can assume that either $\alpha\in\{0,1\}^{|J|} $ or  $\alpha\in\{0,-1\}^{|J|} $.
	Furthermore, by Lemmas \ref{zerocoef} and \ref{negativecoef} we have established that if  $\alpha\leq 0$ (including the case when $\alpha=0$) the inequality  $\alpha^Tz+\beta^T\delta\leq\gamma$ has to be  one of (\ref{mc3})-(\ref{mc8}).
	Therefore, the only remaining case to consider is when $\alpha\in\{0,1\}^{|J|} $ and $\alpha\not=0$.
	
	Let $\bar{S}:=\{j\in J:\alpha_j=1 \}$ and therefore $\alpha^Tz=\sum_{j\in \bar{S}}z_j$.
	Also remember that $\DJD=\{\mydelta^{[0]},\ldots,\mydelta^{[m]}\}$ where the first $p\in I$ components of  $\mydelta^{[p]}\in\{0,1\}^m$ are 1, and the rest components are 0. 
	Then by  Lemma \ref{lem:ineq} the following equations must hold for all  $\mydelta^{[p]}$  with $p\in\{0,\ldots, m-1\}$, 
	\begin{align}
	\gamma-\sum_{i=1}^p\beta_i
	&=~\max\Big\{\bar\alpha^Tz:(z,{\mydelta^{[p]}})\in \XJD\Big\}\nonumber\\
	&=~\max\Big\{\sum_{j\in \bar{S}}z_j:\sum_{j\in J}z_j\leq u;~z_j=0, ~\forall j\in S_p;~
	\sum_{j\in S_{p+1}\setminus S_p}z_j\geq 1~,~z\in\{0,1\}^{|J|}\Big\}\nonumber\\
	&=~\min\Big\{u-\mathbbm{1}_{\{\bar{S}\cap S_{p+1}\setminus S_p=\emptyset \}},~|\bar{S}\setminus S_p|\Big\},\label{uno}
	%	\text{***OR***} ~\min\Big\{~u,|\bar{S}\setminus S_p|-\mathbbm{1}_{\{\bar{S}\cap S_{p+1}\setminus S_p=\emptyset \}}~\Big\}.???
	\end{align}
		where we define $\mathbbm{1}_{A}$ to be 1 if condition $A$ is true, and 0, otherwise.
	Similarly, for ${\mydelta^{[m]}}$, we have 
	\begin{equation}	\gamma-\sum_{i=1}^m\beta_i=\min\{u,~|\bar{S}\setminus S_m|\}. \label{duo}	\end{equation}
	Let $\bar S_i=\bar S\cap S_i$ for $i\in I$  and let $D_1=\bar S_1$ and $D_i=\bar S_{i}\setminus \bar S_{i-1}$  for $i\in \{2,\ldots,m\}$. Note that $\bar S= (\bar S\setminus S_m) \cup (\bigcup_{i=1}^m D_i)$. 
	The unique solution to equations \eqref{uno} and \eqref{duo}  is therefore
	\begin{align*}
	\gamma=&~\min\big\{~u-\mathbbm{1}_{\{D_1=\emptyset \}},~|\bar{S}|~\big\}\\[.1cm]
	\beta_i=&\left\{	\begin{array}{ll}
	\min\big\{~u-\mathbbm{1}_{\{D_{i}=\emptyset \}},~|\bar{S}\setminus S_{i-1}|~\big\}
	-~\min\big\{~u-\mathbbm{1}_{\{D_{i+1}=\emptyset \}},~|\bar{S}\setminus S_{i}|~\big\}
	&\text{for } 1\le i\le m-1,\\%i\in\{1,\ldots, m-1\}\\
	\min\big\{~u-\mathbbm{1}_{\{D_{m}=\emptyset \}},~|\bar{S}\setminus S_{m-1}|~\big\}
	-~\min\big\{~u,~|\bar{S}\setminus S_m|~\big\}&\text{for }  i = m.\end{array}		\right.
	\end{align*}
	We now consider 3 cases:
	
	\noindent{\bf Case 1: }	$|\bar{S}\setminus S_m|\geq u$. 
	In this case, $|\bar{S}\setminus S_i|\geq u$ also holds for all $i\in I$ and %the unique solution to equations \eqref{uno} and \eqref{duo} is:
	\begin{align*}
	\gamma=&~u-\mathbbm{1}_{\{D_1=\emptyset \}},\\
	\beta_i=&\left\{	\begin{array}{ll}
	~~\mathbbm{1}_{\{D_{i+1}=\emptyset \}}-\mathbbm{1}_{\{D_{i}=\emptyset \}},& \quad i\in\{1,\ldots,m-1\},\\
	-\mathbbm{1}_{\{D_m=\emptyset \}}, &\quad i=m.	\end{array}		\right.
	\end{align*}
	Therefore,
\begin{align*}
\alpha^Tz+\beta^T\delta=&\sum_{j\in D_1}z_j-\mathbbm{1}_{\{D_1=\emptyset \}}(\underbrace{\delta_1}_{\geq 1-\sum_{j\in S_1}z_j})+
\sum_{i=1}^{m-1}\Big[\sum_{j\in D_{i+1}}z_j+\mathbbm{1}_{\{D_{i+1}=\emptyset \}}(\underbrace{\delta_i-\delta_{i+1}}_{\leq\sum_{j\in S_{i+1}\setminus S_i}z_j})\Big]+\sum_{j\in \bar{S}\setminus S_m}z_j
\\%	\end{align*}
%	Then by valid inequalities $(\ref{mc2}),(\ref{mc3}),(\ref{mc5})$ and $(\ref{mc7})$,
%	\begin{align*}
%	\alpha^Tz+\beta^T\delta
\leq&\underbrace{\sum_{j\in D_1}z_j-\mathbbm{1}_{\{D_1=\emptyset \}}\Big(1-\sum_{j\in S_1}z_j\Big)}_{\leq\sum_{j\in S_1}z_j-\mathbbm{1}_{\{D_1=\emptyset \}}}+
\sum_{i=1}^{m-1}\Big[\underbrace{\sum_{j\in D_{i+1}}z_j+\mathbbm{1}_{\{D_{i+1}=\emptyset \}}\Big(\sum_{j\in S_{i+1}\setminus S_i}z_j\Big)}_{\leq \sum_{j\in S_{i+1}\setminus S_i}z_j}\Big]+\sum_{j\in \bar{S}\setminus S_m}z_j\\
\leq&\sum_{j\in J}z_j-\mathbbm{1}_{\{D_1=\emptyset \}}~\leq~ u-\mathbbm{1}_{\{D_1=\emptyset \}}~=~\gamma.
\end{align*}
In the first inequality above, we use inequalities  $(\ref{mc2})$ and $(\ref{mc5})$ and in the second inequality we use the fact that 
if $\mathbbm{1}_{\{D_i=\emptyset \}}=1$, then $\sum_{j\in D_i}z_j =0$ for all $i\in I$.
%$$\sum_{j\in D_i}z_j + \mathbbm{1}_{\{D_i=\emptyset \}}(\sum_{j\in S_i}z_j)\le \sum_{j\in S_i}z_j.$$
	
	Therefore, inequalities (\ref{mc3})-(\ref{mc8}) imply $\alpha^Tz+\beta^T\delta\leq\gamma$.
	%i.e., ($I_\alpha$) is implied by $(\ref{mc2}),(\ref{mc3}),(\ref{mc5})$ and $(\ref{mc7})$.
	
	\noindent{\bf Case 2a: }$|\bar{S}\setminus S_m|\leq u-1$ and $|\bar{S}|\leq u-1$. 
	In this case, equations \eqref{uno} and \eqref{duo} imply
	$$	\gamma=|\bar{S}|,~~\text{and}~~	\beta_i=|\bar{S}\setminus S_{i-1}|-|\bar{S}\setminus S_i|=|D_{i}|, \quad i\in I.$$
	In this case, using inequalities (\ref{mc1}) and (\ref{mc8}), we can write
	%\begin{align*}
	%	\alpha^Tz+\beta^T\delta=~&\sum_{j\in D_1}(z_j+\delta_1)+\sum_{i=1}^{m-1}\Big[\sum_{j\in D_{i+1}}(z_j+\delta_{i+1})\Big]+\sum_{j\in \bar{S}\setminus S_m}z_j.\\
	%\leq~&|D_1|+\sum_{i=1}^{m-1}|D_{i+1}|+| \bar{S}\setminus S_m| ~=~|\bar{S}| ~=~\gamma,
	%\end{align*}
	\begin{align*}
	\alpha^Tz+\beta^T\delta=~&\sum_{i=1}^{m}\Big[\sum_{j\in D_{i}}(z_j+\delta_{i})\Big]+\sum_{j\in \bar{S}\setminus S_m}z_j
	\leq~\sum_{i=1}^{m}|D_{i}|+| \bar{S}\setminus S_m| ~=~|\bar{S}| ~=~\gamma.
	\end{align*}
	Therefore, all points that satisfy equations (\ref{mc3})-(\ref{mc8}) also satisfy  $\alpha^Tz+\beta^T\delta\leq\gamma$.
	%i.e., $(I_\alpha)$ is implied by (\ref{mc1}) and (\ref{mc8}).
	
	\noindent{\bf Case 2b: }	 $|\bar{S}\setminus S_m|\leq u-1$ and $|\bar{S}|\geq u$. 
	Let $h:=\min\{i\in I:|\bar{S}\setminus S_i|\leq u-1 \}$. In this case,
	\begin{align*}
	\gamma&=~u-\mathbbm{1}_{\{D_1=\emptyset \}},\\[.2cm]
	\beta_i&~=\left\{	\begin{array}{ll}
	\mathbbm{1}_{\{D_{i+1}=\emptyset \}}-\mathbbm{1}_{\{D_{i}=\emptyset \}}, &i\in\{1,\ldots,h-1\},\\
	u-\mathbbm{1}_{\{D_h=\emptyset \}}-|\bar{S}\setminus S_{h}|,& i=h,\\
	|\bar{S}\setminus S_{i-1}|-|\bar{S}\setminus S_i|~=~|D_{i}|,~~~~~~&i\in\{h+1,\ldots,m\}.
	\end{array}		\right.
	\end{align*}
	Therefore,
	\begin{align*}
	\alpha^Tz+\beta^T\delta=&\sum_{j\in D_1}z_j-\mathbbm{1}_{\{D_1=\emptyset\}}(\underbrace{\delta_1}_{\geq 1-\sum_{j\in S_1}z_j})+
	\sum_{i=1}^{h-1}\Big[\sum_{j\in D_{i+1}}z_j+\mathbbm{1}_{\{D_{i+1}=\emptyset \}}(\underbrace{\delta_i-\delta_{i+1}}_{\leq \sum_{j\in S_{i+1}\setminus S_i}z_j})\Big]\\
	&+\Big[\sum_{j\in D_{h+1}}z_j+(u-|\bar{S}\setminus S_{h}|)\delta_{h}\Big]+
	\sum_{j=h+1}^m\Big[\sum_{j\in D_{i+1}}z_j+|D_{i}|\delta_i\Big]+\sum_{j\in \bar{S}\setminus S_m}z_j.
	\\
	\le~&\underbrace{\sum_{j\in D_1}z_j-\mathbbm{1}_{\{D_1=\emptyset\}}(1-\sum_{j\in S_1}z_j)}_{\leq\sum_{j\in S_1}z_j-\mathbbm{1}_{\{D_1=\emptyset\}}}+
	\sum_{i=1}^{h-1}\Big[\underbrace{\sum_{j\in D_{i+1}}z_j+\mathbbm{1}_{\{D_{i+1}=\emptyset \}}\sum_{j\in S_{i+1}\setminus S_i}z_j}_{\leq\sum_{j\in S_{i+1}\setminus S_i}z_j}\Big]
	\\
	&+\Big[\sum_{j\in \bar{S}_{h+1}\setminus S_{h}}z_j+(u-|\bar{S}\setminus S_{h}|)\delta_{h}\Big]+
	\sum_{j=h+1}^m\Big[\sum_{j\in D_{i+1}}z_j+|D_{i}|\delta_i\Big]+\sum_{j\in \bar{S}\setminus S_m}z_j
	\\
	\leq~&\sum_{j\in\bar{S}\cup S_{h}}z_j+\big(u-|\bar{S}\setminus S_{h}|\big) \delta_{h}+ \sum_{i=h+1}^m |D_{i}|\delta_i-\mathbbm{1}_{\{D_1=\emptyset\}}
	\\
	\leq~& u-\mathbbm{1}_{\{D_1=\emptyset\}}=~\gamma	
	\end{align*}
	where the last  inequality folows from the mixing inequality (\ref{nestedmix}) with $S'=\bar{S}\cup S_{h}$ and $p=h$.
%	Then by valid inequalities (\ref{mc2}), (\ref{mc5}), (\ref{mc7}) and inequality (\ref{nestedmix}) with $S'=\bar{S}\cup S_{h}$ and $p=h$
%	\begin{align*}
%\alpha^Tz+\beta^T\delta~\leq~&\sum_{j\in D_1}z_j-\mathbbm{1}_{\{D_1=\emptyset\}}(1-\sum_{j\in S_1}z_j)+
%\sum_{i=1}^{h-1}\Big[\sum_{j\in D_{i+1}}z_j+\mathbbm{1}_{\{D_{i+1}=\emptyset \}}\sum_{j\in S_{i+1}\setminus S_i}z_j\Big]\\
%&+\Big[\sum_{j\in \bar{S}_{h+1}\setminus S_{h}}z_j+(u-|\bar{S}\setminus S_{h}|)\delta_{h}\Big]+
%\sum_{j=h+1}^m\Big[\sum_{j\in D_{i+1}}z_j+|D_{i}|\delta_i\Big]+\sum_{j\in \bar{S}\setminus S_m}z_j\\
%\leq~&-\mathbbm{1}_{\{D_1=\emptyset\}}+\sum_{j\in\bar{S}\cup S_{h}}z_j+\big(u-|\bar{S}\setminus S_{h}|\big) \delta_{h}+ \sum_{i=h+1}^m |D_{i}|\delta_i\\
%\leq~& u-\mathbbm{1}_{\{D_1=\emptyset\}}=~\gamma.	\\[-1.6cm]
%\end{align*}
\end{proof}

%In Section \ref{nestedconv}, we found inequalities (\ref{nestedmix}) together with the Mccormick's relaxation defines $\conv(\XJD)$. In this section, we show that inequalities (\ref{nestedmix}) can be obtained by mixing \cite{gunluk2001mixing} at most $m$ mixed-integer rounding (MIR) inequalities, which implies that each of inequalities (\ref{nestedmix}) has MIR rank at most $m$ \cite{dash2009mixing}.

\subsection{Convex hull description of $\XJd$}\label{sec:XJd}
In \cite{fischer2018matroid}, the authors study the convex hull description of the following set:
\begin{equation}
\Big\{(x,\delta)\in\{0,1\}^{|J|+|I|}:\delta_i=\prod_{j\in S_i}x_j~~\text{ for }i\in I,~~x\in P_{\mathcal{M}} \Big\}\label{McC_set}
\end{equation}
where $\{S_i\}_{i\in I}$ is a family of nested subsets of a given set $J$ and $P_{\mathcal{M}}$ is the convex hull of incidence vectors associated with independent sets  $\mathcal{U}$ of the matroid ${\mathcal{M}}=(J,\mathcal{U})$ defined on the ground set $J$. 
Note that if we let $\mathcal{U}$ be the set of all subsets of $J$ with cardinality at most  $k$ for some $k\in\Z_+$, the constraint $x\in P_{\mathcal{M}}$ simply becomes $\sum_{j\in J}x_j\le k$.
Consequently, using this  matroid in \eqref{McC_set} leads to a set very similar to the one we have been studying. 
More precisely, taking $k=n-l$ to define the independent sets and replacing $x_j$ with $(1-z_j)$ for $j\in J$,  gives the set $\XJd$. 
Note that due to the complementation of the $x$ variables in \eqref{McC_set}, the upper bound on the sum of the $x$ variables becomes a lower bound on the sum of the $z$ variables. 

Using the particular matroid described above, we next translate the results from \cite{fischer2018matroid} to our context. 
Remember that we use $S_0=\emptyset$ and  $S_{m+1}=J$ for convenience. 
\begin{thm}[Fischer, Fischer and McCormick \cite{fischer2018matroid}]\label{thm:ln}
	Inequalities (\ref{mc3})-(\ref{mc8}) together with 
	\begin{eqnarray}
%	-\sum_{j\in S'}z_j+(|S'\cup S_p|-n+l)\delta_p+\sum_{i=p+1}^m(|S'\cup S_i|-|S'\cup S_{i-1}|)\delta_i~\leq~0,\label{nestedmix2}
	\sum_{j\in S'}z_j~\ge~(|S'\cup S_p|-n+l)\delta_p+\sum_{i=p+1}^m(|S'\cup S_i|-|S'\cup S_{i-1}|)\delta_i,\label{nestedmix2}
	\end{eqnarray}
	for all $p\in I$ and  $S'\subset J$ that satisfy $|S'\cup S_{p-1}|\leq n-l<|S'\cup S_p|$ give a complete description of $\conv(\XJd)$.
\end{thm}

Notice that similar to  inequalities (\ref{nestedmix}), inequalities (\ref{nestedmix2}) above are also defined for subsets of $J$ and both (\ref{nestedmix}) and (\ref{nestedmix2})  have the term $\sum_{j\in S'}z_j$ as well as a telescopic sum involving the $\delta$ variables. 
We next show that (\ref{nestedmix2}) can also be derived using the mixing procedure.
Let $S'\subset J$  be fixed and let $M>n$ be a given constant. 
For any $i\in I$, 
% using the fact that for all $j\in S_i$ we have $z_j,\delta_i\le1$, and $1-\delta_i-z_j\ge 0$. Based on McCormick relaxation (\ref{xmc0})-(\ref{xmc7}),
the following (base) inequality is valid for $\XJd$:
\begin{align*}
\frac{1}{M}\sum_{j\in S'}z_j+(1-\delta_i)&\geq\frac{1}{M}\Big(\sum_{j\in S'\cup S_i}z_j-\sum_{j\in S_i\setminus S'}z_j\Big)+\frac{|S_i\setminus S'|}{M}(1-\delta_i)\\
&=\frac{1}{M}\Big(\underbrace{\sum_{j\in J}z_j}_{\geq l}-\underbrace{\sum_{j\in J\setminus(S'\cup S_i)}z_j}_{\leq |J\setminus(S'\cup S_i)|}\Big)+\frac{1}{M}\sum_{j\in S_i\setminus S'}(\underbrace{1-z_j-\delta_i}_{\geq 0})\\
&\geq \frac{1}{M}(l-(n-|S'\cup S_i|))+0=\frac{|S'\cup S_i|-n+l}{M}.
\end{align*}
When $|S'\cup S_p|\geq n-l+1$, the right-hand side of the inequality is strictly between 0 and 1, and  treating the term $\frac1M({\sum_{j\in S'}z_j})$  as a nonnegative continuous variable and  $(1-\delta_i)$ as an integer variable, we can apply the type I mixing procedure to the base inequalities for $i=p,p+1,\ldots,m$  to obtain \begin{displaymath}
\frac{1}{M}\sum_{j\in S'}z_j+\frac{|S'\cup S_p|-n+l}{M}(1-\delta_p)+\sum_{i=p+1}^m\Big(\frac{|S'\cup S_i|-|S'\cup S_{i-1}|}{M}\Big)(1-\delta_i)~\ge~ \frac{|S'\cup S_m|-n+l}{M},
\end{displaymath}
which can be rewritten as 
\begin{displaymath}
\frac{1}{M}\sum_{j\in S'}z_j~-~\frac{|S'\cup S_p|-n+l}{M}\delta_p~-~\sum_{i=p+1}^m\Big(\frac{|S'\cup S_i|-|S'\cup S_{i-1}|}{M}\Big)\delta_i~\geq~ 0.
\end{displaymath}
Multiplying both sides by $M$ and rearranging the terms gives inequality (\ref{nestedmix2}).
In Section \ref{sec:facets} we will discuss the conditions under which these inequalities are facet-defining for $\XJd$ and in Section \ref{general_mixsect} we will generalize these inequalities to the case when the sets in $\cS$ are not necessarily nested.

We next present our main result:
 \begin{thm}\label{thm:bang}
 		Let  $\cS=\{S_i\}_{i\in I}$ be a family of nested sets and assume that $u \geq 2$ and $l \leq n - |S_{|I|}|$.		
 		Then, $\conv(\XJDd)$ is defined by 
 		\begin{enumerate}[(i)]\setlength{\itemsep}{0pt}\setlength{\parskip}{0pt}
 		\item inequalities (\ref{mc3})-(\ref{mc8}),~
 		\item inequalities \eqref{nestedmix}  for all $p\in I$ and $S'\subseteq J$ such that  $|S'\setminus S_p|\leq u-1$, and,
 		\item inequalities \eqref{nestedmix2} for all $p\in I$ and  $S'\subset J$ that satisfy $|S'\cup S_{p-1}|\leq n-l<|S'\cup S_p|$.\\[-1cm]
 		\end{enumerate}
 	
 	Moreover, given a point $(z,\delta)\not\in\conv(\XJDd)$,  a (most) violated inequality can be found in polynomial time.
 	
\end{thm}
\begin{proof}
	Combining Theorems \ref{thm:decomp}, \ref{thm:0u} and \ref{thm:ln} it follows that $\conv(\XJDd)$ is given by inequalities (\ref{mc3})-(\ref{mc8}) together with inequalities \eqref{nestedmix}  and  \eqref{nestedmix2}.
	
	For the second part of the proof, note that there are a polynomial number of inequalities of the form  (\ref{mc3})-(\ref{mc8}) and there are an exponential number of mixing inequalities as one can write one for each $S'\subset J$ and $p\in I$. 
	However, for each $S'\subseteq J$ and fixed $p\in I$ we can rewrite inequality (\ref{nestedmix}) as\begin{equation}\label{mixreform}
	\sum_{j\in S'\cap S_p}z_j+u\delta_p+\sum_{i=p+1}^m\sum_{j\in S'\cap(S_i\setminus S_{i-1})}(z_j+\delta_i-\delta_p)+\sum_{j\in S'\setminus S_m}(z_j-\delta_p)\leq u.
	\end{equation}
	Given a fractional solution $(\hat{z},\hat{\delta})$, let 
	$$\pi_j=\left\{\begin{array}{lll}	
			\hat z_j, 	& \text{for } j\in S_p,\\
			\hat z_j+\hat\delta_i-\hat\delta_p,	&\text{for } j\in S_i\setminus S_{i-1},~i=p+1,\ldots,m,\\ 	
			\hat z_j-\hat\delta_p,~~~~~~&\text{for } j\in J\setminus S_m.	\end{array}	\right.$$
	Then the left-hand side of \eqref{nestedmix} is maximized by
	\begin{align*}
	S_p^*= \arg\max_{Q\subseteq J}\Big[\sum_{j\in Q}\pi_j\::\: |Q\setminus S_p|\leq u-1\Big]
	\end{align*}
	which can be computed greedily by selecting $j\in J$ with the largest positive $\pi_j$ values while satisfying the cardinality constraint.
%	\begin{align*}
%S_p^*:=&\{j\in S_p:\hat{z}_j>0\}\cup\Big(\bigcup_{i=p+1}^{m}\{j\in S_i\setminus S_{i-1}:\hat{z}_j-\hat{\delta}_p+\hat{\delta}_i>0 \}\Big)\cup\{j\in J\setminus S_m:\hat{z}_j-\hat{\delta}_p>0 \}\\
%\in&\arg\max_{S'\supseteq S_p}\Big[\sum_{j\in S'}\hat{z}_j+(u-|S'\setminus S_p|)\hat{\delta}_p+\sum_{i=p+1}^m(|S'\setminus S_{i-1}|-|S'\setminus S_i|)\hat{\delta}_i\Big].
%\end{align*}
	Therefore, to separate from inequalities \eqref{nestedmix}, one  only needs to check $S'=S_p^*$ for all $p\in I$. Similarly, inequalities \eqref{nestedmix2} can be rewritten as\begin{equation}\label{mixreform2}
	-\sum_{S'\cap S_p}z_j+(|S_p|-n+l)\delta_p+\sum_{i=p+1}^m(|S_i\setminus S_{i-1}|)\delta_i+\sum_{i=p+1}^m\sum_{j\in S'\cap(S_i\setminus S_{i-1})}(\delta_p-\delta_i-z_j)+\sum_{j\in S'\setminus S_m}(\delta_p-z_j)\leq 0.
	\end{equation}
	Given a fraction solution $(\hat{z},\hat{\delta})$, we now define,
	\begin{equation}
	\sigma_j=
	\left\{\begin{array}{lll} - \hat z_j, &\text{for } j\in S_p,\\
	  \hat\delta_p-\hat\delta_i-\hat z_j,~~~ &\text{for } j\in S_i\setminus S_{i-1}, ~i=p+1,\ldots,m,\\
	 \hat \delta_p-\hat z_j, &\text{for }j\in J\setminus S_m.
		\end{array}	\right.
	\end{equation}
	Then the left-hand side of \eqref{mixreform2} is maximized by
	\begin{align*}
	S_p^{**}= \arg\max_{Q\subseteq J}\Big\{\sum_{j\in Q}\sigma_j: |Q\cup S_{p-1}|\leq n-l<|Q\cup S_p|\Big\},
	\end{align*}
which can again be computed greedily by ordering the indices $j\in J$ according to the $\sigma_j$ values.	
Alternatively, one can solve the LP $$\max \Big\{\sum_{j\in J}\sigma_jx_j : \sum_{j\in J\setminus S_{p-1} }x_j \leq n-l-|S_{p-1}|,
~~ \sum_{j\in J\setminus S_{p} }x_j \geq n-l+1-|S_p|,~~{\bf1}\ge x\ge{\bf0}\Big\},$$
which has a totally unimodular constraint matrix.
%	the left-hand side of \eqref{mixreform2} is maximized by\begin{align*}
%	S_p^{**}:=&\{j\in S_p:\hat{z}_j<0 \}\cup\Big(\bigcup_{i=p+1}^{m}\{j\in S_i\setminus S_{i-1}:\hat{\delta}_p-\hat{\delta}_i-\hat{z}_j>0 \}\Big)\cup\{j\in J\setminus S_m:\hat{\delta}_p-\hat{z}_j>0 \}\\
%	\in&\arg\max_{S'\supseteq S_p}\Big[-\sum_{j\in S'}z_j+(|S'\cup S_p|-n+l)\delta_p+\sum_{i=p+1}^m(|S'\cup S_i|-|S'\cup S_{i-1}|)\delta_i\Big].
%	\end{align*}
Consequently, one only needs to check $S'=S^{**}_p$ for all $p\in I$ to separate from inequalities \eqref{nestedmix2}.
\end{proof}

\begin{rmk} \label{remark}
		For the sake of completeness, we now consider the case when  $\cS$ is nested but $u\ge 2$ or $l\le n-|S_{|I|}|$ does not hold. 
		If $u=0$, then $\XJDd$ and its convex hull contains a single point. If $u=1$, then $\delta_i=1-\sum_{j\in S_i}z_j$ for all $i\in I$. These equations, together with
		$\sum_{j\in J}z_j\leq1$, and $1\ge z_j\ge0$ for $j\in J$ %\eqref{mc3}, \eqref{mc7} and \eqref{mc8}, 
		give the convex hull description of $\XJDd$. For the case when $l>n-|S_{|I|}|$, consider $\cS'=\{S_i\}_{i\in I:|S_i|\leq n-l}$. In this case we have $\delta_i=0$ for all $i$ with $|S_i|> n-l$, and the multilinear set associated with $\cS'$ falls into the discussion of Theorem \ref{thm:bang}.
\end{rmk}

In \cite{fischer2018matroid}, the authors show the separation of inequalities \eqref{nestedmix2} can be solved in polynomial time by solving a submodular minimization problem.

\newcommand{\inlineitem}[1][]{%
	\ifnum\enit@type=\tw@
	{\descriptionlabel{#1}}
	\hspace{\labelsep}%
	\else
	\ifnum\enit@type=\z@
	\refstepcounter{\@listctr}\fi
	\quad\@itemlabel\hspace{\labelsep}%
	\fi}

%\section{On Some Facets of $\conv(\XJD)$ and $\conv(\XJd)$ for Nested $\cS$}\label{sec:facets}
\section{Properties of facet-defining inequalities for the nested case}\label{sec:facets}
So far we have presented an inequality description of $\XJDd$ for the nested case using the description of $\conv(\XJD)$ developed in Section \ref{nested_mixsect} and the  description of $\conv(\XJd)$ presented earlier in \cite{fischer2018matroid}.
Not all inequalities in these exponential-size  descriptions are facet-defining and in this section we present  necessary  and sufficient conditions for inequality of the form (\ref{nestedmix}) or (\ref{nestedmix2}) to be facet-defining.
\begin{thm}\label{UBfacetthm1}
	Let $\cS$ be nested and let  $p\in I$ and $S'\subseteq J$ be such that $|S'\setminus S_p|\leq u-1$. 
	Then, without loss of generality, the following conditions are necessary for the associated inequality (\ref{nestedmix}) to define a facet of $\conv(\XJDd)$:
	\begin{enumerate}[~~~~~U1.]\setlength{\itemsep}{0pt}\setlength{\parskip}{0pt} %}\itemjoin={,\quad}
		\item $S'\supseteq S_p$,\qquad\label{U1}
		%		\item $|S'\setminus S_p|\leq u-1$;\label{U2}
		\item $|S'\setminus S_{p-1}|\geq u$ if $p\geq 2$,\label{U3}
	 	\item $|S'|\geq u+1$.\label{U4}
	\end{enumerate}
\end{thm}
%			\begin{enumerate}[~~~~~U1.]\item~~~~~ \begin{inparaenum}
%			\item[U1.] $S'\supseteq S_p$,\qquad\label{U1}
%			\item[U2.] $|S'\setminus S_{p-1}|\geq u$ if $p\geq 2$,\qquad\label{U3}
%			\item[U3.] $|S'|\geq u+1$.\label{U4}
%		\end{inparaenum}	\end{enumerate}

\begin{proof}
	%	Without loss of generality, we assume $u-|S'\setminus S_p|\neq 0$ since otherwise either $p=m$ and inequality (\ref{nestedmix}) is dominated by (\ref{mc1}), or $p\leq m-1$ and we can rewrite (\ref{nestedmix}) as\begin{displaymath}
	%	\sum_{j\in S'}z_j+\Big[u-|S'\setminus S_{p+1}|\Big]\delta_{p+1}+\sum_{i=p+2}^m\Big[|S'\setminus S_{i-1}|-|S'\setminus S_i|\Big]\delta_i\leq u
	%	\end{displaymath}
	%	and consider $p+1$ instead of $p$.
	%	
	%	**Not sure if we should define $p$ with the property that $u-|S'\setminus S_{p+1}|\neq 0$ in (\ref{nestedmix})**
	
	If condition U\ref{U1} is not satisfied, then replacing $S'$ with $S'\cup S_p$  in inequality (\ref{nestedmix}) leads to a stronger inequality as $z_j\ge 0$ for all $j\in J$. 
	%	If condition U\ref{U2} is not satisfied, then inequality (\ref{nestedmix}) is implied by inequality (\ref{nestedmix}) with $p$ replaced by $p+1$ and inequality (\ref{mc4}) with $i=p$ if $p\leq m-1$, or dominated by inequality (\ref{mc1}) if $p=m$. 
	Similarly, if condition U\ref{U3} is not satisfied,  then replacing $p$ with $p-1$  in inequality (\ref{nestedmix}) leads to a stronger inequality as $\delta_p\le\delta_{p-1}.$
	
	If condition U\ref{U4} is not satisfied, then $|S'|\leq u$ and\begin{align*}
	&\hskip-2cm\sum_{j\in S'}z_j+ \Big(u-|S'\setminus S_p|\Big)\delta_p+\sum_{i=p+1}^m\Big(|S'\setminus S_{i-1}|-|S'\setminus S_i|\Big)\delta_i\\
	=&\sum_{j\in S'\cap S_p}(z_j+\delta_p)
	+\sum_{i=p+1}^m\sum_{j\in S'\cap (S_i\setminus S_{i-1})}(z_j+\delta_i)+\sum_{j\in S'\setminus S_m}z_j+\Big(\underbrace{u-|S'|}_{\geq 0}\Big)\delta_p\\
	\leq &~~|S'\cap S_p|+\sum_{i=p+1}^m|S'\cap (S_i\setminus S_{i-1})|+|S'\setminus S_m|+(u-|S'|)~=~u,
	\end{align*}
	where the last inequality is implied by the fact that $z_j\le1$ for all $j\in J$ and $z_j+\delta_i\leq 1$ for all $j\in S_i$, $i\in I$. 
	Therefore, if condition U\ref{U4} is not satisfied, then inequality (\ref{nestedmix}) is implied by other valid inequalities. 
	As $\conv(\XJDd)$ is full-dimensional, we conclude that conditions U\ref{U1}-U\ref{U4} are necessary for inequality (\ref{nestedmix}) to define a facet.
\end{proof}

\begin{thm}\label{UBfacetthm2}
	Let $\cS$ be nested and let  $p\in I$ and $S'\subseteq J$ be such that $|S'\setminus S_p|\leq u-1$. 
	If $p<m$ or $|S_m|<n-l$, then conditions U\ref{U1}-U\ref{U4} together with\begin{enumerate}[~~~~~U4.]\setlength{\itemsep}{0pt}\setlength{\parskip}{0pt}
		\item $S'\cap (S_{p+1}\setminus S_p)\neq\emptyset$ if $p\leq m-1$
		%\item[U5.] $|S'\setminus S_p|\geq l$
	\end{enumerate}
	are sufficient for inequality (\ref{nestedmix}) to define a facet of $\conv(\XJDd)$.
\end{thm}
\begin{proof}(sketch)   
%	Assume  that $S'\subseteq J$ and $p\in I$ satisfy the conditions above.
%	As $S'\supseteq S_p\supseteq S_1$, we can assume $S'=\{s_1,s_2,s_3,s_4,\ldots,s_{|S'|} \}$ where $s_1=1$, $s_2=2$ and $2<s_3<s_4<\ldots<s_{|S'|}$. We first show that the inequality\begin{equation}\label{lift0a}
%	\sum_{j\in S'}z_j+(u-|S'\setminus S_p|)\delta_p\leq u
%	\end{equation}
%	defines a $(n+p-1)$-dimensional face of $\conv(\XJDd)\cap\{(z,\delta):\delta_i=0,i\in\{p+1,\ldots,m\} \}$. Let $Q=\{s_{|S'|-u+1},\ldots,s_{|S'|}\}$. Then $Q$ satisfies $|Q|=u\geq 2$ and $S'\setminus S_p\subset Q\subseteq S'\setminus S_{p-1}$. Note that $p<m$ or $|S_m|<n-l$ implies that $|J\setminus S_p|\geq l+1$. Let $R$ be a set satisfying $S'\setminus S_p\subseteq R\subseteq J\setminus S_p$ and $|R|=\max\{l+1,|S'\setminus S_p|\}\in[l+1,u]$. Define\begin{displaymath}
%	R'=\left\{\begin{array}{ll}
%	R, & \text{if }l+1\leq|S'\setminus S_p|(\leq u-1)\text{, i.e., }R=S'\setminus S_p;\\
%	R\setminus\{j_0\}\text{ for some fixed }j_0\in R\setminus S', & \text{if }l+1>|S'\setminus S_p|\text{, i.e., }R\setminus S'\neq \emptyset.
%	\end{array}
%	\right.
%	\end{displaymath}
%	Then $R'$ satisfies $R'\supseteq S'\setminus S_p$ and $l\leq|R'|\leq u-1$. Let $T$ be a set satisfying $|T|=u-1$ and $T\subseteq S'\setminus S_{p-1}$.
		Assume  that $S'\subseteq J$ and $p\in I$ satisfy the conditions above.
	As $S'\supseteq S_p\supseteq S_1$, we can assume $S'=\{s_1,s_2,s_3,s_4,\ldots,s_{|S'|} \}$ where $s_1=1$, $s_2=2$ and $2<s_3<s_4<\ldots<s_{|S'|}$. We first show that the inequality\begin{equation}\label{lift0a}
	\sum_{j\in S'}z_j+(u-|S'\setminus S_p|)\delta_p\leq u
	\end{equation}
	defines an $(n+p-1)$-dimensional face of $\conv(\XJDd)\cap\{(z,\delta):\delta_i=0,i\in\{p+1,\ldots,m\} \}$. Let $Q=\{s_{|S'|-u+1},\ldots,s_{|S'|}\}$. Then $Q$ satisfies $|Q|=u\geq 2$ and $S'\setminus S_p\subset Q\subseteq S'\setminus S_{p-1}$. Note that $p<m$ or $|S_m|<n-l$ implies that $|J\setminus S_p|\geq l+1$. Let $R$ be a set satisfying $S'\setminus S_p\subseteq R\subseteq J\setminus S_p$ and $|R|=\max\{l+1,|S'\setminus S_p|\}\in[l+1,u]$. Define\begin{displaymath}
	R'=\left\{\begin{array}{ll}
	R, & \text{if }l+1\leq|S'\setminus S_p|(\leq u-1)\text{, i.e., }R=S'\setminus S_p,\\
	R\setminus\{j_0\}\text{ for some fixed }j_0\in R\setminus S', & \text{if }l+1>|S'\setminus S_p|\text{, i.e., }R\setminus S'\neq \emptyset.
	\end{array}
	\right.
	\end{displaymath}
	Then $R'$ satisfies $R'\supseteq S'\setminus S_p$ and $l\leq|R'|\leq u-1$. Let $T$ be a set satisfying $|T|=u-1$ and $T\subseteq S'\setminus S_{p-1}$.
	
	Consider points (using Definition \ref{def:pts}) associated with the sets 
	$$	Q,~~(Q\cup\{1\})\setminus\{j\} ~~\text{for}~ j\in Q,~~(Q\setminus\{s_{|S'|-1},s_{|S'|} \})\cup\{1,j\} ~~\text{for}~ j\in S'\setminus Q\setminus\{1\},$$
	$$ R,~~~R\setminus\{j\}~~\text{for}~ j\in R\setminus S',~~~R'\cup\{j\} ~~\text{for}~j\in J\setminus R\setminus S_p,$$
	$$ (Q\setminus\{s_{|S'|} \})\cup\{2\}~~\text{if}~p\geq 2,~~~T\cup\{k_i\} ~~\text{for}~ {i\in\{2,\ldots,p-1 \}}.$$
%	
%	$$Q,~~(Q\cup\{1\})\setminus\{j\} ~~\text{for}~ j\in Q,~~(Q\setminus\{s_{|S'|-1},s_{|S'|} \})\cup\{1,j\} ~~\text{for}~ j\in S'\setminus Q\setminus\{1\},$$
%	$$R,~~~R\setminus\{j\}~~\text{for}~ j\in R\setminus S',~~~R'\cup\{j\} ~~\text{for}~j\in J\setminus R\setminus S_p,$$
%	$$(Q\setminus\{s_{|S'|} \})\cup\{2\}~~\text{if}~p\geq 2,~~~T\cup\{k_i\} ~~\text{for}~ {i\in\{2,\ldots,p-1 \}}.$$
 Note that some of the index sets used for defining the sets above can be empty, in which case the associated points are not considered. These $(n+p)$ points are feasible and satisfy $\delta_i=0$ for $i\in\{p+1,\ldots,m\}$ (as $(S'\setminus S_p)\cap S_{p+1}\neq\emptyset$ by assumption U4), and lie on the hyperplane\begin{displaymath}
	\sum_{j\in S'}z_j+(u-|S'\setminus S_p|)\delta_p=u.
	\end{displaymath}
	
	In the rest of the proof (presented in Appendix), we first show that these points are affinely independent and therefore inequality (\ref{lift0a}) defines an $(n+p-1)$-dimensional face of $\conv(\XJDd)\cap\{(z,\delta):\delta_i=0,i\in\{p+1,\ldots,m\} \}$.
	We then lift the coefficients of $\delta_{p+1},\ldots,\delta_{m}$ to conclude that inequality (\ref{nestedmix}) is facet-defining.
	
	\end{proof}
\begin{rmk}
	For the case when $p=m$ and $|S_m|=n-l$, inequality (\ref{nestedmix}) is facet-defining for $\conv(\XJDd)$ if and only if $S'=J$.
\end{rmk}

Results similar to Theorem \ref{UBfacetthm1} and \ref{UBfacetthm2} hold for valid inequalities (\ref{nestedmix2}) of the polytope $\conv(\XJd)$. 
In \cite{fischer2018matroid}, the conditions that $|S'\cup S_{p-1}|\leq n-l$ and $n-l+1\le|S'\cup S_p|$ are implicitly imposed on the choice of $p$ based on the rank function associated with the matroid. 
%Other inequalities of the form (\ref{nestedmix2}) not satisfying this condition are redundant. 
We next present a stronger  characterization of the necessary conditions for these inequalities to be facet-defining.
% of the form (\ref{nestedmix2}).
\begin{thm}\label{Thm:nestedmix2-facet}
	Let $\cS$ be nested and let $p\in I$ and $S'\subset J$ be such that $|S'\cup S_{p-1}|\leq n-l<|S'\cup S_p|$. 
	Then the following conditions are necessary for inequality (\ref{nestedmix2}) to define a facet of $\conv(\XJDd)$:
	\begin{enumerate}[~~~~~L1.]\setlength{\itemsep}{0pt}\setlength{\parskip}{0pt}
		\item $S'\cap S_p=\emptyset$;\label{L1}
		%	\item $|S'\cup S_p|\geq n-l+1$;\label{L2}
		%	\item $|S'\cup S_{p-1}|\leq n-l$ if $p\geq 2$;\label{L3}
		\item $|S'|\leq n-l-1$.\label{L4}
	\end{enumerate}
\end{thm}
\begin{proof}
	%Conditions L\ref{L2} and L\ref{L3} are already enforced in (\ref{nestedmix2}). 
	If condition L\ref{L1} is not satisfied, then replacing $S'$ with $S'\setminus S_p$  in inequality (\ref{nestedmix2}) leads to a stronger inequality as $z_j\ge 0$ for all $j\in J$. 
	%	If condition L\ref{L1} is not satisfied, then (\ref{nestedmix2}) is implied by (\ref{nestedmix2}) with $S'$ replaced by $S'\setminus S_p$. 
	If condition L\ref{L4} is not satisfied, then $|S'|\geq n-l$. By valid inequalities (\ref{mc1}), (\ref{mc8}), $\delta_p\leq 1$ and $\sum_{j\in J}z_j\geq l$,\begin{align*}
	&\hskip -2.5cm(|S'\cup S_p|-n+l)\delta_p+\sum_{i=p+1}^m\big(|S'\cup S_i|-|S'\cup S_{i-1}|\big)\delta_i -\sum_{j\in S'}z_j\\
	=&~~\big(\underbrace{|S'|-(n-l)}_{\geq 0}\big)\delta_p+(|S_p\setminus S'|)\delta_p+\sum_{i=p+1}^m\big(|S_i\setminus S_{i-1}\setminus S'|\big)\delta_i-\sum_{j\in S'}z_j\\
	\leq&~~|S'|-(n-l)+\sum_{j\in S_p\setminus S'}(1-z_j)+\sum_{i=p+1}^m\sum_{j\in S_i\setminus S_{i-1}\setminus S'}(1-z_j) -\sum_{j\in S'}z_j\\
	=&~~|S'\cup S_m|-(n-l)-\sum_{j\in S'\cup S_m}z_j\\
	=&\sum_{j\in J\setminus (S'\cup S_m)}z_j+|S'\cup S_m|-(n-l)-\sum_{j\in J}z_j\\
	\leq&~~|J\setminus (S'\cup S_m)|+|S'\cup S_m|-(n-l)-l\\
	=&~~ 0,
	\end{align*}
	where the first inequality is implied by the fact that $\delta_p\le1$ and $z_j+\delta_i\leq 1$, for all $j\in S_i$, $i\in I$ and the second inequality is implied by the fact that  $\sum_{j\in J}z_j\geq l$, $z_j\le1$ for all $j\in J$.
	Therefore, if condition L\ref{L4} is not satisfied, then inequality (\ref{nestedmix2}) is implied by other valid inequalities. 
	As $\conv(\XJDd)$ is full-dimensional, we conclude that conditions L\ref{L1} and L\ref{L4} are necessary for inequality (\ref{nestedmix2}) to define a facet.
	%
	%	i.e., inequality (\ref{nestedmix2}) is implied by valid inequalities (\ref{mc2}), (\ref{mc8}), $\delta_p\leq 1$ and $\sum_{j\in J}z_j\geq l$. Since $\conv(\XJd)$ is full-dimensional, conditions L\ref{L1}-L\ref{L4} are necessary conditions for inequality (\ref{nestedmix2}) to define a facet of $\conv(\XJd)$.
\end{proof}
In \cite[Proposition 23]{fischer2018matroid}, the authors describe three conditions for inequality (\ref{nestedmix2}) to be facet-defining for $\conv(\XJd)$. These conditions involve the rank function of the underlying matroid which, when translated to our context, has rank function 
$$r(S)=\min\{|S|,n-l\}$$ for each subset $S$ of the ground set $J$.  
More precisely, these conditions are
\begin{enumerate}[~~~~~C1.]
	\item \label{C1} Inequality $\sum_{j\in S'}x_j\leq r(S')$ is facet-defining for the set $\conv\{x\in\{0,1\}^{|J|}:\sum_{j\in J}x_j\leq n-l \}$;
	\item \label{C2} Set $S'$ is closed \cite[Definition 4]{fischer2018matroid} and non-separable \cite[Definition 22]{fischer2018matroid}, meaning
	\begin{enumerate}[~~~~~C\ref{C2}a.] 
		\item \label{C2a} $r(S')<r(S'\cup\{j\})$ for all $j\in J\setminus S'$,
		\item \label{C2b} $r(S')<r(S^a)+r(S^b)$ for all nonempty $S^a\subset S'$ and $S^b=S'\setminus S^a$;
	\end{enumerate}
	\item \label{C3}  For all $i\in I$, $\delta_i$ has a strictly positive coefficient in  (\ref{nestedmix2}), i.e.,
	$${\rm C\ref{C3}a.}~p=1,~~{\rm C\ref{C3}b.}~|S'\cup S_1|>n-l~\text{ and }~~{\rm C\ref{C3}c.}~ |S'\cup S_{i-1}|<|S'\cup S_i|~~ \text{for all }i\in\{2,\ldots,m\}.$$
\end{enumerate}
Notice that conditions C\ref{C2a} and C\ref{C2b} cannot hold simultaneously unless $S'$ is equal to the set $J$, or it contains a single element, i.e., $S'=\{j\}$ for some $j\in J$.
Also note that condition C\ref{C1} is satisfied in both cases, i.e. when $S'=J$ or $|S'|=1$.
However, remember that Theorem \ref{Thm:nestedmix2-facet} requires $S'\subset J$ and therefore $S'\not=J$.
Therefore, the only remaining possible choices for $S'$ are  $S'=\{j\}$ for some $j\in J$. 
Finally, condition C\ref{C3}b together with our starting assumption that $|S_m|\leq n-l$ implies that $m=p=1$, $|S_1|=n-l$ and $j\in J\setminus S_1$. 
In conclusion, we observe that conditions C\ref{C1}-C\ref{C3} are  satisfied only in the narrow case when the family $\cS$ defining  $\conv(\XJd)$ contains  a single set $S$ of cardinality $n-l$. In addition,  the set $S'$ must have cardinality one, containing a single element $j\in J\setminus S$.

In the next theorem, we give significantly less restrictive  conditions for inequality (\ref{nestedmix2}) to be facet-defining for $\conv(\XJDd)$.
\begin{thm}\label{Thm:L3}
	Let $\cS$ be nested and let $p\in I$ and $S'\subset J$ be such that $|S'\cup S_{p-1}|\leq n-l<|S'\cup S_p|$. 
	If $p<m$ or $|S_m|<n-l$, then conditions L\ref{L1}-L\ref{L4} together with\begin{enumerate}[~~~~~L3.]\setlength{\itemsep}{0pt}\setlength{\parskip}{0pt}
		\item \label{L3} $S_{p+1}\setminus S_p\nsubseteq S'$ if $p\leq m-1$
	\end{enumerate}
	are sufficient for inequality (\ref{nestedmix2}) to define a facet of $\conv(\XJDd)$.
\end{thm}
\begin{proof}(sketch) 	
	Assume  that $S'\subseteq J$ and $p\in I$ satisfy the conditions above. Then the assumption $p<m$ or $|S_m|<n-l$ implies that $|S_p|<n-l$, and $S'\neq\emptyset$ as $|S'\cup S_p|> n-l$. Assume $S'=\{s_1,\ldots,s_{|S'|}  \}$ with $s_1<\ldots<s_{|S'|}$. We first show that the inequality\begin{equation}\label{lift1a}
	-\sum_{j\in S'}z_j+(|S'\cup S_p|-n+l)\delta_p\leq 0
	\end{equation}
	defines an $(n+p-1)$-dimensional face of $\conv(\XJDd)\cap\{(z,\delta):\delta_i=0,i\in\{p+1,\ldots,m\} \}$.
	
	Let $Q=S_p\cup\{s_1,s_2,\ldots,s_{n-l-|S_p|} \}$. Then $Q$ satisfies $|Q|=n-l$ and $S_p\subset Q\subset S_p\cup S'$. Let $R=(J\setminus S'\setminus S_p)\cup\underbrace{\{1,2,\ldots,l+1-|J\setminus S'\setminus S_p|\}}_{\subseteq S_p}$. Then $R$ satisfies $|R|=l+1\leq u$ and $(J\setminus S'\setminus S_p)\cup\{1,2\}\subseteq R\subseteq J\setminus S'$ as $n-l<|S'\cup S_p|$. Note that $|S'\cup S_{p-1}|\leq n-l$. For $i\in\{1,\ldots,p-1\}$, we let $T_i$ denote the first $l$ elements of $J\setminus S'\setminus S_i$.
	
	Consider the points (using Definition \ref{def:pts}) associated with the sets
	$$	R\setminus \{j\} ~\text{for}~  {j\in R},~~R,~~(R\setminus\{1\})\cup \{j\} ~\text{for}~{j\in S_p\setminus R},~~	\{T_i\}_{i\in\{1,\ldots,p-1\}},$$
	$$(J\setminus Q\setminus\{s_{|S'|}\})\cup\{j\}  ~\text{for}~{j\in Q\setminus S_p},~~(J\setminus Q\setminus\{j\})\cup\{s_1\}  ~\text{for}~ {j\in S'\setminus Q\setminus\{s_{|S'|} \}},~~J\setminus Q.$$
	
%	Assume  that $S'\subseteq J$ and $p\in I$ satisfy the conditions above. Then the assumption $p<m$ or $|S_m|<n-l$ implies that $|S_p|<n-l$, and $S'\neq\emptyset$ as $|S'\cup S_p|> n-l$. Assume $S'=\{s_1,\ldots,s_{|S'|}  \}$ with $s_1<\ldots<s_{|S'|}$. We first show that the inequality\begin{equation}\label{lift1a}
%	-\sum_{j\in S'}z_j+(|S'\cup S_p|-n+l)\delta_p\leq 0
%	\end{equation}
%	defines a $(n+p-1)$-dimensional face of $\conv(\XJDd)\cap\{(z,\delta):\delta_i=0,i\in\{p+1,\ldots,m\} \}$.
%	
%	Let $Q=S_p\cup\{s_1,s_2,\ldots,s_{n-l-|S_p|} \}$. Then $Q$ satisfies $|Q|=n-l$ and $S_p\subset Q\subset S_p\cup S'$. Let $R=(J\setminus S'\setminus S_p)\cup\{1,2,\ldots,l+1-|J\setminus S'\setminus S_p|\}$. Then $R$ satisfies $|R|=l+1\leq u$ and $(J\setminus S'\setminus S_p)\cup\{1,2\}\subseteq R\subseteq J\setminus S'$ as $n-l<|S'\cup S_p|$. Note that $|S'\cup S_{p-1}|\leq n-l$. For $i\in\{1,\ldots,p-1\}$, we can let $T_i$ denote the first $l$ elements of $J\setminus S'\setminus S_i$.
%	
%	Consider the points (using Definition \ref{def:pts}) associated with the following sets 
%	$$R\setminus \{j\} ~~\text{for}~ j\in R,~~~R,~~~(R\setminus\{1\})\cup \{j\} ~~\text{for}~ {j\in S_p\setminus R},~~~(J\setminus Q\setminus\{s_{|S'|}\})\cup\{j\}  ~~\text{for}~ {j\in Q\setminus S_p},$$
%	$$(J\setminus Q\setminus\{j\})\cup\{s_1\}  ~~\text{for}~ {j\in S'\setminus Q\setminus\{s_{|S'|} \}},~~~J\setminus Q,~~~T_i ~~\text{for}~ {i\in\{1,\ldots,p-1\}}.$$ 
	
	These $(n+p)$ points are feasible with $\delta_i=0,i\in\{p+1,\ldots,m \}$ (as $(J\setminus S'\setminus S_p)\cap S_{p+1}\neq\emptyset$ by assumption L3), and lie on the hyperplane\begin{displaymath}
	-\sum_{j\in S'}z_j+(|S'\cup S_p|-n+l)\delta_p=0.
	\end{displaymath}

	In the rest of the proof (presented in Appendix), we first show that these points are affinely independent and therefore inequality (\ref{lift1a}) defines an $(n+p-1)$-dimensional face of $\conv(\XJDd)\cap\{(z,\delta):\delta_i=0,i\in\{p+1,\ldots,m\} \}$.
	We then lift the coefficients of $\delta_{p+1},\ldots,\delta_{m}$ to conclude that inequality (\ref{nestedmix2}) is facet-defining.
\end{proof}	

\begin{rmk}
	For the case when $p=m$ and $|S_m|=n-l$, inequality (\ref{nestedmix2}) is facet-defining for $\conv(\XJDd)$ if and only if $S'=\{j\}$ for some $j\in J\setminus S_m$.
\end{rmk}

\section{Valid inequalities when $\cS$ is not nested  }\label{general_mixsect}
%\subsection{Generalization of (\ref{nestedmix})}\label{Ubmixsubsect}
%========================================================================
In Section \ref{nestedsec}, we described inequalities (\ref{nestedmix}) and (\ref{nestedmix2}) and showed that together with the standard linearization and 2-link inequalities they define $\conv(\XJD)$ and $\conv(\XJd)$, respectively. In this section, we extend these inequalities to the general case when the sets in $\cS$ are not necessarily nested.
%In the general case,  the standard linearization of $\conv(\XJDd)$ is obtained by replacing inequality \eqref{mc2}  in \eqref{mc3}-\eqref{mc8}  with the inequality
%\begin{equation}  1-\delta_i-\sum_{j\in S_i}z_j\leq 0,\label{eq:gen}\end{equation}
%for $i\in I$.

Notice that since we derived inequalities (\ref{nestedmix}) using the mixing procedure, they are still valid for $\conv(\XJDd)$ 
%even when $\cS$ is not a nested family %$S_1\subset S_2\subset\ldots\subset S_m$ is not satisfied 
in the general case, provided that
\begin{equation}
u-1\geq|S'\setminus S_p|\geq|S'\setminus S_{p+1}|\geq\ldots\geq |S'\setminus S_m|\label{qaz} \end{equation} hold.
We next generalize inequalities (\ref{nestedmix}) to the case when \eqref{qaz} is not satisfied.

\begin{prop}\label{gen1}
	Assume that sets $S_{[1]},S_{[2]},\ldots,S_{[t]}\in\cS$ are distinct and let $\delta_{[i]}$ denote the $\delta$ variable associated with $S_{[i]}$.
	For $S'\subseteq J$, the following inequality is valid for $\conv(\XJDd)$ 
	\begin{equation}\label{generalization}
	\sum_{j\in S'}z_j+\Big(u-|S'\setminus S_{[1]}|\Big)\delta_{[1]}
	~+~\sum_{i=2}^t\Big(\Big|S'\cap S_{[i]}\setminus \bigcup_{k=1}^{i-1}S_{[k]}\Big|\Big)\delta_{[i]}~\leq~ u,
	\end{equation}
	provided that $\max_{i=2,\ldots,t}|S'\setminus(S_{[1]}\cap S_{[i]})|\leq u$.
\end{prop}
\def \mmm {t}
\begin{proof}
	As the indices of the sets in $\cS$ are arbitrary, we assume that $S_{[i]} = S_i$ for $i=1,\ldots,t$,  without loss of generality.
	%and $S'\subseteq J$ satisfies $\max_{i=1,\ldots,\mmm}|S'\setminus(S_1\cap S_i)|\leq u$. 
	First note that the following inequality 
	\begin{equation}\label{qwer}
	\sum_{j\in S'}z_j+(u-|S'\setminus S_1|)\delta_1\leq u
	\end{equation}
	is valid for  $\conv(\XJDd)$ as it is implied by $\sum_{j\in S'}z_j\leq u$	when $\delta_1=0$ and $\sum_{j\in S'}z_j\leq |S'\setminus S_1|$ when $\delta_1=1$ (and therefore $z_j=0$ for all $j\in S_1$).
	We will derive inequality (\ref{generalization}) by sequential lifting, starting with inequality \eqref{qwer}
	and showing that if (\ref{generalization}) with $\mmm$ replaced by $\mmm'$ is valid for $\XJDd$
	%for $\mmm=\mmm'$ 
	for $t'\in\{1,2,\ldots,\mmm-1\}$, then 
	\begin{equation}\label{UBlift}
	\sum_{j\in S'}z_j+\Big(u-|S'\setminus S_1|\Big)\delta_1+\sum_{i=2}^{\mmm'}\Big(\Big|S'\cap S_i\setminus\bigcup_{k=1}^{i-1}S_k\Big|\Big)\delta_i\leq u-\Big|S'\cap S_{\mmm'+1}\setminus \bigcup_{k=1}^{\mmm'}S_k\Big|
	\end{equation}
	holds for all $(z,\delta)\in \XJDd$ with $\delta_{\mmm'+1}=1$. 
	This would imply that (\ref{generalization}) with $\mmm$ replaced by $\mmm'+1$ is also valid.
	Fix  $\mmm'\in\{1,\ldots,\mmm-1\}$ and note that for all $(z,\delta)\in\XJDd$ with  $\delta_{\mmm'+1}=1$, we have  $z_j=0$ for  $j\in S_{\mmm'+1}$. Therefore, given any arbitrary $(z,\delta)\in\XJDd$ with $\delta_{\mmm'+1}=1$, we have\begin{align*}
	&\hskip -1cm 
	\sum_{j\in S'}z_j~+~(u-|S'\setminus S_1|)\delta_1~+~\sum_{i=2}^{\mmm'}\Big(\Big|S'\cap S_i\setminus\bigcup_{k=1}^{i-1}S_k\Big|\Big)\delta_i\\
	%	=&\sum_{j\in S'\setminus S_{\mmm'+1}}z_j~+~(u-|S'\setminus S_1|)\delta_1
	%	~+~\sum_{i=2}^{\mmm'}(|S'\cap S_i\setminus\bigcup_{k=1}^{i-1}S_k|)\delta_i\\
	=&\underbrace{\sum_{j\in S'\cap S_{\mmm'+1}}z_j}_{=\,0}~+~\sum_{j\in S'\setminus (\bigcup_{k=1}^{\mmm'+1}S_k)}z_j
	~+~\sum_{j\in S'\cap S_1\setminus S_{\mmm'+1}}(z_j+\delta_1)
	~+~(\underbrace{u-|S'\setminus (S_1\cap S_{\mmm'+1})|}_{\geq\, 0\text{ by assumption}})\delta_1\\
	&\hskip1cm ~+~\sum_{i=2}^{\mmm'}\Big[\sum_{j\in S'\cap S_i\setminus (\bigcup_{k=1}^{i-1}S_k)\setminus S_{\mmm'+1}}(z_j+\delta_i)~ +~ \Big(\Big|S'\cap S_i\cap S_{\mmm'+1}\setminus\bigcup_{k=1}^{i-1}S_k\Big|\Big)\delta_i\Big]\\
	%&\text{**I cannot follow the next step -- Oktay***}\\
	\leq&~~\Big|S'\setminus \bigcup_{k=1}^{\mmm'+1}S_k\Big|
	~+~u-|S'\setminus S_1|
	~+~\sum_{i=2}^{\mmm'}\Big|S'\cap S_i\setminus \bigcup_{k=1}^{i-1}S_k\Big|	\\
	=&~~\Big|S'\setminus \bigcup_{k=1}^{\mmm'+1}S_k\Big|~+~u-\Big|S'\setminus \bigcup_{i=1}^{\mmm'}S_i\Big|~~=~~~u-\Big|S'\cap S_{\mmm'+1}\setminus \bigcup_{k=1}^{\mmm'}S_k\Big|.
	\end{align*}
\end{proof}
We note that inequality (\ref{generalization}) reduces to (\ref{nestedmix}) when $\cS$ is nested by taking ${\mmm}=m-p+1$ and $S_{[i]}=S_{p+i-1}$ for $i=1,\ldots,{\mmm}$.

Similarly, as we have shown that inequalities (\ref{nestedmix2}) can also be derived via mixing, they are valid in the general case as long as $n-l+1\leq|S'\cup S_p|\leq |S'\cup S_{p+1}|\leq\ldots\leq |S'\cup S_m|$. We next extend (\ref{nestedmix2}) to a more general case.
\begin{prop}
	Assume that sets $S_{[1]},S_{[2]},\ldots,S_{[\mmm]}\in\cS$ are distinct and let $\delta_{[i]}$ denote the $\delta$ variable associated with $S_{[i]}$.
	For $S'\subseteq J$, the following inequality is valid for $\conv(\XJDd)$ 
	\begin{equation}\label{generalization2}
	-\sum_{j\in S'}z_j+\Big(|S'\cup S_{[1]}|-n+l\Big)\delta_{[1]}~+~\sum_{i=2}^\mmm\Big(\Big|S_{[i]}\setminus (\bigcup_{k=1}^{i-1}S_{[k]})\setminus S'\Big|\Big)\delta_{[i]}~\leq~ 0,
	\end{equation}
	provided that $\min_{i=2,\ldots,\mmm}|S'\cup(S_{[1]}\cap S_{[i]})|\geq n-l$.
\end{prop}
\begin{proof}
	Without loss of generality, we assume that $S_{[i]}=S_i$ for $i=1,\ldots,\mmm$. Note that the following inequality\begin{equation}\label{asdf}
	-\sum_{j\in S'}z_j+(|S'\cup S_{1}|-n+l)\delta_{1}\leq 0
	\end{equation}
	is valid for $\conv(\XJDd)$ as it is implied by $z_j\geq 0$ for all $j\in S_1$ when $\delta_1=0$ and $|S'\cup S_1|-\sum_{j\in S'}z_j=\sum_{j\in S'\cup S_1}(1-z_j)\leq n-l$ when $\delta_1=1$ (and therefore $z_j=0$ for all $j\in S_1$). We will derive (\ref{generalization2}) by sequential lifting, starting with inequality (\ref{asdf}) and showing that if (\ref{generalization2}) with $\mmm$ replaced by $\mmm'$ is valid for $\XJDd$ for $\mmm'\in\{1,2,\ldots,\mmm-1\}$, then\begin{equation}
	-\sum_{j\in S'}z_j+\Big(|S'\cup S_{1}|-n+l\Big)\delta_{1}~+~\sum_{i=2}^{\mmm'}\Big(\Big|S_{i}\setminus (\bigcup_{k=1}^{i-1}S_{k})\setminus S'\Big|\Big)\delta_{i}~\leq~ -\Big|S_{\mmm'+1}\setminus(\bigcup_{k=1}^{\mmm'}S_k)\setminus S' \Big|
	\end{equation}
	holds for all $(z,\delta)\in\XJDd$ with $\delta_{\mmm'+1}=1$. This would imply that (\ref{generalization2}) with $\mmm$ replaced by $\mmm'+1$ is also valid.
	Fix $\mmm'\in\{1,\ldots,\mmm-1 \}$ and note that for all $(z,\delta)\in \XJDd$ with $\delta_{\mmm'+1}=1$, we have $z_j=0$ for $j\in S_{\mmm'+1}$. Therefore, given any arbitrary $(z,\delta)\in\XJDd$ with $\delta_{\mmm'+1}=1$, we have\begin{align*}
	&\hskip -1cm
	-\sum_{j\in S'}z_j+\Big(|S'\cup S_{1}|-n+l\Big)\delta_{1}~+~\sum_{i=2}^{\mmm'}\Big(\Big|S_{i}\setminus (\bigcup_{k=1}^{i-1}S_{k})\setminus S'\Big|\Big)\delta_{i}\\
	=&\underbrace{\sum_{j\in S_{\mmm'+1}\setminus S'}z_j}_{=0}-\sum_{j\in S'\cup S_{\mmm'+1}}z_j+\Big[\Big(\underbrace{|S'\cup(S_1\cap S_{\mmm'+1})|-n+l}_{\geq 0\text{ by assumption}}\Big)\delta_1+(|S_1\setminus S'\setminus S_{\mmm'+1}|)\delta_1\Big]\\
	&+\sum_{i=2}^{\mmm'}\Big[\Big(\Big|S_i\cap S_{\mmm'+1}\setminus(\bigcup_{k=1}^{i-1}S_k)\setminus S'\Big|\Big)\delta_i+\Big(\Big|S_i\setminus(\bigcup_{k=1}^{i-1}S_k)\setminus S'\setminus S_{\mmm'+1}\Big|\Big)\delta_i\Big]\\
	\leq&-\sum_{j\in S'\cup S_{\mmm'+1}}z_j+\Big[\Big(|S'\cup(S_1\cap S_{\mmm'+1})|-n+l\Big)+\sum_{j\in S_1\setminus S'\setminus S_{\mmm'+1}}(1-z_j)\Big]\\
	&+\sum_{i=2}^{\mmm'}\Big[\Big(\Big|S_i\cap S_{\mmm'+1}\setminus(\bigcup_{k=1}^{i-1}S_k)\setminus S'\Big|\Big)+\sum_{j\in S_i\setminus(\bigcup_{k=1}^{i-1}S_k)\setminus S'\setminus S_{\mmm'+1}}(1-z_j)\Big]\\
	=&-\sum_{j\in S'\cup(\bigcup_{i=1}^{\mmm'+1}S_{i})}z_j+\Big|S'\cup(\bigcup_{i=1}^{\mmm'}S_{i})\Big|-n+l\\
	=&-\sum_{j\in J}z_j+\sum_{j\in J\setminus[S'\cup(\bigcup_{i=1}^{\mmm'+1}S_{i})]}z_j+\Big|S'\cup(\bigcup_{i=1}^{\mmm'}S_{i})\Big|-n+l\\
	\leq &-l+\Big[n-\Big|S'\cup(\bigcup_{i=1}^{\mmm'+1}S_{i})\Big|\Big]+\Big|S'\cup(\bigcup_{i=1}^{\mmm'}S_{i})\Big|-n+l\\
	=&-\Big|S_{\mmm'+1}\setminus(\bigcup_{k=1}^{\mmm'}S_k)\setminus S'\Big|.
	\end{align*}
\end{proof}
Inequality (\ref{generalization2}) reduces to (\ref{nestedmix2}) when $\cS$ is nested by taking ${\mmm}=m-p+1$ and $S_{[i]}=S_{p+i-1}$ for $i=1,\ldots,{\mmm}$.

\section{Conclusions}\label{sec:conclusions}
In this paper, we study the convex hull of the multilinear set with (two-sided) cardinality constraints and give a polyhedral characterization of it when the  sets involved have a nested structure.
%In this paper, we gave a polyhedral characterization of the convex hull of the multilinear set with (two sided) cardinality constraints when the  sets involved have a nested structure.
We first show that the convex hull can be obtained by intersecting the convex hulls of two simpler sets, each with one sided cardinality constraints. Convex hull of one of these sets ($\conv(\XJd)$) has already been characterized earlier in \cite{fischer2018matroid}. 
The description of the second set ($\conv(\XJD)$) is new. 
The two descriptions bear some resemblance due to the fact that the inequalities involved can be derived using the mixing procedure starting with different base inequalities. 
To the best of our knowledge, the similarity between the inequality descriptions of the two sets does not imply that one of the sets can be used (via a complementation) to obtain the other. The authors of \cite{fischer2018matroid} also agree with this assessment.

For the general (non-nested) case, we are able to derive a family of valid inequalities that generalize the inequalities for the nested case. 
Derivation of these inequalities do not involve  the mixing procedure.  
These inequalities do not necessarily yield the convex hull as the polyhedral structure of the general case seems to be significantly more complicated even when only two non-nested sets are involved \cite{chen2021multilinear}.

%We also gave a family of valid inequalities in the non-nested case, but these inequalities do not necessarily yield the convex hull.  
%In a separate study \cite{chen2021multilinear}, we have obtained a polyhedral characterization of this set when it is defined by only two non-nested sets, and observed that the polyhedral structure  is significantly more complicated than it is in the nested case. 

See also \cite{fischer2020matroid}, where  Fischer, Fischer and McCormick extend their earlier work on matroids by considering multilinear terms  defined by all subsets of a fixed subset of the ground set instead of nested sets. Note that all subsets of a set form a proper family by  Proposition~\ref{prop:prop}  provided that $\DJDd=\Delta^{0,n}$, which is one of the assumptions in  \cite{fischer2020matroid}.

\bibliographystyle{plain}%{ieeetr}
\bibliography{ref}

\newpage

\section*{Appendix}
In this section we present the full proofs of Lemma \ref{dimlem} and Theorems \ref{UBfacetthm2} and \ref{Thm:L3}.
\subsubsection*{Proof of Lemma \ref{dimlem}}
%\begin{lem}\label{dimlem}
%	The polytope $\conv(\XJDd)$ is full-dimensional.
%\end{lem}
\begin{proof}
	Given $z\in\{0,1\}^{n}$ satisfying $l\leq\sum_{j\in J}z_j\leq u$, there exists a unique $\delta$ such that $(z,\delta)\in\XJDd$. Therefore, given any $U\subseteq J$ with $l\leq|U|\leq u$, we can define the corresponding point $v^U$ as follows: %$=(z^U,\mydelta^U)\in\XJDd$, where %for each $j\in J$,
	\begin{displaymath}v^U=(z^U,\mydelta^U)\in\XJDd\text{~~~where~~~}
	z^U_j=\left\{\begin{array}{ll}
	1, & \text{if }j\in U,\\
	0, & \text{otherwise},
	\end{array}
	\right.
	%	\end{displaymath}	and for each $i\in I$,\begin{displaymath}
	~~~\text{and}~~~~~
	\mydelta^U_i=	\prod_{j\in S_i}(1-z^U_j).
	\end{displaymath}
	For $j=1,\ldots,n$, let $e^j\in\R^n$ denote the $j$-th unit vector in $\R^n$. % and $e^0\in\R^n$ be the zero vector.
	In addition, let $0_m\in\R^m$ denote the vector of all zeros, and for $i=1,\ldots,m$, let $d^i\in\R^m$ denote the vector whose first $i$ components are one and the rest are zero. 
	We now consider 2 cases:

	\noindent{\bf Case 1: } Assume $l=0$. In this case, we will argue that the following $m+n+1$ points in $\conv(\XJD)$ are  affinely independent: 
	$$ 	v^{\{2\}}	=\begin{bmatrix}	~e^2~	\\		~0_m~	\end{bmatrix},~~	
	v^{\{k_i+1\}}	=\begin{bmatrix}	~ e^{k_i+1}~\\	~~d^i~		\end{bmatrix}~~\text{for}~ i\in I,~~
	v^{\{1\}} 	=\begin{bmatrix}	~ e^1~	\\	0_m	\end{bmatrix},~~
	v^{\{1,j\}} 	=\begin{bmatrix}	~ e^1+e^j~	\\	~~0_m	\end{bmatrix}~~\text{for}~ j\in J\setminus\{1\}. %d^{q(j)}
	$$
	Clearly these points are in $\XJD$ and together they form the following matrix  $V\in \R^{(m+n)\times(m+n+1)}$:
	\begin{equation}
	V=\left[\begin{array}{c|c|c}
	~e^2~& ~~~\mathbbm{K}~~&\begin{array}{c} \mathbbm{1}^T_n\\\hline 0_{(n-1)}~ |~~~~\mathbbm{I}_{(n-1)}~~\end{array}\\\hline
	%	~\mathbb{O}_{n}~& ~~\mathbbm{E}~~&\mathbb{K}\\\hline
	0_m& \mathbb{D}&  \mathbb{O}_{m\times n}	\end{array}\right]
	\label{matrixV}
	\end{equation}
	where 	$\mathbbm{1}_{*}\in\R^{*}$ is a vector/matrix of all ones, $0_{*}\in\R^{*}$ is a vector of all zeros, and, $\mathbb{O}_{*}$ and $\mathbbm{I}_{*}$, respectively, denote the matrix of all zeros and the identity matrix of the specified dimension.
	The $i$-th column of the matrix $ \mathbbm{K} \in\R^{n\times m}$ is equal to $e^{k_i+1}$, and $i$-th column of  $ \mathbbm{D}\in \R^{m\times m}$ is $d^i$. 
	Note that $ \mathbbm{D}$ is an upper triangular matrix with all ones on and above the diagonal. 
	
	To show that the columns of $V$ are affinely independent, we need to argue that the unique solution to the  system of equations: 
	\begin{equation}	V\lambda=0,~~~~~\sum_{t=1}^{m+n+1}\lambda_t=0\label{eq:aff}\end{equation}
	is $\lambda=0$. 
	Note that the first row of  $ \mathbbm{K} $ is all zeros and therefore the first row of $V$ has $m+1$ consecutive zeros followed by $n$ ones.
	Therefore, the first row of  $V\lambda=0$ implies that $\sum_{t=m+2}^{m+n+1}\lambda_t=0$ and consequently $\sum_{t=1}^{m+1}\lambda_t=0$.
	
	As $ \mathbbm{D}$ is an upper triangular matrix of ones,  the last $m$ rows of $V\lambda=0$ imply that 
	$$0~=\sum_{t=2}^{m+1}\lambda_t=\sum_{t=3}^{m+1}\lambda_t=\ldots=\sum_{t=m+1}^{m+1}\lambda_t=0$$
	and therefore $\lambda_t=0$ for $t=2,\ldots,m+1$.
	Moreover,  $\sum_{t=1}^{m+1}\lambda_t=0$, implies that $\lambda_1=0$ as well.
	
	As the first $m+1$ components of $\lambda$ have to be zero,  the first $n$ rows of $V\lambda=0$ now imply that 
	$$\sum_{t=m+2}^{m+n+1}\lambda_t=0,~~ \text{ and } ~~\lambda_t=0~\text{ for } t=m+3,\ldots, m+n+1.$$
	Using the first equation, we have  $\lambda_{m+2}=0$ as well and the columns of $V$ are indeed affinely independent.

	\noindent{\bf Case 2: }Assume $l\geq 1$. In this case, we let $Q=\{n-l+1,\ldots,n\}\subseteq J\setminus S_m$, where $|Q|=l$.
	We now consider the following $m+n+1$ points in $\conv(\XJDd)$:
	$$ 	v^{\{2\}\cup Q}	,~~~	v^{\{k_i+1\}\cup Q}~~\text{for}~ i\in I,~~~
	v^{\{1\}\cup Q} ,~~~
	v^{\{1,j\}\cup Q\setminus\{n\}}~~\text{for}~ j\in J\setminus\{1\}\setminus Q,~~~
	v^{\{1\}\cup Q\setminus\{j\}}~~\text{for}~ j\in Q.
	$$
	These points form the matrix
	$$ V^l=\left[\begin{array}{c|c|c}
	~e^2+\sum_{i\in Q}e^i~& ~~~\mathbbm{K'}~~&
	\begin{array}{c} 
	\mathbbm{1}^T_n\\\hline
	\begin{array}{c|c|c}	 
	0_{(n-l-1)} &~~~~\mathbbm{I}_{(n-l-1)}~~&\mathbb{O}_{(n-l-1)\times l}\\\hline 
	\mathbbm{1}_{l}
	&\begin{matrix}\mathbbm{1}_{(l-1)\times( n-l-1)}\\\hline0^T_{(n-l-1)}\end{matrix}
	&\mathbbm{1}_{l\times l}-\mathbbm{I}_{l}
	\end{array}
	\end{array}\\\hline
	%	~\mathbb{O}_{n}~& ~~\mathbbm{E}~~&\mathbb{K}\\\hline
	0_m& \mathbb{D}&  \mathbb{O}_{m\times n}	\end{array}\right],
	$$
	where $\mathbbm{K'}$ is a matrix with all entries of its first row being zero and $\mathbbm{D}$ is the upper triangular matrix described in \eqref{matrixV}.
	
	As in Case 1, we first observe that  the first row of $V^l$ has $m+1$ consecutive zeros followed by $n$ ones and argue that  $\sum_{t=m+2}^{m+n+1}\lambda_t=0$ and  $\sum_{t=1}^{m+1}\lambda_t=0$.
	In addition, as the last $m$ rows of $V^l$ are the same as $V$, we also conclude that the first $m+1$ components of of $\lambda$ have to be zero.
	
	Finally,  note that the $n$ by $n$ matrix on the upper right corner of $V^l$ is nonsingular as 
	adding rows 2 to $n-l$ of this matrix to the last ($n$-th) row and then 
	subtracting its first row from each one of the last $l$ rows leads to the upper triangular matrix:
	$$\left[
	\begin{array}{c} 
	\mathbbm{1}^T_n\\\hline
	\begin{array}{c|c|c}	 
	0_{(n-l-1)} &~~~~\mathbbm{I}_{(n-l-1)}~~&\mathbb{O}_{(n-l-1)\times l}\\\hline 
	0_{l}&\mathbbm{O}_{l\times (n-l)}&-\mathbbm{I}_{l}
	\end{array}
	\end{array}\right].$$
	
	Therefore, we conclude that $\conv(\XJDd)$ is full-dimensional.
	%\vskip-1cm	
\end{proof}

%------------------------------------------------------------------------------
\subsubsection*{Proof of Theorem \ref{UBfacetthm2}}
%\begin{thm}\label{UBfacetthm2}
%	Let $\cS$ be nested and let  $p\in I$ and $S'\subseteq J$ be such that $|S'\setminus S_p|\leq u-1$. 
%	Then conditions U\ref{U1}-U\ref{U4} together with\begin{enumerate}
%		\item[U4.] $S'\cap (S_{p+1}\setminus S_p)\neq\emptyset$ if $p\leq m-1$
%		%\item[U5.] $|S'\setminus S_p|\geq l$
%	\end{enumerate}
%	are sufficient for inequality (\ref{nestedmix}) to define a facet of $\conv(\XJDd)$.
%\end{thm}
\begin{proof}
	Assume  that $S'\subseteq J$ and $p\in I$ satisfy the conditions above.
	As $S'\supseteq S_p\supseteq S_1$, we can assume $S'=\{s_1,s_2,s_3,s_4,\ldots,s_{|S'|} \}$ where $s_1=1$, $s_2=2$ and $2<s_3<s_4<\ldots<s_{|S'|}$. We first show that the inequality\begin{equation}\label{lift0}
	\sum_{j\in S'}z_j+(u-|S'\setminus S_p|)\delta_p\leq u
	\end{equation}
	defines an $(n+p-1)$-dimensional face of $\conv(\XJDd)\cap\{(z,\delta):\delta_i=0,i\in\{p+1,\ldots,m\} \}$. Let $Q=\{s_{|S'|-u+1},\ldots,s_{|S'|}\}$. Then $Q$ satisfies $|Q|=u\geq 2$ and $S'\setminus S_p\subset Q\subseteq S'\setminus S_{p-1}$. Note that $p<m$ or $|S_m|<n-l$ implies that $|J\setminus S_p|\geq l+1$. Let $R$ be a set satisfying $S'\setminus S_p\subseteq R\subseteq J\setminus S_p$ and $|R|=\max\{l+1,|S'\setminus S_p|\}\in[l+1,u]$. Define\begin{displaymath}
	R'=\left\{\begin{array}{ll}
	R, & \text{if }l+1\leq|S'\setminus S_p|(\leq u-1)\text{, i.e., }R=S'\setminus S_p;\\
	R\setminus\{j_0\}\text{ for some fixed }j_0\in R\setminus S', & \text{if }l+1>|S'\setminus S_p|\text{, i.e., }R\setminus S'\neq \emptyset.
	\end{array}
	\right.
	\end{displaymath}
	Then $R'$ satisfies $R'\supseteq S'\setminus S_p$ and $l\leq|R'|\leq u-1$. Let $T$ be a set satisfying $|T|=u-1$ and $T\subseteq S'\setminus S_{p-1}$.
	
	Consider points (using Definition \ref{def:pts}) associated with the sets\begin{gather}
	Q,~~(Q\cup\{1\})\setminus\{j\} ~~\text{for}~ j\in Q,~~(Q\setminus\{s_{|S'|-1},s_{|S'|} \})\cup\{1,j\} ~~\text{for}~ j\in S'\setminus Q\setminus\{1\},\label{Qsets}\\
	R,~~~R\setminus\{j\}~~\text{for}~ j\in R\setminus S',~~~R'\cup\{j\} ~~\text{for}~j\in J\setminus R\setminus S_p,\label{Rsets}\\
	(Q\setminus\{s_{|S'|} \})\cup\{2\}~~\text{if}~p\geq 2,~~~T\cup\{k_i\} ~~\text{for}~ {i\in\{2,\ldots,p-1 \}}.\label{pgeq2sets}
	\end{gather}
	Note that some of the index sets used for defining the sets in \eqref{pgeq2sets} can be empty, in which case the associated points are not considered but sets in \eqref{pgeq2sets}  would always contribute $p-1$ points in total. These $(n+p)$ points are feasible and satisfy $\delta_i=0$ for $i\in\{p+1,\ldots,m\}$ (as $(S'\setminus S_p)\cap S_{p+1}\neq\emptyset$ by assumption U4), and lie on the hyperplane $	\sum_{j\in S'}z_j+(u-|S'\setminus S_p|)\delta_p=u$ associated with inequality \eqref{lift0}.

	We next reorder the $(\{z_j\}_{j\in J},\delta_1,\ldots,\delta_p)$ coordinates of the points (see Figure \ref{fig:zjUB}) in the ordering\begin{displaymath}
	(\{z_j\}_{j\in S_p\setminus Q},\{z_j\}_{j\in Q},\{z_j\}_{j\in R\setminus Q},\{z_j\}_{j\in J\setminus(S_p\cup Q\cup R)},\delta_1,\ldots,\delta_p),
	\end{displaymath}
	and consider the matrix $V$ formed by these reordered coordinates of the (column) points. 
	
		\begin{figure}[H]
		\centering
		\begin{tikzpicture}[level/.style={},decoration={brace,mirror,amplitude=7}]
		\draw[|-|] (0,0) -- (2,0);
		\draw[|-|] (0,0) -- (3,0);
		\draw[|-|] (0,0) -- (5,0);
		\draw[|-|] (0,0) -- (8,0);
		\draw[dashed] (0,0) -- (0,-1);
		\draw[dashed] (3,0) -- (3,-1);
		\draw[dashed] (0,0) -- (0,1);
		\draw[dashed] (5,0) -- (5,1);
		\draw [decorate] (0,-0.1) --(2,-0.1) node[midway,below=3mm]{$S_{p-1}$};
		\draw [decorate] (0,-1) --(3,-1) node[midway,below=3mm]{$S_{p}$};
		\draw [decorate] (3,-0.1) --(6,-0.1) node[midway,below=3mm]{$R$};
		\draw [decorate] (5,1) -- (0,1) node[midway,above=3mm]{$S'$};
		\draw [decorate] (8,0.1) -- (5,0.1) node[midway,above=3mm]{$J\setminus S'$};
		\draw [decorate] (5,0.1) -- (2.5,0.1) node[midway,above=3mm]{$Q$};
		\end{tikzpicture}
		\caption{Reordered $z_j$ coordinates in the proof of Theorem \ref{UBfacetthm2}}
		\label{fig:zjUB}
	\end{figure}
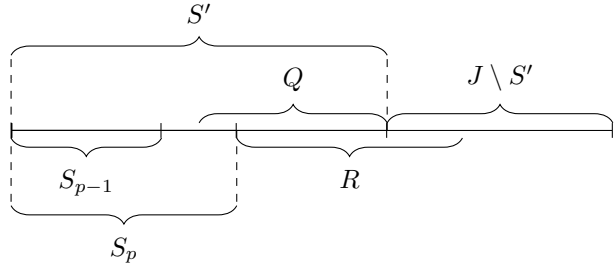
We would now argue that the unique solution to the system of equations\begin{equation}\label{UBVlameqn}
	V\lambda=0,~~\sum_{t=1}^{n+p}\lambda_t=0
	\end{equation}
	is $\lambda=0$.
	We separately consider two cases, namely $p\geq 2$ and $p=1$.
	
	First consider the case when $p\geq 2$. If this case, $|S'|\geq u+|S_{p-1}|\geq u+2$. Therefore, $\{1,2\}\cap Q=\emptyset$. We look at the matrix $V_p$ formed by the last $p$ rows of $V$ corresponding to the $\{\delta_i\}_{i=1}^p$ coordinates:\begin{displaymath}
	V_p=\left[\begin{array}{c|c|c|c|c|c|c}
	d^{p-1} & \mathbb{O}_{p\times (|S'|-1)} & \mathbbm{1}_{p\times(n+1-|S'|)} & 0_p & d^1 & \ldots & d^{p-2}
	\end{array}
	\right],
	\end{displaymath}
	where $d^i$ is defined in the proof of Lemma \ref{dimlem}. Equations $V_p\lambda=0$ imply $\lambda_1=0$ and $\lambda_{n+3}=\ldots=\lambda_{n+p}=0$. 
	\begin{figure}[H]
	\centering
	\begin{tikzpicture}[level/.style={},decoration={brace,mirror,amplitude=7}]
	\draw[dashed] (0,-1) -- (7,-1);
	\draw[dashed] (0,-1) -- (0,-2);
	\draw[dashed] (7,-1) -- (7,-2);
	\draw[dashed] (4.2,-1) -- (4.2,-2);
	\draw[dotted] (0.2,-2) -- (0.2,-3);
	\draw[dotted] (4.4,-2) -- (4.4,-3);
	\draw (0,-3) -- (7,-3);
	\draw (0,-2) -- (0,-3);
	\draw (0,-2) -- (7,-2);
	\draw (7,-3) -- (7,-2);
	\draw (0.2,-1) -- (4.4,-1);
	\draw (0.2,-1) -- (0.2,-2);
	\draw (4.4,-1) -- (4.4,-2);
	\draw [decorate] (2.1,-0.8) --(0,-0.8) node[midway,above=3mm]{Sets in \eqref{Qsets}};
	\draw [decorate] (4.2,-0.8) --(2.1,-0.8) node[midway,above=3mm]{Sets in \eqref{Rsets}};
	\draw [decorate] (7,-0.8) --(4.2,-0.8) node[midway,above=3mm]{Sets in \eqref{pgeq2sets}};
	\draw [decorate] (7.2,-2) --(7.2,-1) node[midway,right=3mm]{$z_j:j\in J$};
	\draw [decorate] (7.2,-3) --(7.2,-2) node[midway,right=3mm]{$\delta_i:i\in \{1,2,\ldots,p\}$};
	\node at (2.3,-1.7) {$\bar{V}\setminus V_p$};
	\node at (3.5,-2.5) {$V_p$};
	\node at (-1.2,-2) {$V~=~$};
	\end{tikzpicture}
%	\begin{tikzpicture}[level/.style={},decoration={brace,mirror,amplitude=7}]
%\draw[dashed] (0,0) -- (6,0);
%\draw[dashed] (0,0) -- (0,-2);
%\draw[dashed] (6,0) -- (6,-2);
%\draw[dashed] (4.2,0) -- (4.2,-2);
%\draw[dotted] (0.2,-2) -- (0.2,-3);
%\draw[dotted] (4.4,-2) -- (4.4,-3);
%\draw (0,-3) -- (6,-3);
%\draw (0,-2) -- (0,-3);
%\draw (0,-2) -- (6,-2);
%\draw (6,-3) -- (6,-2);
%\draw (0.2,0) -- (4.4,0);
%\draw (0.2,0) -- (0.2,-2);
%\draw (4.4,0) -- (4.4,-2);
%\draw [decorate] (2.1,0.2) --(0,0.2) node[midway,above=3mm]{Sets in \eqref{Qsets}};
%\draw [decorate] (4.2,0.2) --(2.1,0.2) node[midway,above=3mm]{Sets in \eqref{Rsets}};
%\draw [decorate] (6,0.2) --(4.2,0.2) node[midway,above=3mm]{Sets in \eqref{pgeq2sets}};
%\draw [decorate] (6.2,-2) --(6.2,0) node[midway,right=3mm]{$z_j:j\in J$};
%\draw [decorate] (6.2,-3) --(6.2,-2) node[midway,right=3mm]{$\delta_i:i\in \{1,2,\ldots,p\}$};
%\node at (2.3,-1) {$\bar{V}\setminus V_p$};
%\node at (3,-2.5) {$V_p$};
%\node at (-1.2,-1.5) {$V~=~$};
%\end{tikzpicture}
	\caption{Matrix $V$ in the proof of Theorem \ref{UBfacetthm2}}
\end{figure}
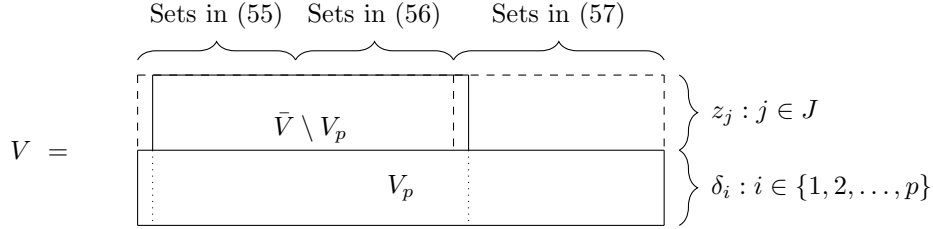
%
%
%	\begin{figure}[H]
%		\centering
%		\begin{tikzpicture}[level/.style={},decoration={brace,mirror,amplitude=7}]
%		\draw[dashed] (0,0) -- (6,0);
%		\draw[dashed] (0,0) -- (0,-6);
%		\draw[dashed] (6,0) -- (6,-6);
%		\draw[dashed] (4.2,0) -- (4.2,-4);
%		\draw[dotted] (0.2,-4) -- (0.2,-6);
%		\draw[dotted] (4.4,-4) -- (4.4,-6);
%		\draw (0,-6) -- (6,-6);
%		\draw (0,-4) -- (0,-6);
%		\draw (0,-4) -- (6,-4);
%		\draw (6,-4) -- (6,-6);
%		\draw (0.2,0) -- (4.4,0);
%		\draw (0.2,0) -- (0.2,-4);
%		\draw (4.4,0) -- (4.4,-4);
%		\draw [decorate] (2.1,0.2) --(0,0.2) node[midway,above=3mm]{Sets in \eqref{Qsets}};
%		\draw [decorate] (4.2,0.2) --(2.1,0.2) node[midway,above=3mm]{Sets in \eqref{Rsets}};
%		\draw [decorate] (6,0.2) --(4.2,0.2) node[midway,above=3mm]{Sets in \eqref{pgeq2sets}};
%		\draw [decorate] (6.2,-4) --(6.2,0) node[midway,right=3mm]{$z_j:j\in J$};
%		\draw [decorate] (6.2,-6) --(6.2,-4) node[midway,right=3mm]{$\delta_i:i\in \{1,2,\ldots,p\}$};
%		\node at (2.3,-2) {$\bar{V}\setminus V_p$};
%		\node at (3,-5) {$V_p$};
%		\end{tikzpicture}
%		\caption{Matrix $V$ in Theorem \ref{UBfacetthm2}}
%	\end{figure}
	Therefore, (\ref{UBVlameqn}) reduces to equations\begin{equation}\label{UBredeqn}
	\bar{V}\bar{\lambda}=0,~~\sum_{t=2}^{n+2}\lambda_t=0
	\end{equation}
	where $\bar{V}$ is a matrix formed by columns $2$ to $n+2$ of $V$ and $\bar{\lambda}=(\lambda_2,\ldots,\lambda_{n+2})^T$. Note that matrix $\bar{V}$ is of the form\begin{displaymath}
	\left[\begin{array}{c|c|c}
	\begin{array}{c}
	\mathbbm{1}^T_{|S'|-1}\\ \hline
	\begin{array}{c|c}
	\mathbb{O}_{(|S_p\setminus Q|-1)\times |Q|}
	& \mathbbm{I}_{(|S'\setminus Q|-1)}\\
	\hline
	\mathbbm{1}_{|Q|\times|Q|}-\mathbbm{I}_{|Q|}
	& \begin{array}{c}
	\mathbbm{1}_{(|Q|-2)\times(|S'\setminus Q|-1)}\\ \hline
	\mathbb{O}_{2\times(|S'\setminus Q|-1)}
	\end{array}
	\end{array}\\
	\hline
	\mathbb{O}_{(n-|S'|)\times(|S'|-1)}
	\end{array} &
	\begin{array}{c}
	\mathbb{O}_{|S_p|\times(n+1-|S'|)}\\ \hline
	\begin{array}{c|c}
	\begin{array}{c|c}
	\mathbbm{1}_{|R|} &\begin{array}{c}
	\\[-.3cm]\mathbbm{1}_{|S'\setminus S_p|\times|R\setminus S'|}\\[.3cm] \hline
	\mathbbm{1}_{|R\setminus S'|\times|R\setminus S'|}-\\
	\mathbbm{I}_{|R\setminus S'|}
	\end{array}
	\end{array} & *\\ \hline
	\mathbb{O}_{|J\setminus R\setminus S_p|\times(|R\setminus S'|+1)} & \mathbb{I}_{|J\setminus R\setminus S_p|}
	\end{array}
	\end{array} & \begin{array}{c}
	0\\ \hline
	1\\ \hline
	0_{|S'\setminus Q|-2}\\ \hline
	\mathbbm{1}_{|Q|-1}\\ \hline
	0_{n+1-|S'|}
	\end{array}
	\\
	\hline
	\mathbb{O}_{p\times (|S'|-1)} & \mathbbm{1}_{p\times(n+1-|S'|)} & 0_p
	\end{array}
	\right].
	\end{displaymath}
	By looking at the first and the last row of $\bar{V}\bar{\lambda}=0$ and $\sum_{t=2}^{n+2}\lambda_t=0$, we have $\lambda_{n+2}=0$.
	It is then easy to verify that $\bar{V}\bar{\lambda}=0$ and $\lambda_{n+2}=0$ imply $\lambda_{|Q|+2}=\ldots=\lambda_{|S'|}=0$ and $\lambda_{n-|J\setminus R\setminus S_p|+2}=\ldots=\lambda_{n+1}=0$ by looking at the $\{z_j\}_{j\in S'\setminus Q\setminus\{1\}}$ and $\{z_j\}_{j\in J\setminus R\setminus S_p}$ coordinates, respectively. The remaining columns of $\bar{V}$ are of the form:\begin{displaymath}
	\left[\begin{array}{c|c}
	\begin{array}{c}
	\mathbbm{1}^T_{|Q|}\\ \hline
	\mathbb{O}_{(|S_p\setminus Q|-1)\times |Q|}\\
	\hline
	\mathbbm{1}_{|Q|\times|Q|}-
	\mathbbm{I}_{|Q|}\\
	\hline
	\mathbb{O}_{(n-|S'|)\times|Q|}
	\end{array} &
	\begin{array}{c}
	\mathbb{O}_{|S_p|\times(|R\setminus S'|+1)}\\ \hline
	\begin{array}{c}
	\begin{array}{c|c}
	\mathbbm{1}_{|R|} &\begin{array}{c}
	\mathbbm{1}_{|S'\setminus S_p|\times|R\setminus S'|}\\ \hline
	\mathbbm{1}_{|R\setminus S'|\times|R\setminus S'|}-\\
	\mathbbm{I}_{|R\setminus S'|}
	\end{array}
	\end{array}
	\end{array}\\ \hline
	\mathbb{O}_{|J\setminus R\setminus S_p|\times(|R\setminus S'|+1)}
	\end{array}
	\\
	\hline
	\mathbb{O}_{p\times |Q|} & \mathbbm{1}_{p\times(|R\setminus S'|+1)}
	\end{array}
	\right].
	\end{displaymath}
	By looking at the $z_1,\{z_j\}_{j\in Q\cup R},\delta_p$ coordinates:\begin{displaymath}
	\left[
	\begin{array}{c|c}
	\mathbbm{1}^T_{|Q|} & 0^T_{|R\setminus S'|+1}\\ \hline
	\mathbbm{1}_{|Q|\times|Q|}-
	\mathbbm{I}_{|Q|} & \begin{array}{c}
	\mathbb{O}_{|Q\cap S_p|\times(|R\setminus S'|+1)}\\ \hline
	\mathbbm{1}_{|R\cap Q|\times(|R\setminus S'|+1)}
	\end{array}\\ \hline
	\mathbb{O}_{|R\setminus S'|\times|Q|} & \begin{array}{c|c}
	\mathbbm{1}_{|R\setminus S'|} & \begin{array}{c}
	\mathbbm{1}_{|R\setminus S'|\times|R\setminus S'|}
	-\mathbbm{I}_{|R\setminus S'|}
	\end{array}
	\end{array}\\ \hline
	0^T_{|Q|} & \mathbbm{1}^T_{|R\setminus S'|+1}
	\end{array}
	\right],
	\end{displaymath}
	we can finally conclude that the unique solution of \eqref{UBredeqn} is $\bar{\lambda}=0$ as these columns are linearly independent. 
	
	When $p=1$, sets defined in \eqref{pgeq2sets} would disappear and the matrix $V$ is of the form:
	\begin{displaymath}
	\left[\begin{array}{c|c}
	\begin{array}{c}
	\\[-.3cm]\mathbbm{1}^T_{|S'|}\\[.3cm] \hline
	\begin{array}{c|c}
	\mathbb{O}_{(|S_p\setminus Q|-1)\times (|Q|+1)}
	& \mathbbm{I}_{(|S'\setminus Q|-1)}\\
	\hline\begin{array}{c|c}
	\mathbbm{1}_{|Q|} & \mathbbm{1}_{|Q|\times|Q|}-\mathbbm{I}_{|Q|}
	\end{array}
	& \begin{array}{c}
	\mathbbm{1}_{(|Q|-2)\times(|S'\setminus Q|-1)}\\ \hline
	\mathbb{O}_{2\times(|S'\setminus Q|-1)}
	\end{array}
	\end{array}\\
	\hline
	\mathbb{O}_{(n-|S'|)\times|S'|}
	\end{array} &
	\begin{array}{c}
	\mathbb{O}_{|S_p|\times(n+1-|S'|)}\\ \hline
	\begin{array}{c|c}
	\begin{array}{c|c}
	\mathbbm{1}_{|R|} &\begin{array}{c}
	\mathbbm{1}_{|S'\setminus S_p|\times|R\setminus S'|}\\ \hline
	\mathbbm{1}_{|R\setminus S'|\times|R\setminus S'|}-\\
	\mathbbm{I}_{|R\setminus S'|}
	\end{array}
	\end{array} & *\\ \hline
	\mathbb{O}_{|J\setminus R\setminus S_p|\times(|R\setminus S'|+1)} & \mathbb{I}_{|J\setminus R\setminus S_p|}
	\end{array}
	\end{array}
	\\
	\hline
	0^T_{|S'|} & \mathbbm{1}^T_{n+1-|S'|}
	\end{array}
	\right].
	\end{displaymath}
	We get a matrix of the structure similar to the submatrix of $V$ formed from its first $n+1$ rows and first $n+1$ columns.	We can verify that in this case the columns of $V$ are affinely independent based on the discussion for the $p\geq 2$ case.
	
	Therefore, the given points are affinely independent and inequality (\ref{lift0}) defines an $(n+p-1)$-dimensional face of $\conv(\XJDd)\cap\{(z,\delta):\delta_i=0,i\in\{p+1,\ldots,m\} \}$.
	
	We finish the proof by lifting the coefficients of $\delta_{p+1},\ldots,\delta_{m}$. By validity of (\ref{nestedmix}), the following inequalities are valid:\begin{align*}
	|S'\setminus S_{m'}|-|S'\setminus S_{m'+1}|\leq u-\max\Big\{&\sum_{j\in S'}z_j+(u-|S'\setminus S_p|)\delta_p+\sum_{i=p+1}^{m'}(|S'\setminus S_{i-1}|-|S'\setminus S_i|)\delta_i:\\
	&(z,\delta)\in\XJDd,\delta_{m'+1}=1,\delta_{i}=0,i>m'+1\Big\},~~m'=p,\ldots,m-1.
	\end{align*}
	And the above inequalities hold at equality for the points of $(S'\setminus S_{m'+1})\cup Q_{m'}$ for $m'=p,\ldots,m-1$, respectively. Here $Q_{m'}=J\setminus S'\setminus S_{m'+1}$ if $|J\setminus S_{m'+1}|\leq u$. Otherwise, we construct $Q_{m'}\subset J\setminus S'\setminus S_{m'+1}$ (see Figure \ref{fig:QmUB}) such that\begin{enumerate}
		\item $|Q_{m'}|=\min\{u-|S'\setminus S_{m'+1}|,|J\setminus S'\setminus S_{m'+1}|\}$, (this implies $|(S'\setminus S_{m'+1})\cup Q_{m'}|\geq l$ as either $|(S'\setminus S_{m'+1})\cup Q_{m'}|=u$, or $|(S'\setminus S_{m'+1})\cup Q_{m'}|=|J\setminus S_{m'+1}|\geq|J\setminus S_m|\geq l$),
		\item $((S'\setminus S_{m'+1})\cup Q_{m'})\cap S_{m'+2}\neq \emptyset$ if $m'<m-1$.
	\end{enumerate}
	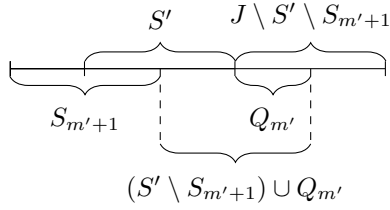
\begin{figure}[H]
		\centering
		\begin{tikzpicture}[level/.style={},decoration={brace,mirror,amplitude=7}]
		\draw[|-|] (0,0) -- (1,0);
		\draw[|-|] (0,0) -- (3,0);
		\draw[|-|] (0,0) -- (5,0);
		\draw[dashed] (2,0) -- (2,-1);
		\draw[dashed] (4,0) -- (4,-1);
		\draw [decorate] (0,-0.1) --(2,-0.1) node[midway,below=3mm]{$S_{m'+1}$};
		\draw [decorate] (2,-1) --(4,-1) node[midway,below=3mm]{$(S'\setminus S_{m'+1})\cup Q_{m'}$};
		\draw [decorate] (3,-0.1) --(4,-0.1) node[midway,below=3mm]{$Q_{m'}$};
		\draw [decorate] (3,0.1) -- (1,0.1) node[midway,above=3mm]{$S'$};
		\draw [decorate] (5,0.1) -- (3,0.1) node[midway,above=3mm]{$J\setminus S'\setminus S_{m'+1}$};
		\end{tikzpicture}
		\caption{Construction of $Q_{m'}$ in the proof of Theorem \ref{UBfacetthm2}}
		\label{fig:QmUB}
	\end{figure}
	In total, we find $(n+m)$ affinely independent points lying on the hyperplane\begin{displaymath}
	\sum_{j\in S'}z_j+ \Big(u-|S'\setminus S_p|\Big)\delta_p+\sum_{i=p+1}^m\Big(|S'\setminus S_{i-1}|-|S'\setminus S_i|\Big)\delta_i=u.
	\end{displaymath}
	Therefore, inequality (\ref{nestedmix}) is facet-defining.
\end{proof}

%------------------------------------------------------------------------------

\subsubsection*{Proof of Theorem \ref{Thm:L3}}

%\begin{thm}\label{Thm:L3}
%	*** I did not read this proof carefully -- Oktay ****
%	Let $\cS$ be nested and let $p\in I$ and $S'\subset J$ be such that $|S'\cup S_{p-1}|\leq n-l<|S'\cup S_p|$. 
%	If $p<m$ or $|S_m|<n-l$ then conditions L\ref{L1}-L\ref{L4} together with\begin{enumerate}
%		\item[L3.] $S_{p+1}\setminus S_p\nsubseteq S'$ if $p\leq m-1$
%	\end{enumerate}
%	are sufficient for inequality (\ref{nestedmix2}) to define a facet of $\conv(\XJDd)$.
%\end{thm}
\begin{proof}
	Assume  that $S'\subseteq J$ and $p\in I$ satisfy the conditions above. Then the assumption $p<m$ or $|S_m|<n-l$ implies that $|S_p|<n-l$, and $S'\neq\emptyset$ as $|S'\cup S_p|> n-l$. Assume $S'=\{s_1,\ldots,s_{|S'|}  \}$ with $s_1<\ldots<s_{|S'|}$. We first show that the inequality\begin{equation}\label{lift1}
	-\sum_{j\in S'}z_j+(|S'\cup S_p|-n+l)\delta_p\leq 0
	\end{equation}
	defines an $(n+p-1)$-dimensional face of $\conv(\XJDd)\cap\{(z,\delta):\delta_i=0,i\in\{p+1,\ldots,m\} \}$.
	
	Let $Q=S_p\cup\{s_1,s_2,\ldots,s_{n-l-|S_p|} \}$. Then $Q$ satisfies $|Q|=n-l$ and $S_p\subset Q\subset S_p\cup S'$. Let $R=(J\setminus S'\setminus S_p)\cup\{1,2,\ldots,l+1-|J\setminus S'\setminus S_p|\}$. Then $R$ satisfies $|R|=l+1\leq u$ and $(J\setminus S'\setminus S_p)\cup\{1,2\}\subseteq R\subseteq J\setminus S'$ as $n-l<|S'\cup S_p|$. Note that $|S'\cup S_{p-1}|\leq n-l$. For $i\in\{1,\ldots,p-1\}$, we can let $T_i$ denote the first $l$ elements of $J\setminus S'\setminus S_i$.
	Consider the points (using Definition \ref{def:pts}) associated with the sets\begin{gather}
	 	R\setminus \{j\} ~\text{for}~  {j\in R},~~R,~~(R\setminus\{1\})\cup \{j\} ~\text{for}~{j\in S_p\setminus R},\label{sets1}\\
	(J\setminus Q\setminus\{s_{|S'|}\})\cup\{j\}  ~\text{for}~{j\in Q\setminus S_p},~~(J\setminus Q\setminus\{j\})\cup\{s_1\}  ~\text{for}~ {j\in S'\setminus Q\setminus\{s_{|S'|} \}},~~J\setminus Q,\label{sets2}\\
	 T_i ~\text{for}~ {i\in\{1,\ldots,p-1\}}.\label{sets3}
	\end{gather}
%	Consider the points (using Definition \ref{def:pts}) associated with the sets\begin{gather}
%\{R\setminus \{j\} \}_{j\in R},R,\{(R\setminus\{1\})\cup \{j\} \}_{j\in S_p\setminus R},\label{sets1}\\
%\{(J\setminus Q\setminus\{s_{|S'|}\})\cup\{j\} \}_{j\in Q\setminus S_p},\{(J\setminus Q\setminus\{j\})\cup\{s_1\} \}_{j\in S'\setminus Q\setminus\{s_{|S'|} \}},J\setminus Q,\label{sets2}\\
%\{T_i\}_{i\in\{1,\ldots,p-1\}}.\label{sets3}
%\end{gather}
	These $(n+p)$ points are feasible with $\delta_i=0,i\in\{p+1,\ldots,m \}$ (as $(J\setminus S'\setminus S_p)\cap S_{p+1}\neq\emptyset$ by assumption L3), and lie on the hyperplane $	-\sum_{j\in S'}z_j+(|S'\cup S_p|-n+l)\delta_p=0.$
	\begin{figure}[H]
		\centering
		\begin{tikzpicture}[level/.style={},decoration={brace,mirror,amplitude=7}]
		\draw[|-|] (0,0) -- (2,0);
		\draw[|-|] (0,0) -- (4,0);
		\draw[|-|] (0,0) -- (7,0);
		\draw[dashed] (2,0) -- (2,1);
		\draw[dashed] (4,0) -- (4,1);
		\draw [decorate] (2,0.1) --(0,0.1) node[midway,above=3mm]{$S_{p}(\subset Q)$};
		\draw [decorate] (2,-0.1) --(3,-0.1) node[midway,below=3mm]{$S_{p+1}\setminus S_p\setminus S'$};
		\draw [decorate] (4,-0.1) --(6,-0.1) node[midway,below=3mm]{$Q\setminus S_p$};
		\draw [decorate] (7,0.1) -- (4,0.1) node[midway,above=3mm]{$S'$};
		\draw [decorate] (4,1) -- (2,1) node[midway,above=3mm]{$J\setminus S_p\setminus S'(\subset R)$};
		\draw [decorate] (0,-0.1) -- (1,-0.1) node[midway,below=3mm]{$R\cap S_p$};
		\end{tikzpicture}
		\caption{Reordered $z_j$ coordinates in the proof of Theorem \ref{Thm:L3}}
		\label{fig:zjLB}
	\end{figure}
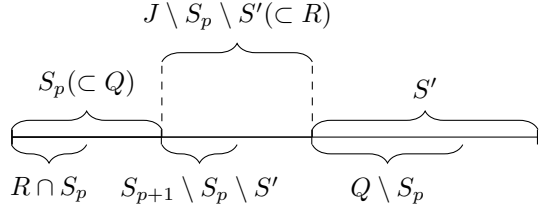
	We reorder the $(\{z_j\}_{j\in J},\delta_1,\ldots,\delta_p)$ coordinates of the points (see Figure \ref{fig:zjLB}) in the ordering\begin{displaymath}
	(\{z_j\}_{j\in S_p},\{z_j\}_{j\in J\setminus(S_p\cup S')},\{z_j\}_{j\in S'},\delta_1,\ldots,\delta_p),
	\end{displaymath}
	and consider the matrix $V$ formed by these reordered coordinates of the (column) points. 
		\begin{figure}[H]
		\centering
		\begin{tikzpicture}[level/.style={},decoration={brace,mirror,amplitude=7}]
		\draw (-1,0) -- (-1,-3);
		\draw (-1,0) -- (2.2,0);
		\draw (2.2,0) -- (2.2,-1);
		\draw (-1,-1) -- (4.2,-1);
		\draw (-1,-2) -- (7,-2);
		\draw (4.2,-1) -- (4.2,-2);
		\draw (-1,-3) -- (7,-3);
		\draw (7,-2) -- (7,-3);
		\draw[dashed] (2.2,0) -- (7,0);
		\draw[dashed] (4.2,-1) -- (7,-1);
		\draw[dashed] (4.2,0) -- (4.2,-1);
		\draw[dashed] (7,0) -- (7,-2);
		\draw [decorate] (7.2,-1) --(7.2,0) node[midway,right=3mm]{$z_j:j\in J\setminus S'$};
		\draw [decorate] (7.2,-2) --(7.2,-1) node[midway,right=3mm]{$z_j:j\in S'$};
		\draw [decorate] (7.2,-3) --(7.2,-2) node[midway,right=3mm]{$\delta_i:i\in\{1,2,\ldots,p\}$};
		\draw [decorate] (2.2,0.2) --(-1,0.2) node[midway,above=3mm]{Sets in \eqref{sets1}};
		\draw [decorate] (4.2,0.2) --(2.2,0.2) node[midway,above=3mm]{Sets in \eqref{sets2}};
		\draw [decorate] (7,0.2) --(4.2,0.2) node[midway,above=3mm]{Sets in \eqref{sets3}};
		\node at (0.6,-.51) {$\bar{\bar{V}}$};
		\node at (1.6,-1.53) {$\bar{V}^{S'}$};
		\node at (3,-2.5) {$V_p$};
		\node at (-2,-1.7) {$V~=~$};
		\end{tikzpicture}
		\caption{Matrix $V$ in the proof of Theorem \ref{Thm:L3}}
	\end{figure}
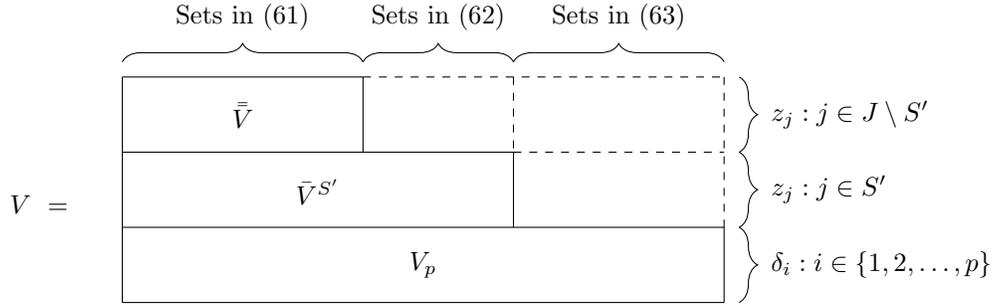
	We will argue that the unique solution to the system of equations\begin{equation}\label{VlamEqn}
	V\lambda=0,~~\sum_{t=1}^{n+p}\lambda_t=0
	\end{equation}
	is $\lambda=0$. First consider matrix $V_{p}$ formed by the last $p$ rows of $V$ corresponding to the $\{\delta_i\}_{i=1}^p$ coordinates: \begin{displaymath}
	V_{p}=\left[\begin{array}{c|c|c|c|c}
	\mathbb{O}_{p\times(|J\setminus S'|+1)} & \mathbbm{1}_{p\times |S'|} & d^1 & \ldots & d^{p-1}
	\end{array}
	\right],
	\end{displaymath}
	where $d^i$ is defined in the proof of Lemma \ref{dimlem}. Equations $V_p\lambda=0$ imply $\lambda_{n+2}=\ldots=\lambda_{n+p}=0$ and $\lambda_{|J\setminus S'|+2}+\ldots+\lambda_{n+1}=0$.
%	\begin{figure}[H]
%	\centering
%	\begin{tikzpicture}[level/.style={},decoration={brace,mirror,amplitude=7}]
%	\draw (0,0) -- (0,-6);
%	\draw (0,0) -- (2.2,0);
%	\draw (2.2,0) -- (2.2,-2);
%	\draw (0,-2) -- (4.2,-2);
%	\draw (0,-4) -- (6,-4);
%	\draw (4.2,-2) -- (4.2,-4);
%	\draw (0,-6) -- (6,-6);
%	\draw (6,-4) -- (6,-6);
%	\draw[dashed] (2.2,0) -- (6,0);
%	\draw[dashed] (4.2,-2) -- (6,-2);
%	\draw[dashed] (4.2,0) -- (4.2,-2);
%	\draw[dashed] (6,0) -- (6,-4);
%	\draw [decorate] (6.2,-2) --(6.2,0) node[midway,right=3mm]{$z_j:j\in J\setminus S'$};
%	\draw [decorate] (6.2,-4) --(6.2,-2) node[midway,right=3mm]{$z_j:j\in S'$};
%	\draw [decorate] (6.2,-6) --(6.2,-4) node[midway,right=3mm]{$\delta_i:i\in\{1,2,\ldots,p\}$};
%	\draw [decorate] (2.2,0.2) --(0,0.2) node[midway,above=3mm]{Sets in \eqref{sets1}};
%	\draw [decorate] (4.2,0.2) --(2.2,0.2) node[midway,above=3mm]{Sets in \eqref{sets2}};
%	\draw [decorate] (6,0.2) --(4.2,0.2) node[midway,above=3mm]{Sets in \eqref{sets3}};
%	\node at (1.1,-1) {$\bar{\bar{V}}$};
%	\node at (2.1,-3) {$\bar{V}^{S'}$};
%	\node at (3,-5) {$V_p$};
%	\end{tikzpicture}
%	\caption{Matrix $V$ in Theorem \ref{Thm:L3}}
%\end{figure}
	Therefore, (\ref{VlamEqn}) reduces to equations\begin{displaymath}
	\bar{V}\bar{\lambda}=0,~~\sum_{t=1}^{n+1}\lambda_t=0
	\end{displaymath}
	where $\bar{V}$ is the matrix formed by the first $n+1$ columns of $V$ and $\bar{\lambda}=(\lambda_1,\ldots,\lambda_{n+1})^T$.
	Then we write down the matrix $\bar{V}^{S'}$ formed by rows $|J\setminus S'|+1$ to $|J|$ of $\bar{V}$:\begin{displaymath}
	\bar{V}^{S'}=\left[
	\begin{array}{c|c}
	\mathbb{O}_{|S'|\times(|J\setminus S'|+1)} & \begin{array}{c|c|c}
	\mathbbm{I}_{|S'\cap Q|} & 
	\begin{array}{c}
	\mathbbm{1}^T_{|S'\setminus Q|-1}\\ \hline
	\mathbb{O}_{(|S'\cap Q|-1)\times(|S'\setminus Q|-1)}
	\end{array}
	& 0_{|S'\cap Q|}\\ \hline
	\begin{array}{c}
	\mathbbm{1}_{(|S'\setminus Q|-1)\times|S'\cap Q|}\\ \hline
	0^T_{|S'\cap Q|}
	\end{array}
	& 
	\begin{array}{c}
	\mathbbm{1}_{(|S'\setminus Q|-1)\times (|S'\setminus Q|-1)}-\mathbbm{I}_{(|S'\setminus Q|-1)}\\ \hline
	\mathbbm{1}^T_{|S'\setminus Q|-1}
	\end{array}
	& \mathbbm{1}_{|S'\setminus Q|}
	\end{array}
	\end{array}
	\right].
	\end{displaymath}
	\if0
	We can verify that the submatrix\begin{displaymath}
	\left[
	\begin{array}{c|c|c}
	\mathbbm{I}_{|S'\cap Q|} & 
	\begin{array}{c}
	\mathbbm{1}^T_{|S'\setminus Q|-1}\\ \hline
	\mathbb{O}_{(|S'\cap Q|-1)\times(|S'\setminus Q|-1)}
	\end{array}
	& 0_{|S'\cap Q|}\\ \hline
	\begin{array}{c}
	\mathbbm{1}_{(|S'\setminus Q|-1)\times|S'\cap Q|}\\ \hline
	0^T_{|S'\cap Q|}
	\end{array}
	& 
	\begin{array}{c}
	\mathbbm{1}_{(|S'\setminus Q|-1)\times (|S'\setminus Q|-1)}-\mathbbm{I}_{(|S'\setminus Q|-1)}\\ \hline
	\mathbbm{1}^T_{|S'\setminus Q|-1}
	\end{array}
	& \mathbbm{1}_{|S'\setminus Q|}
	\end{array}
	\right]
	\end{displaymath}
	is nonsingular. %Not always nonsingular
	\fi
	By looking at $\{z_j\}_{j\in S'\setminus Q}$ and $\{z_j\}_{j\in S'\cap Q}$ coordinates,
	equations $\bar{V}^{S'}\bar{\lambda}=0$ and $\lambda_{|J\setminus S'|+2}+\ldots+\lambda_{n+1}=0$ imply that $\lambda_{n+2-|S'\setminus Q|}=\ldots=\lambda_{n+1}=0$ and $\lambda_{|J\setminus S'|+2}=\ldots=\lambda_{n+1-|S'\setminus Q|}=0$. Therefore, (\ref{VlamEqn}) further reduces to equations\begin{displaymath}
	\bar{\bar{V}}\bar{\bar{\lambda}}=0,~~\sum_{t=1}^{|J\setminus S'|+1}\lambda_t=0
	\end{displaymath}
	where $\bar{\bar{V}}$ is the submatrix of $V$ formed from its first $|J\setminus S'|$ rows and first $|J\setminus S'|+1$ columns and $\bar{\bar{\lambda}}=(\lambda_1,\ldots,\lambda_{|J\setminus S'|+1})^T$. The matrix $\bar{\bar{V}}$ is of the form:
	\begin{displaymath}
	\left[
	\begin{array}{c|c|c|c}
	\mathbbm{1}_{|R\cap S_p|\times |R\cap S_p|}-\mathbbm{I}_{|R\cap S_p|} & \mathbbm{1}_{|R\cap S_p|\times |J\setminus S'\setminus S_p|} 
	& \mathbbm{1}_{|R\cap S_p|} & \begin{array}{c}
	0^T_{|S_p\setminus R|}\\ \hline
	\mathbbm{1}_{(|R\cap S_p|-1)\times |S_p\setminus R|}
	\end{array}\\ \hline
	\mathbb{O}_{|S_p\setminus R|\times |R\cap S_p|} & \mathbb{O}_{|S_p\setminus R|\times|J\setminus S'\setminus S_p|}  & 0_{|S_p\setminus R|} & \mathbbm{I}_{|S_p\setminus R|}\\ \hline
	\mathbbm{1}_{|J\setminus S'\setminus S_p|\times |R\cap S_p|} & \mathbbm{1}_{|J\setminus S'\setminus S_p|\times|J\setminus S'\setminus S_p|}-\mathbbm{I}_{|J\setminus S'\setminus S_p|}  & \mathbbm{1}_{|J\setminus S'\setminus S_p|} & \mathbbm{1}_{|J\setminus S'\setminus S_p|\times|S_p\setminus R|}
	\end{array}
	\right].
	\end{displaymath}
	Rows of $\bar{\bar{V}}$ with index $j\in J\setminus S'\setminus S_p$ together with $\sum_{t=1}^{|J\setminus S'|+1}\lambda_t=0$ imply $\lambda_{|R\cap S_p|+1}=\ldots=\lambda_{|R|}=0$.
	Rows of $\bar{\bar{V}}$ with index $j\in S_p\setminus R$ imply $\lambda_{|R|+2}=\ldots=\lambda_{|J\setminus S'|+1}=0$.	The rest of rows together with $\sum_{t=1}^{|J\setminus S'|+1}\lambda_t=0$ imply $\lambda_1=\ldots=\lambda_{|R\cap S_p|}=0$ and $\lambda_{|R|+1}=0$.
	%		\begin{displaymath}
	%		\bar{\bar{V}}=\left[
	%		\begin{array}{c|c}
	%		\mathbbm{1}_{|J\setminus S'|\times |J\setminus S'|}-\mathbbm{I}_{|J\setminus S'|} & \mathbbm{1}_{|J\setminus S'|}\\ \hline
	%		\mathbb{O}_{(|S'|+p)\times|J\setminus S'|} & 0_{|S'|+p}
	%		\end{array}
	%		\right].
	%		\end{displaymath}
	Therefore, the given points are affinely independent and inequality (\ref{lift1}) defines an $(n+p-1)$-dimensional face of $\conv(\XJDd)\cap\{(z,\delta):\delta_i=0,i\in\{p+1,\ldots,m\} \}$.
	
	We finish the proof by lifting the coefficients of $\delta_{p+1},\ldots,\delta_m$. Define $S_{m+1}=J$ and $\delta_{m+1}=0$. By the validity of (\ref{nestedmix2}), for each $m'\in\{p,p+1,\ldots,m-1 \}$\begin{align*}
	|S'\cup S_{m'+1}|-|S'\cup S_{m'}|\leq -\max\Big\{-\sum_{j\in S'}z_j+(|S'\cup S_p|-n+&l)\delta_p+\sum_{i=p+1}^{m'}(|S'\cup S_i|-|S'\cup S_{i-1}|)\delta_i:\\
	&(z,\delta)\in\XJDd,\delta_{m'+1}=1,\delta_{i}=0,i>m'+1 \Big\}.
	\end{align*}
	Actually the above inequality holds at equality by taking $(z,\delta)$ as the points of $(J\setminus S'\setminus S_{m'+1})\cup L_{m'}$ for $m'=p,\ldots,m-1$, respectively. Here $L_{m'}\subset S'\setminus S_{m'+1}$ can be constructed by starting with an element in $S_{m'+2}$ if $(J\setminus S'\setminus S_{m'+1})\cap S_{m'+2}=\emptyset$ and then augmenting it to have cardinality $|S'\cup S_{m'+1}|-(n-l)~(\geq 1)$. Set $L_{m'}\subset S'\setminus S_{m'+1}$ (see Figure \ref{fig:LmLB}) satisfies\begin{enumerate}
		\item $|L_{m'}|=|S'\cup S_{m'+1}|-(n-l)=l-|J\setminus S'\setminus S_{m'+1}|$,
		\item $((J\setminus S'\setminus S_{m'+1})\cup L_{m'})\cap S_{m'+2}\neq\emptyset$ if $m'<m-1$.
	\end{enumerate}
	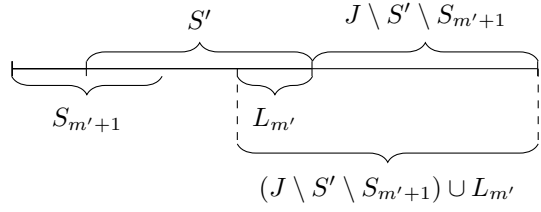
\begin{figure}[H]
		\centering
		\begin{tikzpicture}[level/.style={},decoration={brace,mirror,amplitude=7}]
		\draw[|-|] (0,0) -- (1,0);
		\draw[|-|] (0,0) -- (4,0);
		\draw[|-|] (0,0) -- (7,0);
		\draw[dashed] (3,0) -- (3,-1);
		\draw[dashed] (7,0) -- (7,-1);
		\draw [decorate] (0,-0.1) --(2,-0.1) node[midway,below=3mm]{$S_{m'+1}$};
		\draw [decorate] (3,-1) --(7,-1) node[midway,below=3mm]{$(J\setminus S'\setminus S_{m'+1})\cup L_{m'}$};
		\draw [decorate] (3,-0.1) --(4,-0.1) node[midway,below=3mm]{$L_{m'}$};
		\draw [decorate] (4,0.1) -- (1,0.1) node[midway,above=3mm]{$S'$};
		\draw [decorate] (7,0.1) -- (4,0.1) node[midway,above=3mm]{$J\setminus S'\setminus S_{m'+1}$};
		\end{tikzpicture}
		\caption{Construction of $L_{m'}$ in the proof of Theorem \ref{Thm:L3}}
		\label{fig:LmLB}
	\end{figure}
	In total, we find $(n+m)$ affinely independent points lying on the hyperplane\begin{displaymath}
	-\sum_{j\in S'}z_j+(|S'\cup S_p|-n+l)\delta_p+\sum_{i=p+1}^m(|S'\cup S_i|-|S'\cup S_{i-1}|)\delta_i=0.
	\end{displaymath}
	Therefore, inequality (\ref{nestedmix2}) is facet-defining.
\end{proof}

\end{document}